\documentclass[10pt]{article}
\usepackage{amsmath,amsfonts,amssymb,amsthm,booktabs,color,epsfig,graphicx,hyperref,url}
\usepackage{bbm,bm,caption,subcaption,xcolor}
\usepackage[top=20mm, bottom=20mm, left=15mm, right=15mm]{geometry}
\usepackage{natbib}
\allowdisplaybreaks
\newtheorem{prop}{Proposition}
\newtheorem{thm}{Theorem}
\newtheorem{lemma}{Lemma}

\def\EDIT#1{{\textcolor{black}{#1}}}
\def\ES#1{{\textcolor{black}{#1}}}

\begin{document}

\begin{center}
{\bf\Large A geometric investigation into the tail dependence of vine copulas}\\
{\large Emma S.\ Simpson, Jennifer L.\ Wadsworth and Jonathan A.\ Tawn\\}
Lancaster University
\end{center}

\begin{abstract}
\EDIT{Vine copulas are a type of multivariate dependence model, composed of a collection of bivariate copulas that are combined according to a specific underlying graphical structure.} Their flexibility and practicality in moderate and high dimensions have contributed to the popularity of vine copulas, but relatively little attention has been paid to their extremal properties. To address this issue, we present results on the tail dependence properties of some of the most widely studied vine copula classes. We focus our study on the coefficient of tail dependence and the asymptotic shape of the sample cloud, which we calculate using the geometric approach of \cite{Nolde2014}. We offer new insights by presenting results for trivariate vine copulas constructed from asymptotically dependent and asymptotically independent bivariate copulas, focusing on bivariate extreme value and inverted extreme value copulas, with additional detail provided for logistic and inverted logistic examples. We also present new theory for a class of higher dimensional vine copulas, constructed from bivariate inverted extreme value copulas.
\end{abstract}

\noindent{\bf Keywords:} coefficient of tail dependence; gauge function; multivariate extremes; vine copula.

\section{Introduction}\label{sec:intro}
In multivariate extreme value analysis, the tail dependence properties of variables are an important consideration for model selection. In particular, one may be interested in whether or not they exhibit so-called asymptotic dependence, where the largest values can occur simultaneously across all variables; see \cite{Coles1999}. Suppose we are interested in a model for the $d$ random variables $\bm{X}=(X_1,\ldots,X_d)$, which we assume have standard exponential margins to focus only on dependence, i.e., $\Pr(X_i<x)=1-e^{-x}$, $x\geq 0$, for $i=1,\ldots,d$. \EDIT{For any subset of these variables, $\bm{X}_\mathcal{C} = (X_i:i\in \mathcal{C})$, with $\mathcal{C}\subseteq \mathcal{D}=\{1,\ldots,d\}$ and $|\mathcal{C}|\geq 2$, and any $j\in\mathcal{C}$, one can consider the measure
\begin{align}
	\chi_\mathcal{C}=\lim_{u\rightarrow\infty}\frac{\Pr\left(X_i>u;i\in \mathcal{C}\right)}{\Pr(X_j>u)}=\lim_{u\rightarrow\infty}e^u\Pr\left(X_i>u;i\in \mathcal{C}\right),
\label{eqn:chi}
\end{align}}
which corresponds to the limiting probability that all variables are above some high threshold $u$, given that any one of the variables exceeds $u$. If $\chi_\mathcal{C}>0$, all variables in $\bm{X}_\mathcal{C}$ exhibit asymptotic dependence, while $\chi_\mathcal{C}=0$ means that not all variables in $\bm{X}_\mathcal{C}$ can be simultaneously large. In the latter case, if $|\mathcal{C}|=2$, the two variables cannot be simultaneously extreme, and are said to exhibit asymptotic independence; if $|\mathcal{C}|>2$, it is still possible to have $\chi_{\tilde{\mathcal{C}}}>0$ for any $\tilde{\mathcal{C}}\subset \mathcal{C}$. That is, variables indexed by $\tilde{\mathcal{C}}$ could take their largest values simultaneously while at least one of those indexed by $\mathcal{C}\setminus\tilde{\mathcal{C}}$ are of smaller order. The collection of all sets of variables which can or cannot be simultaneously extreme corresponds to a more complicated extremal dependence structure; see \cite{Goix2017} or \cite{Simpson2020}.

Moreover, if $\chi_\mathcal{C}=0$, there could be some sub-asymptotic dependence between $\bm{X}_\mathcal{C}$, despite the lack of asymptotic dependence in the limit, and the measure $\chi_\mathcal{C}$ does not tell the full story. To investigate this behaviour further, it is common to consider the coefficient of tail dependence, introduced by \cite{Ledford1996}. Again for a subset of exponential variables $\bm{X}_\mathcal{C}$, with $|\mathcal{C}|\geq 2$, this is defined via the relation
\begin{align}
	\Pr\left(X_i>x:i\in \mathcal{C}\right)\sim L_\mathcal{C}(e^x)e^{-x/\eta_\mathcal{C}},
\label{eqn:eta}
\end{align}
as $x\rightarrow\infty$, where $L_\mathcal{C}$ denotes a function that is slowly varying at infinity, and $\eta_\mathcal{C}\in(0,1]$. If $\eta_\mathcal{C}=1$ and $L_\mathcal{C}(x)\not\rightarrow 0$, as $x\rightarrow\infty$, the variables in $\bm{X}_\mathcal{C}$ are asymptotically dependent. For $\eta_\mathcal{C}<1$, these variables cannot be simultaneously large, and the value of the coefficient quantifies the strength of sub-asymptotic dependence between the variables. The set of measures $\{\chi_\mathcal{C},\eta_\mathcal{C}:\mathcal{C}\subseteq \mathcal{D}, |\mathcal{C}|\geq 2\}$ therefore provide a summary of the key extremal dependence features of $\bm{X}$.

Often, the value of $\eta_\mathcal{C}$ can be calculated directly from~\eqref{eqn:eta} for a given model, but in some cases, only the joint density of $\bm{X}_\mathcal{C}$ can be specified in closed form, and not the required joint survivor function. \cite{Nolde2014} presents a strategy to overcome this issue, based on the geometry of scaled random samples from the joint distribution of $\bm{X}_\mathcal{C}$. A simplified version of the approach of \cite{Nolde2014}, which we discuss further in Section~\ref{subsec:Nolde}, assumes standard exponential margins, and joint density $f_\mathcal{C}(\bm{x}_\mathcal{C})$, with the idea being to study the gauge function $g_\mathcal{C}(\bm{x}_\mathcal{C})$ such that
\begin{align}
	-\ln f_\mathcal{C}(t\bm{x}_\mathcal{C}) \sim tg_\mathcal{C}(\bm{x}_\mathcal{C}),
\label{eqn:gaugeMain}
\end{align}
as $t\rightarrow\infty$, with $g_\mathcal{C}(\bm{x}_\mathcal{C})$ being homogeneous of order 1. The limiting shape of suitably scaled samples from $\bm{X}_\mathcal{C}$ is described by the set of points where the gauge function is at most one, i.e., $\{\bm{x}_\mathcal{C}:g_\mathcal{C}(\bm{x}_\mathcal{C})\leq 1\}$, and studying this set can reveal the value of $\eta_\mathcal{C}$ and provide insight into other aspects of the extremal dependence structure; see \cite{Nolde2020}. We present some example gauge function calculations in Section~\ref{subsec:bivariateEtas} for the case where $|\mathcal{C}|=2$, for both asymptotically dependent and asymptotically independent models, and demonstrate how they can be used to obtain $\eta_\mathcal{C}$. 

One drawback of this method is that it is only applicable when the joint density of $\bm{X}_\mathcal{C}$ can be obtained analytically. It may be the case that we have a closed form joint density for variables $\bm{X}$, but not for all $\bm{X}_\mathcal{C}$ with $\mathcal{C}\subset \mathcal{D}$. \cite{Nolde2020} show how to derive lower-dimensional gauge functions from higher-dimensional ones, and in Section~\ref{sec:gauge}, we review this technique for the calculation of $g_\mathcal{C}(\bm{x}_\mathcal{C})$ and $\eta_\mathcal{C}$ in such cases. In our study of vine copulas, which have dimension $d\geq 3$, this approach can be necessary to obtain even some bivariate results.

\EDIT{In this paper, we focus on investigating the tail behaviour of vine copulas. These models exploit the wide range of existing parametric bivariate copula models to create parametric copula models for higher dimensions, where there are fewer options available.} This allows for the construction of flexible models with the possibility of capturing a wide range of dependence features. The idea of combining bivariate copulas in this way was first proposed by \cite{Joe1996}; developed further by \cite{Bedford2001,Bedford2002}, who proposed the use of a type of graphical model called vines to aid the modelling procedure; and further studied in an inferential context by \cite{Aas2009}. A textbook treatment of these models is provided by \cite{Kurowicka2010}. We give an introduction to vine copula modelling in Section~\ref{sec:vines}. \EDIT{Vine copulas are widely used in financial applications;  a summary of these applications is provided by \cite{Aas2016}, with examples including \cite{Brechmann2012} and \cite{Dissmann2013}.}

%Vine copula models have the potential to be used in extremal dependence modelling, studied for instance by \cite{Joeetal2010}. Their motivation differs from ours as they study both upper and lower tail dependence, focusing on the possibility of exploiting vine structures to achieve asymmetric tail features, and they obtain results linked to the measure $\chi_\mathcal{C}$ rather than $\eta_\mathcal{C}$. They also study the result of imposing asymptotic independence or asymptotic dependence in certain pair copulas, and how this leads to constraints on $\chi_\mathcal{C}$ for different subsets of variables. Although this addresses some of the issues, it is not fully understood how the bivariate copulas and underlying graphical structure used in the construction of a vine copula affect the extremal dependence structure of the variables. Our aim is to further investigate the tail dependence properties of vine copulas, focusing on the calculation of the coefficient $\eta_\mathcal{C}$ in~\eqref{eqn:eta} and the shape of the gauge function.

\EDIT{Vine copulas reduce model formulation to a series of pairwise copula selections, and therefore appear to be ideal for modelling extremal dependence since, as noted earlier, such dependence is known to have complex structure. For vine copulas over variables indexed by $\mathcal{D}$, \cite{Joeetal2010} have made major progress in deriving general results about $\chi_{\mathcal{D}}$, defined in \eqref{eqn:chi}, for any vine copula. In particular, they have determined some relationships of $\chi_{\mathcal{D}}$ with the values of $\chi_{\mathcal{C}}$ (for $|\mathcal{C}|=2$) associated with the bivariate copulas used in the vine construction. Primarily, they consider the case where all pairwise $\chi_\mathcal{C}$ are non-zero, and therefore focus only on asymptotically dependent copulas. They also study the result of imposing asymptotic independence or asymptotic dependence in certain pair copulas, and how this leads to constraints on $\chi_{\mathcal{D}}$.}

\EDIT{Some pairwise copulas have $\chi_\mathcal{C}=0$, with a range of examples given by \cite{Heffernan2000}. Particularly noteworthy cases include the pairwise Gaussian copula and the Morgernstern copula (see Example~2.1 of \cite{Joeetal2010}), with parameters $-1<\rho<1$ and $-1<\theta<1$, respectively. For vine copulas, when some of the pairwise $\chi_\mathcal{C}=0$, the results of \cite{Joeetal2010} only give that $\chi_{\mathcal{D}}=0$, and they fail to give any information about the dependence in the joint tail of variables, e.g., the probability that all of the variables are simultaneously large. In that case we are interested in the numerator of~\eqref{eqn:chi} (prior to it being taken to its limit for $\chi_{\mathcal{D}}$). When $\chi_{\mathcal{D}}=0$, all we know  is that this joint probability is smaller order than the marginal probability of one of the variables being large (as in the denominator of~\eqref{eqn:chi}). Such joint tail probabilities are important in characterising the tail, and for assessing risk in applications \citep{Coles1994,Heffernan2004}.}

\EDIT{The tail parameter $\eta_{\mathcal{C}}$ for all $\mathcal{C}\subseteq \mathcal{D}$ in~\eqref{eqn:eta} is important, as it captures the level of asymptotic independence, with $\eta_{\mathcal{C}}=1$ corresponding to asymptotic dependence (for all cases where $\chi_{\mathcal{C}}>0$) and $0<\eta_{\mathcal{C}}<1$ corresponding to levels of asymptotic independence. For example, for the bivariate Gaussian copula, $\eta_{\{1,2\}}=(1+\rho)/2$ (with $-1<\rho<1$ denoting the usual ``correlation coefficient''), and for the Morgenstern copula, $\eta_{\{1,2\}}=1/2$ for all $\theta$. So, although these two copulas have identical $\chi_{\{1,2\}}$ values, they have different values of $\eta_{\{1,2\}}$ unless $\rho=0$.}

\EDIT{Our paper therefore differs from \cite{Joeetal2010} in that we aim to find $\eta_{\mathcal{C}}$ in cases where their results simply give $\chi_{\mathcal{C}}=0$, for all $\mathcal{C}\subseteq \mathcal{D}$. With this, we seek to better understand how the bivariate copulas and underlying graphical structure used in the construction of a vine copula affect these additional extremal dependence features of the variables, and to be the first to study the gauge function for vine copulas.}

Throughout this paper, we consider exponential marginal distributions, but allow a variety of different bivariate copulas to be used in the vine copula construction. If only bivariate Gaussian copulas are used in this construction, the overall joint distribution of the variables will also be Gaussian \citep{Joe1996,Joeetal2010}. Since the tail dependence features of the Gaussian model are well-studied in the literature, we focus on cases where the pair copulas are from extreme value or inverted extreme value classes of distributions \citep{Ledford1997,Papastathopoulos2016}. These classes are widely studied in the extreme value literature; while they are not in themselves parametric distributions, they do include a range of well-known parametric examples \citep{ColesAndTawn1991, Cooley2010, Ballani2011}. Bivariate extreme value distributions exhibit asymptotic dependence, while their inverted counterparts exhibit asymptotic independence. Studying these two classes is therefore sufficient to reveal a rich variety of structures within the vine copula framework.

Vine copula models provide an example of when the joint distribution function of the variables generally cannot be calculated analytically. Moreover, the joint densities corresponding to certain subsets of the variables often do not have closed forms. To study the tail behaviour of these models, we calculate $\eta_\mathcal{C}$ for several examples, through the application of the geometric approach of \cite{Nolde2014}. Our investigation reveals interesting features of the shape of the gauge function in~\eqref{eqn:gaugeMain} for vine copula models.

Having introduced the geometric methodology for studying extremal dependence in Section~\ref{sec:gauge}, and provided an overview of vine copula modelling in Section~\ref{sec:vines}, the remainder of the paper is structured as follows. In Section~\ref{sec:vinesIEV}, we present calculations and results for cases where each pair copula is from the inverted extreme value family of distributions. In higher than three dimensions, the underlying graphical structure of the copula is a further consideration, and we also present results for inverted extreme value pair copulas here, with two different types of underlying vine structure. In the trivariate and higher-dimensional examples, we present results for inverted logistic examples as a special case. In Section~\ref{sec:vinesIEVEV}, we return to the trivariate case, presenting results for vine copulas constructed from combinations of extreme value and inverted extreme value pair copulas.

\section{Geometric approaches for calculating $\eta_\mathcal{C}$}\label{sec:gauge}
\subsection{The geometric approach of \cite{Nolde2014}}\label{subsec:Nolde}
\cite{Nolde2014} proposes a method to calculate the coefficient of tail dependence $\eta_\mathcal{C}$ based on the shape of scaled random samples from the vector $\bm{X}_\mathcal{C}$. This follows earlier work by \cite{Balkema2010}, who showed that the limiting shape of the sample cloud could be used to determine the presence of asymptotic independence. Theorem 2.1 of \cite{Nolde2014} provides the result for marginal distributions with Weibull-type tails. We take a simplified approach by focusing on the special case where all margins have standard exponential distributions, which is possible without losing information about the extremal dependence properties of the variables. 

Interest lies with the gauge function $g_\mathcal{C}(\bm{x}_\mathcal{C})$, satisfying equation~\eqref{eqn:gaugeMain}. In this case, we consider the scaled random sample $\left(\bm{X}_{\mathcal{C},1}/\ln n,\ldots,\bm{X}_{\mathcal{C},n}/\ln n\right)$, as $n\rightarrow\infty$, with the scaling function $\ln n$ chosen due to the exponential margins. The sample cloud converges onto the compact set $G^*_\mathcal{C}=\{\bm{x}_\mathcal{C}\in\mathbb{R}^{|\mathcal{C}|}:g_\mathcal{C}(\bm{x}_\mathcal{C})\leq 1\}\subseteq[0,1]^{|\mathcal{C}|}$, which also has the property of being star-shaped, i.e., if $\bm{x}_\mathcal{C}\in G^*_\mathcal{C}$, then $t\bm{x}_\mathcal{C}\in G^*_\mathcal{C}$ for all $t\in(0,1)$. We denote the part of the boundary of this set where the gauge function equals one by $G_\mathcal{C}=\{\bm{x}_\mathcal{C}\in\mathbb{R}^{|\mathcal{C}|}:g_\mathcal{C}(\bm{x}_\mathcal{C})= 1\}\subset G^*_{\mathcal{C}}\subseteq[0,1]^{|\mathcal{C}|}$.

\cite{Nolde2014} shows that the coefficient of tail dependence can be calculated as
\begin{align}
\eta_\mathcal{C}=\min\left\{r:G_\mathcal{C}\cap [r,\infty)^{|\mathcal{C}|} \neq\emptyset\right\},
\label{eqn:NoldeEta}
\end{align}
which \cite{Nolde2020} show is equivalent to
\begin{align}
\eta_\mathcal{C}=\left\{\min\limits_{\bm{x}_\mathcal{C}:\min(\bm{x}_\mathcal{C})=1}g_\mathcal{C}\left(\bm{x}_\mathcal{C}\right)\right\}^{-1}=\left\{\min\limits_{\bm{x}_\mathcal{C}:\min(\bm{x}_\mathcal{C})\geq 1}g_\mathcal{C}\left(\bm{x}_\mathcal{C}\right)\right\}^{-1}.
\label{eqn:NWetaGauge}
\end{align}
The case where $\arg\min_{\bm{x}_\mathcal{C}:\min(\bm{x}_\mathcal{C})=1}g_\mathcal{C}\left(\bm{x}_\mathcal{C}\right)=\bm{1}\in\mathbb{R}^{|\mathcal{C}|}$ occurs if and only if $\eta_\mathcal{C}=1/g_\mathcal{C}(\bm{1})$, corresponding to the intersection in~\eqref{eqn:NoldeEta} occurring when all variables are equal, i.e., when $x_j=x_k$ for all $j,k\in \mathcal{C}$. We also note that the quantity $1/g_\mathcal{C}(\bm{1})$ will always provide a lower bound for $\eta_\mathcal{C}$, and if $g_\mathcal{C}(\bm{1})=1$, then it must be the case that $\eta_\mathcal{C}=1$. When analytical minimisation of $g_\mathcal{C}$ is difficult or impossible, numerical investigation can be used to determine where the minimum occurs. 

This numerical minimisation may be undertaken using optimisation software such as \texttt{optim} in \textsf{R}. Implementation is simpler using the second form in~\eqref{eqn:NWetaGauge}, i.e., with $\min(\bm{x}_\mathcal{C})\geq 1$, and one can check that the numerical optimisation gives $\min(\bm{x}_\mathcal{C})=1$. It is advisable to compare results across a range of starting values of $\bm{x}_\mathcal{C}$ to ensure convergence. If convergence is not reached, an alternative is to carry out the investigation across the range of subspaces of interest, i.e., where different subsets of $\bm{x}_\mathcal{C}$ are equal to one, and to compare these results. \EDIT{While it is clearly preferable to use theoretical results, where this is not possible, this numerical approach can be useful. Moreover, where theoretical results are difficult to obtain, numerical studies may provide additional insight that can facilitate analytical calculations.}

We subsequently drop the subscript $\mathcal{C}$ from the set $G_\mathcal{C}$, density $f_\mathcal{C}$ and gauge function $g_\mathcal{C}$ when discussing the overall vector of variables $\bm{X}$, i.e., when $\mathcal{C}=\mathcal{D}$.

\subsection{Lower dimensional subsets}\label{subsec:lowerDimEta}
The method of \cite{Nolde2014} can only be used to calculate the coefficient $\eta_\mathcal{C}$ in cases where the density $f_\mathcal{C}(\bm{x}_\mathcal{C})$ can be obtained analytically. In some cases, including many vine copula examples, we may have the form of $f(\bm{x})$, but no closed form of $f_\mathcal{C}(\bm{x}_\mathcal{C})$ for certain subsets $\mathcal{C}\subset \mathcal{D}$, so the method cannot be directly applied to obtain $\eta_\mathcal{C}$. \cite{Nolde2020} use results on projections of sample clouds to show that the gauge function $g_\mathcal{C}(\bm{x}_\mathcal{C})$ can be obtained from the gauge function $g(\bm{x})$ for any set $\mathcal{C}\subset\mathcal{D}$ with $|\mathcal{C}|\in[1,d-1]$, by 
\begin{align}
g_\mathcal{C}\left(\bm{x}_\mathcal{C}\right)=\min_{x_i:i\notin \mathcal{C}}g(\bm{x}).
\label{eqn:lowerdimG}
\end{align}
Once this gauge function has been obtained, the remainder of the procedure to calculate $\eta_\mathcal{C}$ continues as in Section~\ref{subsec:Nolde}. \EDIT{In particular, this implies that
\begin{align}
 \eta_\mathcal{C} = \left\{\min\limits_{\bm{x}_\mathcal{C}:\min(\bm{x}_\mathcal{C})\geq 1}g_\mathcal{C}(\bm{x}_\mathcal{C})\right\}^{-1} = \left\{\min\limits_{\substack{\bm{x}:\min\{x_i:i\in\mathcal{C}\geq 1\}\\ ~~\min\{x_i:i\notin\mathcal{C}\geq 0\}}}g(\bm{x})\right\}^{-1}.
\label{eqn:etaCsubset}
\end{align}
If the gauge function still cannot be obtained analytically via the minimisation in~\eqref{eqn:lowerdimG}, numerical methods can again be exploited. Numerical calculation of $\eta_\mathcal{C}$ will require optimising equation~\eqref{eqn:lowerdimG} within equation~\eqref{eqn:NWetaGauge}, as in~\eqref{eqn:etaCsubset}, and the result of this optimisation procedure may again be used to motivate the theoretical calculation of $\eta_\mathcal{C}$.} In Sections~\ref{sec:vinesIEV}~and~\ref{sec:vinesIEVEV} we will study cases where we have the form of $g_{\{1,2,3\}}(x_1,x_2,x_3)$ and wish to deduce $\eta_{\{1,3\}}$, where the analytical form of $g_{\{1,3\}}(x_1,x_3)$ is not known. \EDIT{If theoretical arguments or numerical investigations suggest that the minimum in~\eqref{eqn:NWetaGauge} occurs when $x_1=x_3=1$, i.e., $\eta_{\{1,3\}}=1/g_{\{1,3\}}(1,1)$, one can focus solely on this case in~\eqref{eqn:etaCsubset}.} That is, only calculation of $g_{\{1,3\}}(1,1)=\min_{v} g(1,v,1)$ is needed, which may be possible even when the minimisation in~\eqref{eqn:lowerdimG} over the full range of $(x_1,x_3)$ values is not. \EDIT{In addition, if we can find any $\bm{x}$ with $x_i\geq 1,i\in\mathcal{C}$ and $x_i\geq 0,i\notin\mathcal{C}$, such that $g(\bm{x})=1$, then this must correspond to the required minimum, and $\eta_\mathcal{C}=1$.}

\subsection{Bivariate examples}\label{subsec:bivariateEtas}
To demonstrate the geometric approach discussed in Section~\ref{subsec:Nolde}, we consider six bivariate examples, corresponding to three distributions belonging to each of the asymptotic independence and asymptotic dependence classes. We also comment on interesting features relating to the shape of the set $G$ in each case. In this section, we generally use plots to determine where the intersection in~\eqref{eqn:NoldeEta} occurs, but we note that equation~\eqref{eqn:NWetaGauge} also holds in all cases.

We begin with the example of independent exponential variables, where it is straightforward to obtain the gauge function as $g(x_1,x_2)=x_1+x_2$, i.e., the set $G$ corresponds to the straight line $x_2=1-x_1$ for $x_1,x_2\in[0,1]$. This is demonstrated in case~(i) of Fig.~\ref{fig:bivariateGauge}, and it is clear that the smallest value of $r$ such that $G$ and $[r,\infty)^2$ do not intersect yields $\eta_{\{1,2\}}=1/2$. In Fig.~\ref{fig:bivariateGauge} we plot scaled samples of size 1000 from the various models; note that since the gauge function calculations are based on asymptotic results, some of this finite sample may lie outside the set $G$.

As a second asymptotically independent model, we consider a bivariate Gaussian copula model having exponential margins and covariance matrix $\Sigma$ with $\Sigma_{1,1}=\Sigma_{2,2}=1$ and $\Sigma_{1,2}=\Sigma_{2,1}=\rho\in[0,1)$. \cite{Nolde2014} shows that this model has gauge function
\[
	g(x_1,x_2) = (1-\rho^2)^{-1}\left(x_1 + x_2 -2\rho x_1^{1/2}x_2^{1/2}\right),
\]
with $x_1,x_2\geq 0$. This is demonstrated in case~(ii) of Fig.~\ref{fig:bivariateGauge} for $\rho=0.5$. Minimisation as in~\eqref{eqn:NWetaGauge} reveals that $\eta_{\{1,2\}}=1/g(1,1)=(1+\rho)/2=0.75$, corresponding to the known coefficient for a Gaussian model \citep{Ledford1996}.

Examples (iii), (v) and (vi) are based on the class of bivariate extreme value distributions, which in exponential margins, have a joint distribution function written as
\begin{align}
	F(x_1,x_2) = \exp\left[-V\left\{\frac{-1}{\ln\left(1-e^{-x_1}\right)},\frac{-1}{\ln\left(1-e^{-x_2}\right)}\right\}\right],
\label{eqn:BEVcop}
\end{align}
for $x_1,x_2\geq 0$ and some exponent measure $V(x,y)$ that is homogeneous of order $-1$ and takes the form
\begin{align}
	V(x,y) = 2\int_0^1\max\left(\frac{w}{x},\frac{1-w}{y}\right)dH(w),
\label{eqn:exponentMeasure}
\end{align}
for $x,y>0$, and spectral distribution $H$ satisfying the moment constraint $\int_0^1wdH(w)=1/2$. We let $V_1$, $V_2$ and $V_{12}$ denote the derivatives of the exponent measure with respect to the first, second and both components, respectively, where these are assumed to exist. One example of a method belonging to this class is the logistic distribution, with exponent measure 
\begin{align}
V(x,y)=\left(x^{-1/\alpha}+y^{-1/\alpha}\right)^\alpha, ~~~~\text{for some }\alpha\in(0,1] \text{ and } x,y>0,
\label{eqn:Vlogistic}
\end{align}
see \cite{Tawn1988}. We return to the extreme value distribution in case~(v), but first use these results to consider a final asymptotically independent model: the inverted bivariate extreme distribution. Models of this class are obtained by exchanging the upper and lower tail features in extreme value distribution~\eqref{eqn:BEVcop}. In exponential margins, the model has distribution function
\[
	F(x_1,x_2) = 1 - e^{-x_1} - e^{-x_2} + \exp\left\{-V\left(x_1^{-1},x_2^{-1}\right)\right\},
\]
for $x_1,x_2\geq 0$. Differentiating with respect to both components to obtain the density $f(x_1,x_2)$, we have
\begin{align}
	-\ln f(tx_1,tx_2) =& 2\ln(tx_1) + 2\ln(tx_2) + V\left\{(tx_1)^{-1} , (tx_2)^{-1}\right\} \nonumber\\
	&\hspace{1cm}- \ln\left[V_1\left\{(tx_1)^{-1} , (tx_2)^{-1}\right\}V_2\left\{(tx_1)^{-1}, (tx_2)^{-1}\right\}-V_{12}\left\{(tx_1)^{-1} , (tx_2)^{-1}\right\}\right]\nonumber\\
	=&2\ln(tx_1) + 2\ln(tx_2) + tV\left(x_1^{-1} , x_2^{-1}\right) - \ln\left\{t^4 V_1\left(x_1^{-1} , x_2^{-1}\right)V_2\left(x_1^{-1} , x_2^{-1}\right)-t^3V_{12}\left(x_1^{-1} , x_2^{-1}\right)\right\}\nonumber\\
	=&tV\left(x_1^{-1} , x_2^{-1}\right) + O(\ln t),
\label{eqn:bivariateIEV}
\end{align}
as $t\rightarrow\infty$, by exploiting the homogeneity of the exponent measure. That is, the gauge function is given by $g(x_1,x_2) = V\left(x_1^{-1} , x_2^{-1}\right)$, for $x_1,x_2\geq 0$. For the bivariate inverted logistic example, this corresponds to the gauge function  $g(x_1,x_2)=(x_1^{1/\alpha}+x_2^{1/\alpha})^\alpha$. This is demonstrated in case (iii) of Fig.~\ref{fig:bivariateGauge} for $\alpha=0.5$. The smallest value of $r$ such that $G$ and $[r,\infty)^2$ do not intersect is $2^{-0.5}=1/2^{\alpha}$, occurring when $x_1=x_2$. This corresponds to the known value of $\eta_{\{1,2\}}$ for this copula \citep{Ledford1996}.

\begin{figure}[t]
\begin{center}
\includegraphics[width=0.9\textwidth]{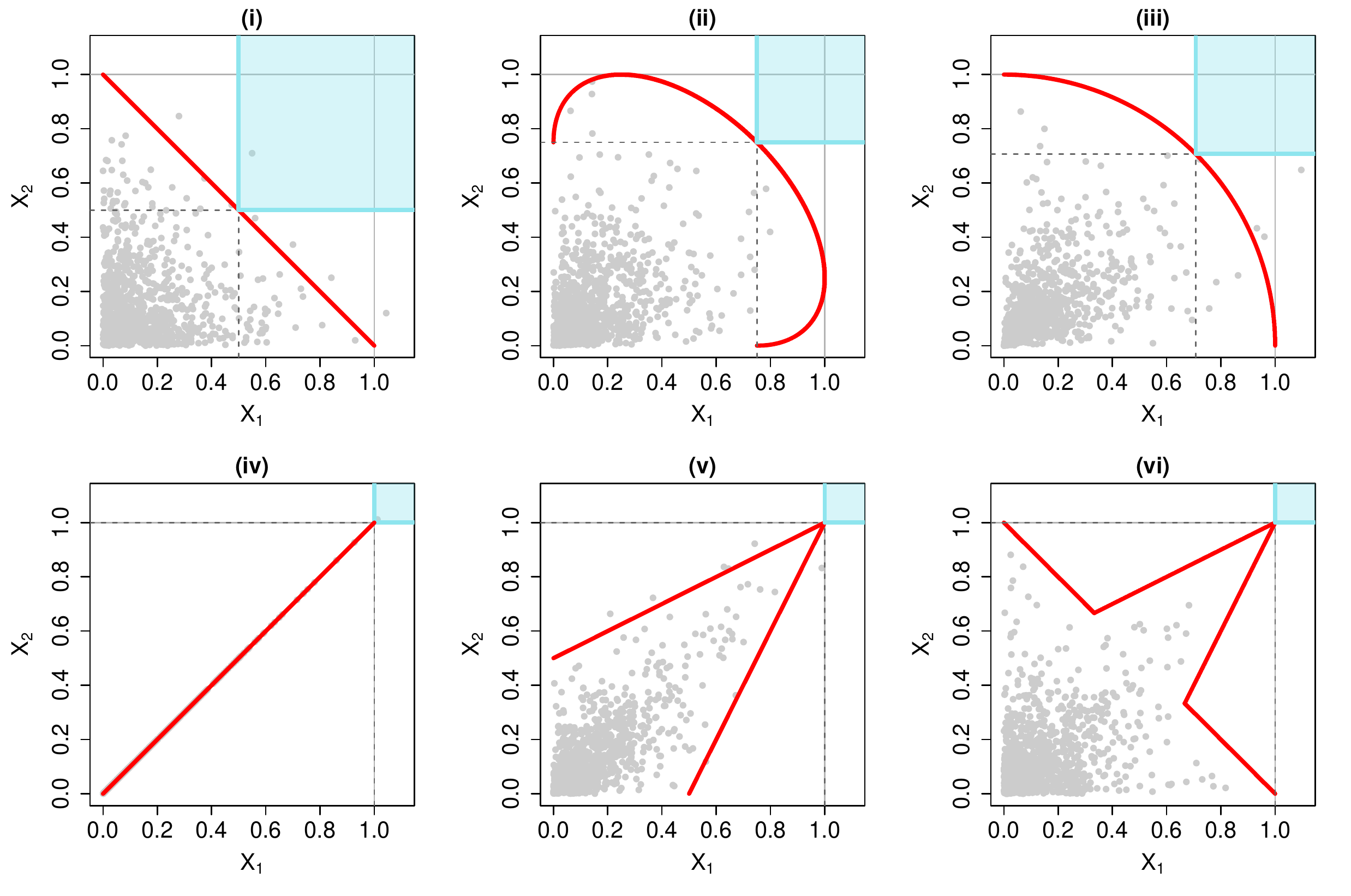}
\caption{Example of the geometric interpretation of $\eta_{\{1,2\}}$ for (i) independence, (ii) Gaussian, (iii) inverted logistic, (iv) perfect dependence, (v) logistic and (vi) asymmetric logistic models. For each model, we show 1000 scaled samples on exponential margins (grey); the set $G=\{(x_1,x_2)\in\mathbb{R}^2:g(x_1,x_2)=1\}$ (red); and the set $[\eta_{\{1,2\}},\infty)^2$ (blue).}
\label{fig:bivariateGauge}
\end{center}
\end{figure}

We now turn our attention to asymptotically dependent models, the most simple example of which corresponds to perfect dependence, demonstrated by case~(iv) of Fig.~\ref{fig:bivariateGauge}. In this case, the density does not exist, but since the set $G$ describes the boundary of the scaled sample cloud, it is clear that this corresponds to the line $x_1=x_2\in[0,1]$, and that considering the intersection of $[r,\infty)^2$ and $G$ in the usual way gives $\eta_{\{1,2\}}=1$.

\EDIT{Returning to the bivariate extreme value copula, with exponent measure~\eqref{eqn:exponentMeasure}, general results cannot easily be derived, as illustrated in Section~\ref{subsec:-logc_EV} of the Supplementary Material.} However, progress is possible if we assume that the corresponding spectral density $h(w)$ places no mass on $\{0\}$ or $\{1\}$ and has regularly varying tails, as in \cite{Heffernan2004} and \cite{Simpson2020}. Specifically, let 
\begin{align}
h(w)\sim c_1(1-w)^{s_1} \text{ as } w\nearrow 1; ~~~ h(w)\sim c_2w^{s_2} \text{ as } w\searrow 0,
\label{eqn:EVassumption}
\end{align} 
for $c_1,c_2\in \mathbb{R}$ and $s_1,s_2>-1$. In this case the gauge function, as shown in Section~\ref{subsec:-logc_EV} of the Supplementary Material, is
\begin{align}
	g(x_1,x_2) &= \left(2+s_1\mathbbm{1}_{\{x_1\geq x_2\}}+s_2\mathbbm{1}_{\{x_1< x_2\}}\right)\max(x_1,x_2)- \left(1+s_1\mathbbm{1}_{\{x_1\geq x_2\}}+s_2\mathbbm{1}_{\{x_1< x_2\}}\right)\min(x_1,x_2),
\label{eqn:BEVgauge}
\end{align}
$x_1,x_2\geq 0$. For the logistic model with dependence parameter $\alpha\in (0,1)$, we have $s_1=s_2=1/\alpha-2$. Hence the gauge function is
\[
	g(x_1,x_2)=\frac{1}{\alpha}\max(x_1,x_2) + \left(1-\frac{1}{\alpha}\right)\min(x_1,x_2),
\] 
for $x_1,x_2\geq 0$. In this case, the point $(x_1,x_2)=(1,1)$ satisfies $g(x_1,x_2)=1$, and since both variables are at most 1 in the set $G$, we must have $\eta_{\{1,2\}}=1$, which is the only possible value under the known asymptotic dependence of this model. This is demonstrated in case~(v) of Fig.~\ref{fig:bivariateGauge} for $\alpha=0.5$, with $G$ being piecewise linear. In the case where $s_1\neq s_2$ in \eqref{eqn:BEVgauge}, the set $G$ will no longer be symmetric about the line $x_1=x_2$, but will still correspond to two straight lines with intercepts with the axes at $\left((s_1+2)^{-1},0\right)$ and $\left(0,(s_2+2)^{-1}\right)$, and intersection at the point $(1,1)$. This still corresponds to $\eta_{\{1,2\}}=1$.

We finally consider the asymmetric logistic model \citep{Tawn1988} with exponent measure 
\begin{align*}
	V(x,y)= \theta_1/x + \theta_2/y + \left[\left\{(1-\theta_1)/x\right\}^{1/\alpha}+\left\{(1-\theta_2)/y\right\}^{1/\alpha}\right]^\alpha,
\end{align*}
with $x,y>0$, $\alpha\in(0,1]$ and $\theta_1, \theta_2 \in [0,1]$. This model does not satisfy the condition used when calculating the gauge function for bivariate extreme value copulas, that the spectral density places no mass on $\{0\}$ or $\{1\}$, since $H\left(\{0\}\right)=\theta_2$ and $H\left(\{1\}\right)=\theta_1$. However, calculating the gauge function for this model directly, we obtain
\[
	g(x_1,x_2) = \min\left\{(x_1+x_2) ; \frac{1}{\alpha}\max(x_1,x_2) + \left(1-\frac{1}{\alpha}\right)\min(x_1,x_2) \right\},
\]
for $x_1,x_2\geq 0$ and all $\alpha,\theta_1,\theta_2\in(0,1)$. We note that this gauge function does not depend on the values of $\theta_1$ and $\theta_2$, i.e., the mass on the boundaries of $H$. This is demonstrated by case (vi) of Fig.~\ref{fig:bivariateGauge} for $\alpha=0.5$, and we again find that since $g(1,1)=1$, the coefficient of tail dependence has value $\eta_{\{1,2\}}=1$. The bivariate asymmetric logistic copula is essentially a mixture of independence and logistic models of cases~(i)~and~(v); this is reflected in the gauge function, which is a combination of the gauge functions corresponding to the two mixture components. 

\EDIT{The geometric approach for deriving extremal properties from the gauge function does extend to cases where a joint distribution has singular components, i.e., mass on lower dimensional subspaces. \ES{While the density representation is convenient when it exists, more general theory for deriving the limit set is available; see e.g., \cite{Balkemaetal2010}.} Examples with singular components include perfect dependence (case~(iv) of Fig.~\ref{fig:bivariateGauge}), and the copula of the Marshall-Olkin distribution \citep{MarshallOlkin1967}, which arises as the limit of the asymmetric logistic distribution as $\alpha\rightarrow 0$. For this copula, and a bivariate extreme value copula with underlying measure $H$ placing all of its mass at a finite set of atoms, the set $G$ is identical to that of the independence case, but with a line $y=x$ from $(1/2,1/2)$ to $(1,1)$, inclusive. As the set $G$ contains $(1,1)$, i.e., $g(1,1)=1$, we have $\eta_{\{1,2\}}=1$. For a copula based on the inverted bivariate extreme value distribution with underlying measure $H$ placing all its mass at a finite number of atoms, the $\eta_{\{1,2\}}$ value follows from \cite{Ledford1996}, with $\eta_{\{1,2\}}<1$ (unless $H(\{1/2\})=1$), and the asymptotic shape of the sample cloud following from results in \cite{Kereszturi2017}. However, in such cases, extremal properties are much easier to derive directly from the joint distribution function, without studying the gauge function.}

In the bivariate examples we have studied here, the intersection of interest between the sets $[r,\infty)^2$ and $G$ occurred when $x_1$ and $x_2$ were equal, but we note that this is not the case in general. For sets $G$ corresponding to cases~(i)-(v) are all convex, but this is not true of the asymmetric logistic model in case~(vi). This links to another interesting feature of the sets $G$, which is that they can be used to consider the possible values of one variable when the other variable is large. To study the case where $X_1$ takes its largest values, we can consider the intersection of the set $G$ with the line $x_1=1$. For the independence and inverted logistic examples, cases (i)~and~(iii), we see that the intersection occurs at $(1,0)$, so the largest values of $X_1$ occur only with the smallest values of $X_2$, while for the Gaussian case~(ii), the intersection occurs at $(1,\rho^2)$, meaning that larger (although not the most extreme) values of $X_2$ occur when $X_1$ takes its largest values. For all three asymptotically dependent cases, the two variables take their largest values simultaneously, with intersection at $(1,1)$, but for the asymmetric logistic example of case~(vi), the line $x_1=1$ intersects the set $G$ twice, indicating that $X_2$ can take either its smallest \emph{or} largest values with the largest values of $X_1$. \cite{Nolde2020} elaborate further on how the shape of $G$ links to a broader description of extremal dependence than the coefficients $\eta_\mathcal{C}$.

\section{Vine copula modelling}\label{sec:vines}
\subsection{Preliminaries}
As discussed in Section~\ref{sec:intro}, our aim is to apply the methods of Section~\ref{sec:gauge} to investigate some of the extremal dependence properties of vine copulas. \EDIT{These are models for $d>2$ variables, created using $d(d-1)/2$ bivariate copulas according to an underlying graphical structure. We provide a summary of the key ideas here, where our focus is on models for continuous variables.}

By Sklar's theorem \citep{Sklar1959}, the joint distribution function $F$ of variables $\bm{X}=(X_1,\ldots,X_d)$ with $X_i\sim F_i$, for $i=1,\ldots,d$, can be written in terms of a unique copula function $C$ as
\[
	F(x_1,\ldots,x_d) = C\left\{F_1(x_1),\ldots,F_d(x_d)\right\}, ~~~\text{for $x_i\in\mathbb{R},~i=1,\ldots,d$}.
\]
Differentiating this with respect to each variable gives the joint density function as
\begin{align}
	f(x_1,\ldots,x_d) = c\left\{F_1(x_1),\ldots,F_d(x_d)\right\}\prod_{i=1}^d f_i(x_i),
	\label{eqn:copulaDensity}
\end{align}
for $f_{i}(x_i)$, $i=1,\ldots,d$, representing the marginal densities, and copula density
\[
	c(u_1,\ldots,u_d)=\partial^dC(u_1,\ldots,u_d)/\prod_{i=1}^d \partial u_i, ~~~~~\text{with $u_i\in[0,1]$ for $i=1,\ldots,d$}.
\]
As outlined by \cite{Aas2009}, the joint density can be decomposed as
\begin{align}
	f(x_1,\ldots,x_d) = f_d(x_d)&\cdot f_{d-1\mid d}(x_{d-1}\mid x_d)\cdot f_{d-2\mid d-1,d}(x_{d-2}\mid x_{d-1},x_d)\ldots f_{1\mid 2,\ldots,d}(x_1\mid x_2,\ldots,x_d),
	\label{eqn:decomposedDensity}
\end{align}
and by repeatedly applying decomposition (\ref{eqn:copulaDensity}) to each term in the right-hand side of (\ref{eqn:decomposedDensity}), it is possible to write the joint density of the variables $\bm{X}$ in terms of only marginal and bivariate copula densities. For instance, in the bivariate case, $f(x_1,x_2) = f_2(x_2)\cdot f_{1\mid 2}(x_1\mid x_2)$ from (\ref{eqn:decomposedDensity}) and $f(x_1,x_2)=f_1(x_1)\cdot f_2(x_2) \cdot c_{12}\left\{F_1(x_1),F_2(x_2)\right\}$ from (\ref{eqn:copulaDensity}), so that 
\[
	f_{1\mid 2}(x_1\mid x_2) = f_1(x_1)\cdot c_{12}\left\{F_1(x_1),F_2(x_2)\right\}.
\]
Similarly, in the trivariate case,
\begin{align*}
f(x_1,x_2,x_3) &=f_3(x_3)\cdot f_{2\mid 3}(x_2\mid x_3) \cdot f_{1\mid 23}(x_1\mid x_2,x_3)\\
&=f_3(x_3)\cdot f_2(x_2)\cdot c_{23}\left\{F_2(x_2),F_3(x_3)\right\}\cdot f_{1\mid 23}(x_1\mid x_2,x_3).
\end{align*}
Again following \cite{Aas2009}, 
\begin{align*}
	f_{1\mid 23}(x_1\mid x_2,x_3) &= c_{13\mid 2}\left\{F_{1\mid 2}(x_1\mid x_2),F_{3\mid 2}(x_3\mid x_2)\right\}\cdot f_{1\mid 2}(x_1\mid x_2)\\
	&=c_{13\mid 2}\left\{F_{1\mid 2}(x_1\mid x_2),F_{3\mid 2}(x_3\mid x_2)\right\}\cdot c_{12}\left\{F_1(x_1),F_2(x_2)\right\}\cdot f_1(x_1),
\end{align*}
so that a full decomposition of $f(x_1,x_2,x_3)$ is given by
\begin{align}
f(x_1,x_2,x_3) = &f_1(x_1) \cdot f_2(x_2) \cdot f_3(x_3) \nonumber\\
			&\cdot c_{12}\left\{F_1(x_1),F_2(x_2)\right\}\cdot c_{23}\left\{F_2(x_2),F_3(x_3)\right\}\nonumber \\
			&\cdot c_{13|2}\left\{F_{1|2}(x_1|x_2),F_{3|2}(x_3|x_2)\right\}.
\label{eqn:trivariateVine}
\end{align}
For modelling purposes, different bivariate copula densities can be chosen for each of $c_{12}$, $c_{23}$ and $c_{13\mid 2}$, and different marginal distributions can be selected for each variable, i.e., $F_1$, $F_2$, and $F_3$,  showing the flexibility in this class of model. A similar process can be applied to obtain models in higher than three dimensions in terms of bivariate copulas.

\EDIT{The decomposition of density $f$ is not unique, as we have a choice about the conditioning variable used in each step of the decomposition. 
%The impact of the choice of decomposition on the extremal dependence features of the variables is not currently understood. We provide some first stage investigations for constructions involving bivariate inverted extreme value copulas. 
\cite{Bedford2001,Bedford2002} proposed an approach to address this issue through the use of regular vines, a class of graphical model, to represent the underlying structure of certain decompositions and help to systematize the different possibilities. Construction~\eqref{eqn:trivariateVine} gives an example of a vine copula form in the trivariate case. An introduction to the graphical representation of vine copulas is given in \cite{Kurowicka2010_Ch3}, with formal definitions provided in \cite{Kurowicka2006}.} We discuss this further in Section~\ref{subsec:vineReps}.

\subsection{Graphical representations of vine copulas}\label{subsec:vineReps}
Suppose we are interested in modelling variables $\bm{X}$. A regular vine, first introduced by \cite{Bedford2002}, corresponding to these $d$ variables consists of $d-1$ connected trees labelled $T_1,\ldots,T_{d-1}$, with tree $T_i$ having $d+1-i$ nodes and $d-i$ edges. The nodes in tree $T_1$ each have a different label in the set $\mathcal{D}$, and the edges are labelled according to the pair of nodes they connect. The labels of the nodes in tree $T_{i+1}$ correspond to the labels of the edges in tree $T_i$, for $i=1,\ldots,d-2$, creating a nested structure among the set of all trees. In tree $T_i$, $i\geq 2$, the pair of nodes connected by each edge will have $i-1$ variable labels in common; these become the conditioning variables in the corresponding edge label of $T_i$. The underlying vine structure for copula~\eqref{eqn:trivariateVine} is shown in Fig.~\ref{fig:triVine}.

\begin{figure}
\begin{center}
\includegraphics[width=0.3\textwidth]{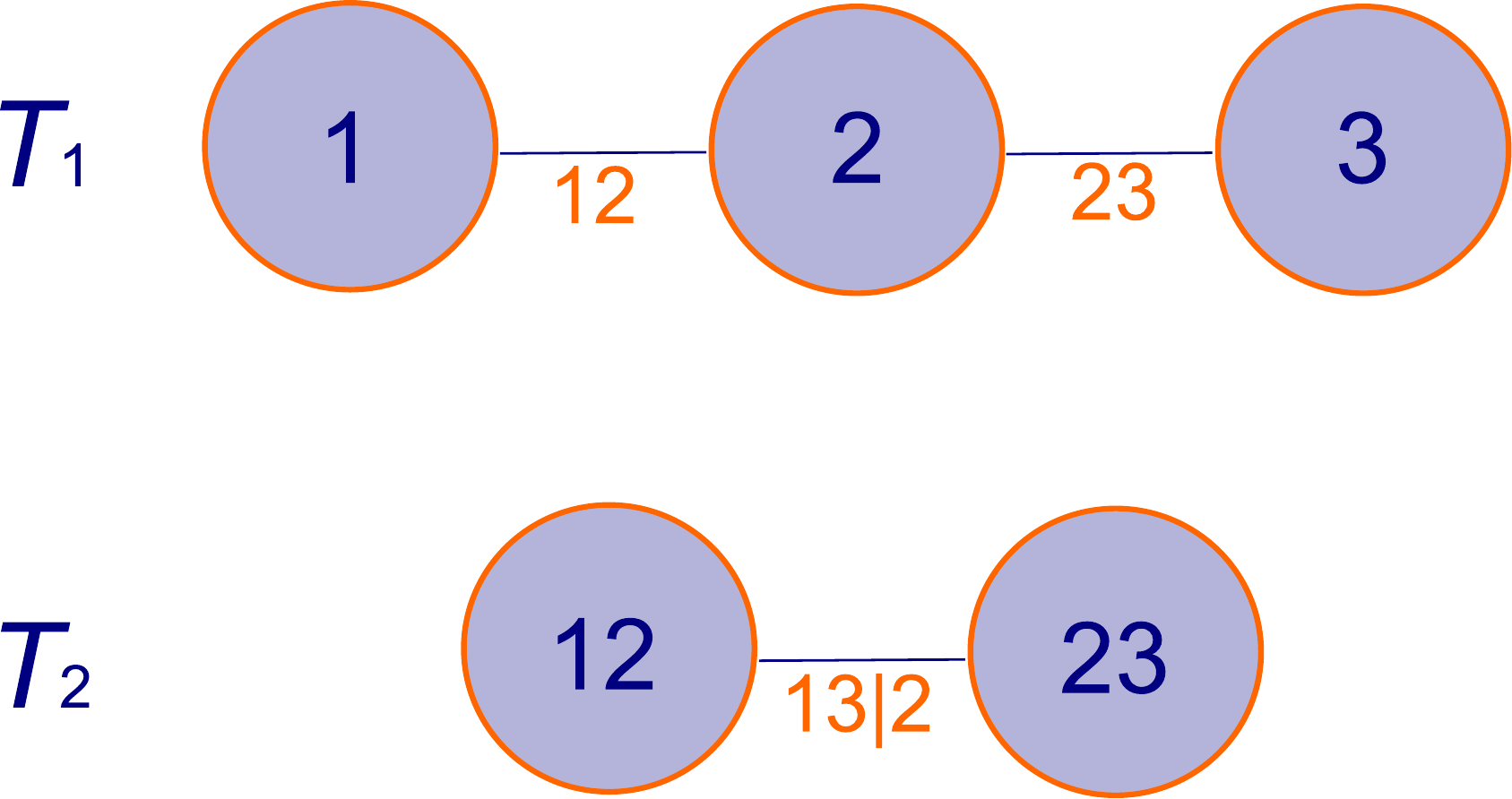}
\caption{Trivariate vine structure.}
\label{fig:triVine}
\end{center}
\end{figure}

\EDIT{Each edge in a regular vine can be used to represent one of the copula densities used in the decomposition of the joint density. There are certain subclasses of vine copula that are often of interest. These include $D$-vines, where each tree is a path, and $C$-vines, where each tree has exactly one node that is connected to all other nodes.} Fig.~\ref{fig:4D_Vines} gives an example of these vine structures for $d=4$. For the $C$-vine example, the corresponding decomposition of the density is
\begin{align*}
	f(x_1,x_2,x_3,x_4) = &f_1(x_1)\cdot f_2(x_2)\cdot f_3(x_3)\cdot f_4(x_4)\\
			&\cdot c_{12}\left\{F_1(x_1),F_2(x_2)\right\}\cdot c_{13}\left\{F_1(x_1),F_3(x_3)\right\}\cdot c_{14}\left\{F_1(x_1),F_4(x_4)\right\}\\
			&\cdot c_{23\mid 1}\left\{F_{2\mid 1}(x_2\mid x_1),F_{3\mid 1}(x_3\mid x_1)\right\}\cdot c_{24\mid 1}\left\{F_{2\mid 1}(x_2\mid x_1),F_{4\mid 1}(x_4\mid x_1)\right\}\\
			&\cdot c_{34\mid 12}\left\{F_{3\mid 12}(x_3\mid x_1,x_2),F_{4\mid 12}(x_4\mid x_1,x_2)\right\},
\end{align*}
with the result for the $D$-vine found in a similar way. More detail on regular vines and the subclasses of $D$-vines and $C$-vines can be found in \cite{Kurowicka2010_Ch3}.

\begin{figure}
	\begin{center}
		\begin{subfigure}{.45\textwidth}
			\vspace{1cm}
			\centering
 			\includegraphics[width=.8\linewidth]{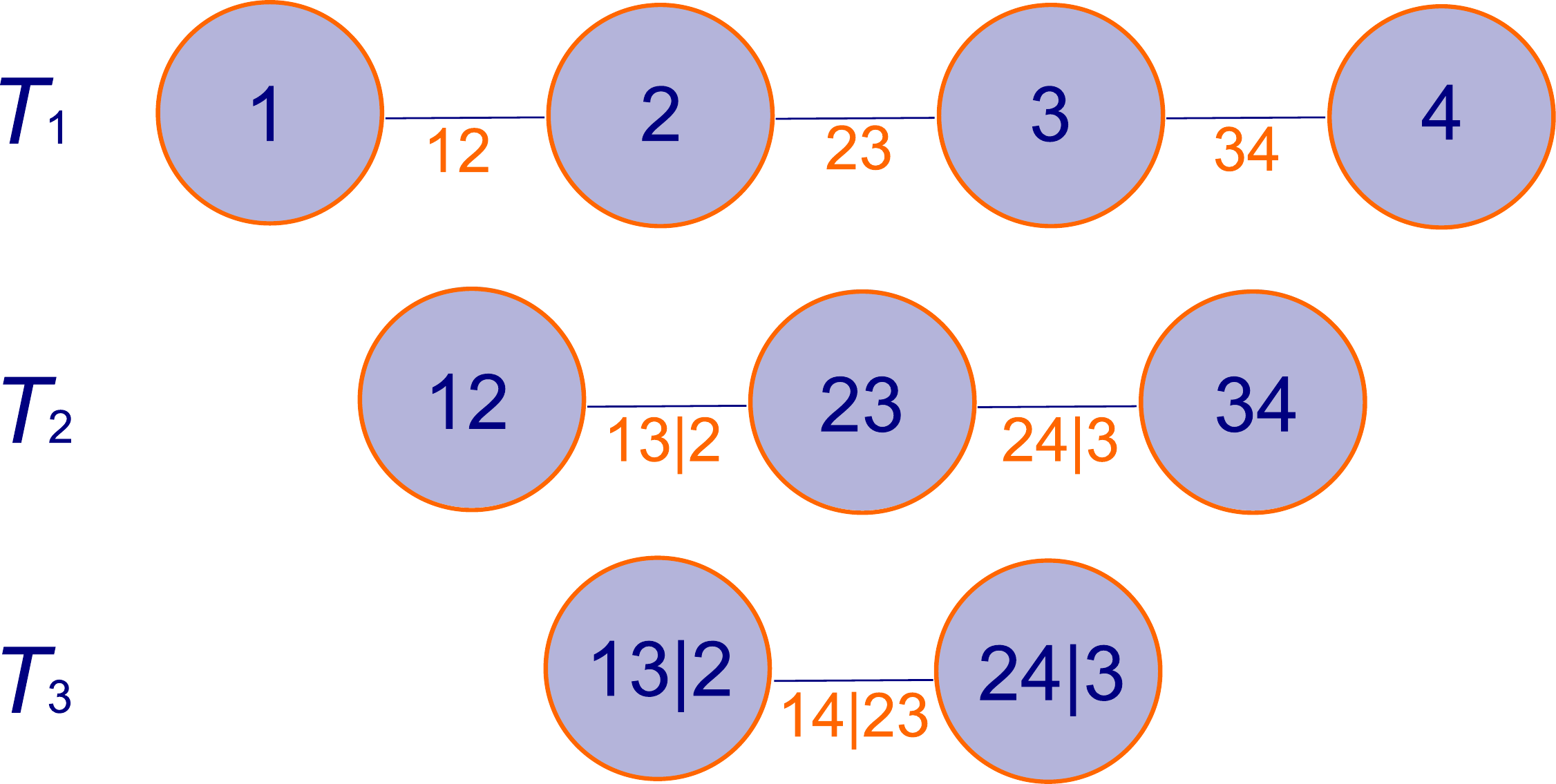}
		\end{subfigure}%
		\begin{subfigure}{.45\textwidth}
			\centering
  			\includegraphics[width=.8\linewidth]{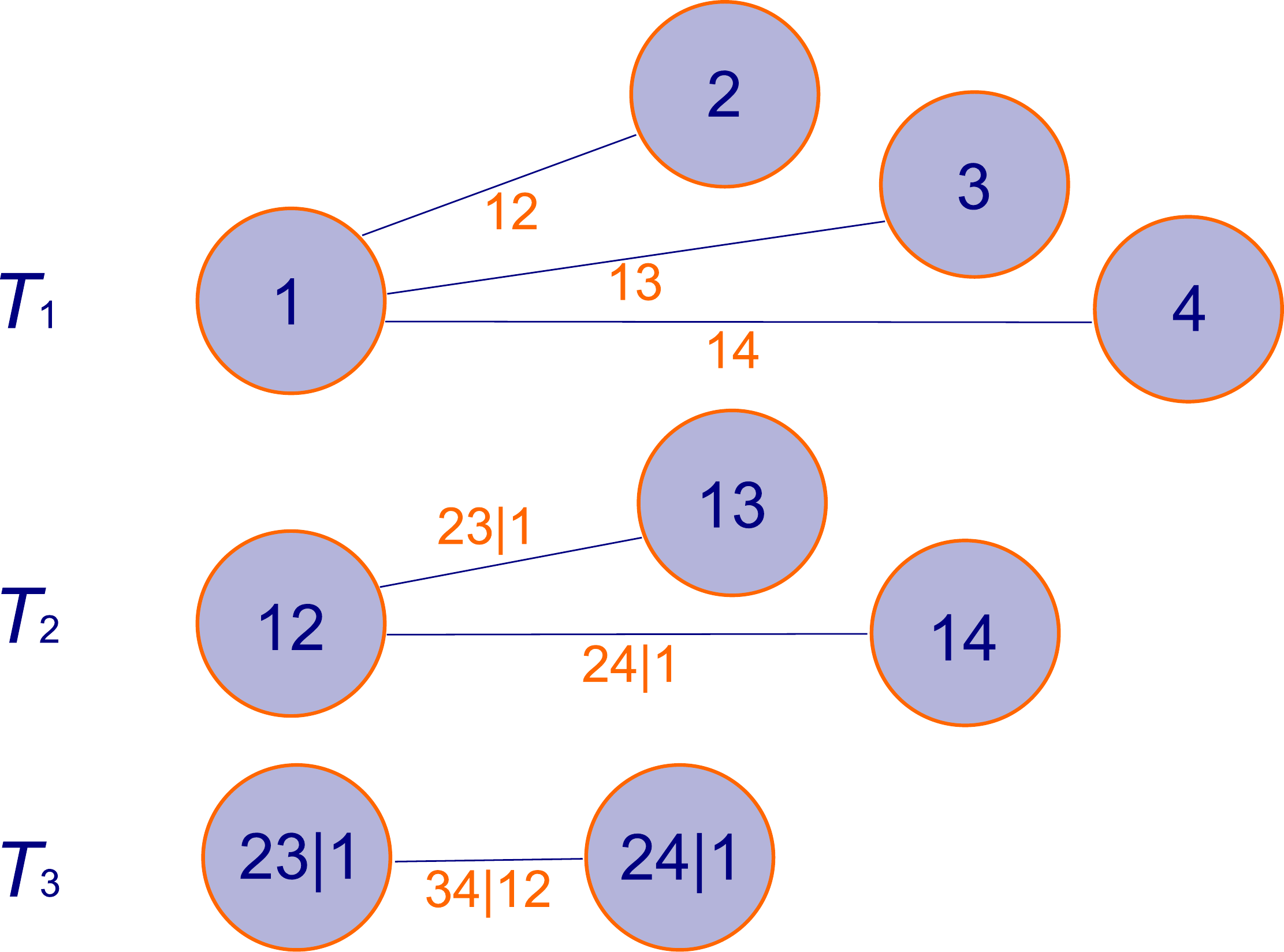}
		\end{subfigure}
	\caption{Four dimensional vine copula models; $D$-vine (left) and $C$-vine (right).}
	\label{fig:4D_Vines}
	\end{center}
\end{figure}

For modelling $d$ variables, there are $d!/2$ possible $D$-vines, and the same number of possible $C$-vines \citep{Aas2009}. For $d=3$, all vine structures are equivalent, with different decompositions only occurring with different labelling of the nodes. For $d=4$, all possible structures fall into either the $D$-vine or $C$-vine category. For $d\geq 5$, the more general regular vines provide a greater range of possible structures, with structure selection considered by \cite{Dissmann2013,Zhu2020}, for example,  but we only study $D$-vines and $C$-vines here.

\section{Vine copulas with inverted extreme value copula components}\label{sec:vinesIEV}
\subsection{Trivariate gauge calculation}\label{subsec:trivariateIEV}
We now turn our attention to applying the methods discussed in Section~\ref{sec:gauge} to calculate the coefficient of tail dependence for various vine copulas, initially focusing on cases where all bivariate copulas used in the construction belong to the family of inverted extreme value models. Our first vine copula gauge function calculation is for a trivariate vine, with graphical structure as in Fig.~\ref{fig:triVine} and density~\eqref{eqn:trivariateVine}, constructed from three inverted extreme value pair copulas.

A bivariate inverted extreme value copula with exponent measure $V$ has the form
\[
	C(u,v) = u + v - 1 + \exp\left[-V\left\{\frac{-1}{\ln(1-u)},\frac{-1}{\ln(1-v)}\right\}\right],~~~u,v\in[0,1].
\]
Let $V_1$, $V_2$ and $V_{12}$ denote the derivative of the exponent measure with respect to the first, second, and both components, respectively. Differentiating $C(u,v)$ with respect to the second component gives the conditional distribution function
\begin{align}
	F(u\mid v) = 1 + &\left(\frac{1}{1-v}\right)\left\{-\ln(1-v)\right\}^{-2}V_2\left\{\frac{-1}{\ln(1-u)},\frac{-1}{\ln(1-v)}\right\}\cdot\exp\left[-V\left\{\frac{-1}{\ln(1-u)},\frac{-1}{\ln(1-v)}\right\}\right],
\label{eqn:conditional}
\end{align}
for $u,v\in[0,1]$, and subsequently differentiating with respect to the first component gives the copula density
\begin{align}
	c(u,v) = &\left(\frac{1}{1-u}\right)\left(\frac{1}{1-v}\right)\left\{-\ln(1-u)\right\}^{-2}\left\{-\ln(1-v)\right\}^{-2}\cdot\exp\left[-V\left\{\frac{-1}{\ln(1-u)},\frac{-1}{\ln(1-v)}\right\}\right]\nonumber\\
	&\cdot\bigg[V_1\left\{\frac{-1}{\ln(1-u)},\frac{-1}{\ln(1-v)}\right\}V_2\left\{\frac{-1}{\ln(1-u)},\frac{-1}{\ln(1-v)}\right\}-V_{12}\left\{\frac{-1}{\ln(1-u)},\frac{-1}{\ln(1-v)}\right\}\bigg].
\label{eqn:IEVcopuladensity}
\end{align}
In calculating values of $\eta$ for a trivariate vine with density~\eqref{eqn:trivariateVine}, we are interested in the behaviour of
\begin{align}
-\ln f(t\bm{x})= &-\ln f_1(tx_1) -\ln f_2(tx_2) -\ln f_3(tx_3)\nonumber\\
				  &-\ln c_{12}\left\{F_1(tx_1),F_2(tx_2)\right\} -\ln c_{23}\left\{F_2(tx_2),F_3(tx_3)\right\}\nonumber\\
 				  &-\ln c_{13|2}\left\{F_{1|2}(tx_1|tx_2),F_{3|2}(tx_3|tx_2)\right\},
\label{eqn:negLogTriVine}
\end{align}
for $x_1,x_2,x_3\geq 0$, as $t\rightarrow\infty$. In Section~\ref{SMsec:triIEVproof} of the Supplementary Material, we show that the gauge function of a trivariate vine copula with three inverted extreme value components is 
\begin{align}
 g(\bm{x}) = x_2 + V^{\{13\mid 2\}}\left[\left\{V^{\{12\}}\left(x_1^{-1},x_2^{-1}\right)-x_2\right\}^{-1},\left\{V^{\{23\}}\left(x_2^{-1},x_3^{-1}\right)-x_2\right\}^{-1} \right],~~~x_1,x_2,x_3\geq 0,
\label{eqn:ievievievgauge}
\end{align}
where the superscripts of the exponent measures $V^{\{12\}}, V^{\{23\}}$ and $V^{\{13\mid 2\}}$ correspond to the pair copulas used to construct the vine. 

\EDIT{We note that a general exponent measure $V(x,y)$ is non-increasing in $x$ and $y$, so it follows that $V(x^{-1},y^{-1})$ is non-decreasing in $x$ and $y$. From this, we can deduce that~\eqref{eqn:ievievievgauge} is non-decreasing in $x_1$ and $x_3$, so the minimum required to solve equation~\eqref{eqn:NWetaGauge} must occur when $x_1=x_3=1$. For $x_2$, the problem is more subtle. In the following section, we consider an example where all components are taken to be inverted logistic copulas, with the form of their exponent measures given by~\eqref{eqn:Vlogistic}. In this case, we demonstrate that the minimum also occurs at $x_2=1$, and suggest that a similar approach could be taken for other cases.}

\subsubsection{Inverted logistic example}
Let $V^{\{12\}}$, $V^{\{23\}}$ and $V^{\{13\mid 2\}}$ have dependence parameters $\alpha,\beta,\gamma\in(0,1)$, respectively. Then the corresponding gauge function is
\begin{align}
	g(\bm{x}) = x_2 + \left[\left\{\left(x_1^{1/\alpha}+x_2^{1/\alpha}\right)^\alpha - x_2\right\}^{1/\gamma} + \left\{\left(x_2^{1/\beta}+x_3^{1/\beta}\right)^\beta - x_2\right\}^{1/\gamma}\right]^\gamma,~~~x_1,x_2,x_3\geq 0.
	\label{eqn:ilogilogilogGauge}
\end{align} 
Fig.~\ref{fig:ilogilogilogGauge} demonstrates the sets $G=\{\bm{x}\in\mathbb{R}^3:g(\bm{x})=1\}$ for this gauge function, with $\alpha\in\{0.25,0.5,0.75\}$, $\beta=0.25$ and $\gamma=0.5$. As was the case with the bivariate inverted logistic copula, the surface corresponding to the set $G$ is smooth and convex, and considering the intersection of $G$ with the lines $x_i=1$, $i=1,2,3$, shows that each variable takes its largest values while the other two take their smallest.

\begin{figure}
\begin{center}
	\includegraphics[width=\textwidth]{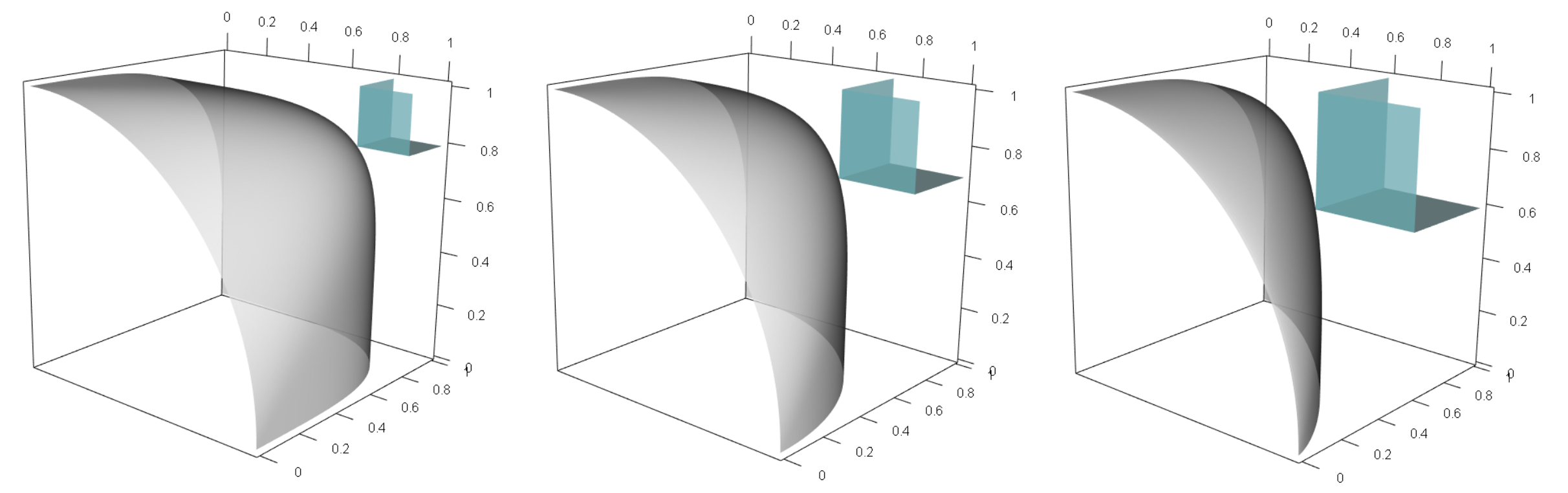}
	\caption{The set $G=\{\bm{x}\in\mathbb{R}^3:g(\bm{x})=1\}$ for a trivariate vine with three inverted logistic pair copula components (grey) and the boundary of the set $[\eta_{\{1,2,3\}},1)^3$ (blue): $\alpha=0.25$ (left), $\alpha=0.5$ (centre), $\alpha=0.75$ (right); $\beta=0.25$ and $\gamma=0.5$.}
	\label{fig:ilogilogilogGauge}
\end{center}
\end{figure}

\EDIT{From Section~\ref{subsec:trivariateIEV}, we already know that the minimum in~\eqref{eqn:NWetaGauge} occurs when $x_1=x_3=1$, since $g(\bm{x})$ is increasing with respect to both these variables. In Section~\ref{SMsec:4.1proofs} of the Supplementary Material, we show that the gauge function is also increasing with respect to $x_2\geq 1$. Hence, we know that the minimum occurs at $\bm{x}=\bm{1}$, i.e., that the intersection of $G$ and $[\eta_{\{1,2,3\}},\infty)^3$ occurs on the diagonal $x_1=x_2=x_3$. This is supported by the plots in Fig.~\ref{fig:ilogilogilogGauge}, and yields}
\begin{align}
	\eta_{\{1,2,3\}}=g(1,1,1)^{-1}= \left[1+\left\{\left(2^\alpha-1\right)^{1/\gamma}+\left(2^\beta-1\right)^{1/\gamma}\right\}^\gamma\right]^{-1}.
\label{eqn:ilogilogilogEta123}
\end{align} As $\min\left(\alpha,\beta,\gamma\right)\rightarrow 1$, $\eta_{\{1,2,3\}}\rightarrow 1/3$, corresponding to complete independence, and as $\max\left(\alpha, \beta, \gamma\right) \rightarrow 0$, $\eta_{\{1,2,3\}}\rightarrow 1$.

\subsubsection{Calculation of $\eta_{\{1,3\}}$}\label{subsec:eta13IEV}
We now consider the coefficient of tail dependence for the variables $X_1$ and $X_3$, i.e., the pair that is not directly linked in tree $T_1$ of the underlying vine. The joint density of $X_1$ and $X_3$ cannot be found analytically for a trivariate vine with inverted logistic pair copula components; we therefore use the method discussed in Section~\ref{subsec:lowerDimEta}, with the gauge function for this pair of variables being $g_{\{1,3\}}(x_1,x_3)=\min_{x_2\geq 0}g(\bm{x})$, for $g(\bm{x})$ in \eqref{eqn:ilogilogilogGauge}. \EDIT{To demonstrate the boundary of the scaled sample cloud, we carry out this minimisation numerically. In the left panel of Fig.~\ref{fig:ilogilogilogGauge13}, we plot the set $G_{\{1,3\}}$ for $(\alpha,\beta,\gamma)=(0.5,0.25,0.5)$, chosen to match the parameter values for the central panel of Fig.~\ref{fig:ilogilogilogGauge}.} 

%As for the $\eta_{\{1,2,3\}}$ case, the intersection of $G_{\{1,3\}}$ and $[\eta_{\{1,3\}},\infty)^2$ occurs on the diagonal, i.e., when $x_1=x_3$. Further investigations suggest this is generally the case for vine copulas constructed from inverted logistic pair copulas, but not necessarily for those also involving logistic components, as we will see in Section~\ref{subsec:trivariateGauges}.

\begin{figure}[t]
\centering
\begin{subfigure}{.5\textwidth}
  \centering
\includegraphics[width=0.7\textwidth]{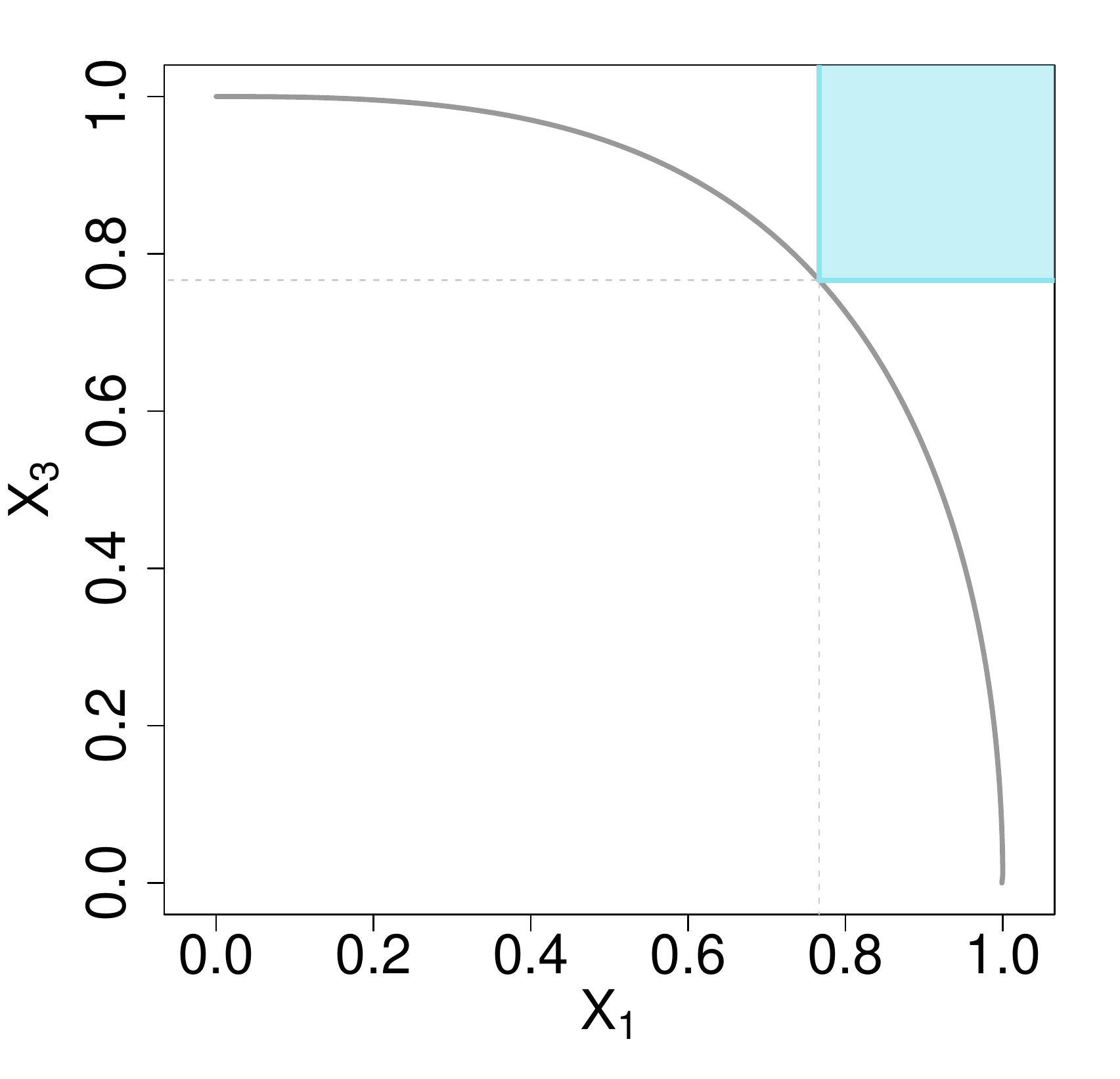}
\end{subfigure}%
\begin{subfigure}{.5\textwidth}
  \centering
  \includegraphics[width=0.7\textwidth]{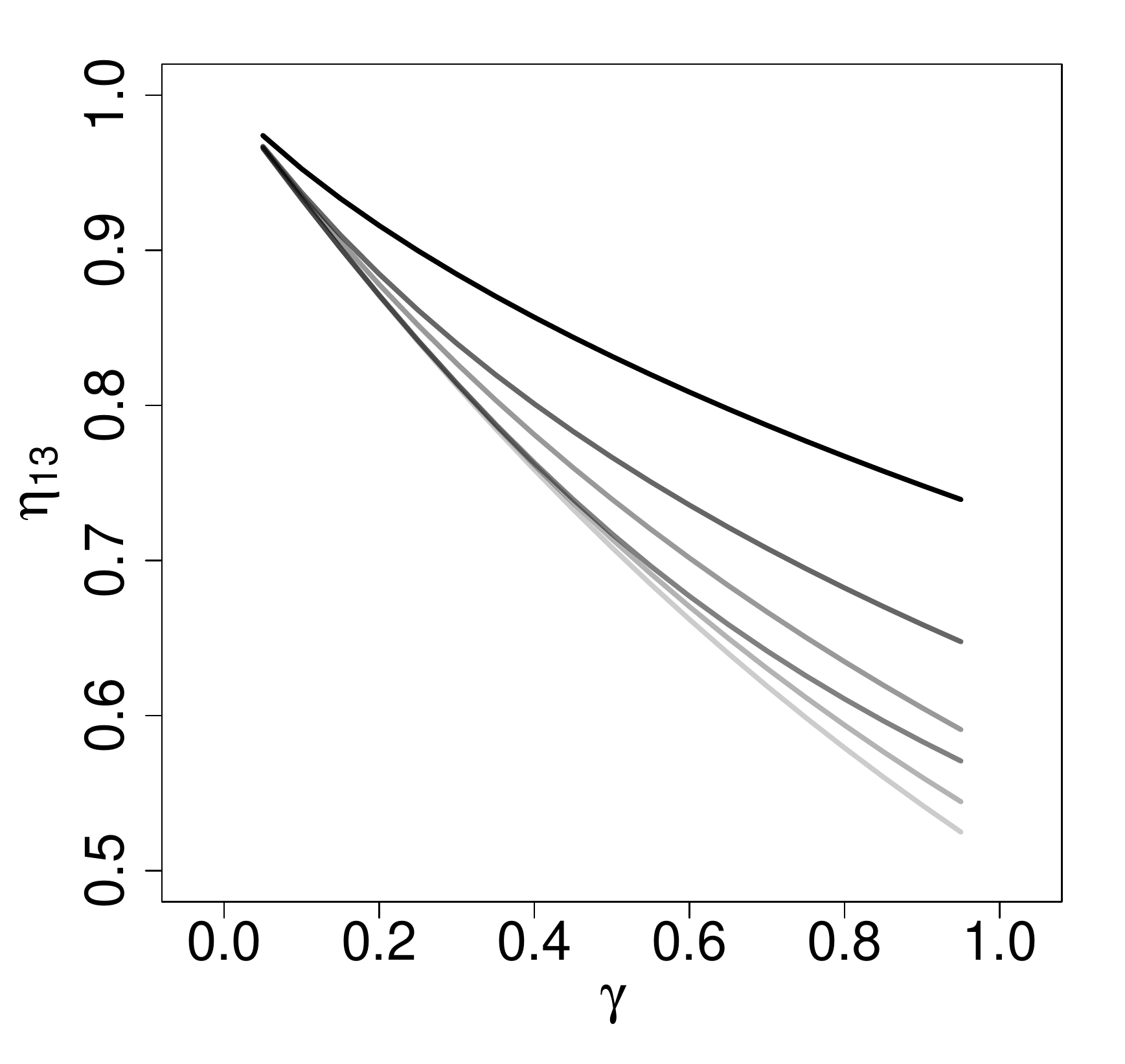}
\end{subfigure}
\caption{Left: for $(\alpha,\beta,\gamma)=(0.5,0.25,0.5)$, the sets $G_{\{1,3\}}$ (grey) and $[\eta_{\{1,3\}},1)^2$ (blue). Right: values of $\eta_{\{1,3\}}$ for $\gamma\in[0.05,0.95]$ and $(\alpha,\beta)\in\{(0.25,0.25)$,$(0.25,0.5)$,$(0.5,0.5)$,$(0.25,0.75)$,$(0.5,0.75)$,$(0.75,0.75)\}$ (top to bottom).}
\label{fig:ilogilogilogGauge13}
\end{figure}

\EDIT{To calculate $\eta_{\{1,3\}}$, we follow the steps in Section~\ref{subsec:lowerDimEta}, where we have $\eta_{\{1,3\}}=\left\{\min_{x_1,x_3\geq 1}g_{\{1,3\}}(x_1,x_3)\right\}^{-1}=\left\{\min_{x_1,x_3\geq 1,x_2\geq 0}g(x_1,x_2,x_3)\right\}^{-1} $. We have already seen that $g(\bm{x})$ is increasing with respect to $x_1$ and $x_3$, so we focus on $x_1=x_3=1$, and have $\eta_{\{1,3\}}=\left\{\min_{v\geq 0}g(1,v,1)\right\}^{-1}$.} That is,
\[
\eta_{\{1,3\}} = \left(v + \left[\left\{\left(1+v^{1/\alpha}\right)^\alpha-v\right\}^{1/\gamma} + \left\{\left(1+v^{1/\beta}\right)^\beta-v\right\}^{1/\gamma} \right]^\gamma\right)^{-1},
\]
with $v$ satisfying $d g(1,v,1)/d v=0$, i.e.,
\begin{align}
	&1+\left[\left\{\left(1+v^{1/\alpha}\right)^\alpha-v\right\}^{1/\gamma} + \left\{\left(1+v^{1/\beta}\right)^\beta-v\right\}^{1/\gamma}\right]^{\gamma-1}\cdot\bigg[ \left\{(1+v^{-1/\alpha})^{\alpha-1}-1\right\} \left\{(1+v^{1/\alpha})^\alpha-v\right\}^{-1+1/\gamma}\nonumber\\
	&\hspace{3cm} + \left\{(1+v^{-1/\beta})^{\beta-1}-1\right\}\left\{(1+v^{1/\beta})^\beta-v\right\}^{-1+1/\gamma}\bigg] = 0.
	\label{eqn:ilogilogilog_tconstraint}
\end{align}
\EDIT{In Section~\ref{SMsec:4.1proofs} of the Supplementary Material, we show that~\eqref{eqn:ilogilogilog_tconstraint} has a unique solution that lies in the range $(0,1)$.} In general, equation~\eqref{eqn:ilogilogilog_tconstraint} has no closed form solution, except in the case where $\alpha=\beta$, which leads to 
\[
v=\left\{(1-2^{-\gamma})^{-1/(1-\alpha)}-1\right\}^{-\alpha}~~;~~\eta_{\{1,3\}}=\frac{\left\{\left(1-2^{-\gamma}\right)^{-1/(1-\alpha)}-1\right\}^\alpha}{1 -2^\gamma + 2^\gamma\left(1-2^{-\gamma}\right)^{-\alpha/(1-\alpha)}},
\] 
but it can be solved numerically when $\alpha\neq\beta$. In the right panel of Fig.~\ref{fig:ilogilogilogGauge13}, we demonstrate the resulting value of $\eta_{\{1,3\}}$ for a variety of $\alpha,\beta$ and $\gamma$ values. \EDIT{We note that $\eta_{\{1,3\}}\in(0.5,1)$, revealing flexibility in the asymptotic independence features this model can capture. In particular, for the $\alpha=\beta$ case, $\eta_{\{1,3\}}=1-\gamma\ln 2 + o(\gamma)\nearrow 1$, as $\gamma\searrow 0$.}

\subsection{Higher dimensional results}\label{subsec:highDimIEV}
We now extend the results of Section~\ref{subsec:trivariateIEV} by considering vine copulas with dimension $d>3$ constructed from inverted extreme value pair copulas, with the aim being to find the gauge function and value of $\eta_\mathcal{D}$ in each case. We focus on copulas with two types of underlying structure: the class of vine copulas known as $D$-vines, where all trees in the vine are paths; and $C$-vines, which have exactly one node that is connected to all other nodes in each tree. These correspond to the two classes demonstrated in Fig.~\ref{fig:4D_Vines} for the case $d=4$. In the final part of this section, we demonstrate the values of $\eta_\mathcal{D}$ calculated using these gauge functions for both classes of model.

\subsubsection{Gauge functions for $D$-vines}\label{subsubsec:DvineGauges}
A $d$-dimensional $D$-vine is made up of $(d-1)$ trees, labelled $T_1,\ldots,T_{d-1}$, and a total of $(d-1)d/2$ edges. We suppose that the pair copula represented by each edge is an inverted extreme value copula, with the superscript on the exponent measure corresponding to the edge-label, as in the trivariate case. For the four-dimensional example in Fig.~\ref{fig:4D_Vines}, we have
\begin{align}
	-&\ln f(t\bm{x}) = -\ln f_1(tx_1) -\ln f_2(tx_2) -\ln f_3(tx_3) -\ln f_4(tx_4)\nonumber\\
				  &~~-\ln c_{12}\left\{F_1(tx_1),F_2(tx_2)\right\} -\ln c_{23}\left\{F_2(tx_2),F_3(tx_3)\right\} -\ln c_{34}\left\{F_3(tx_3),F_4(tx_4)\right\}\nonumber\\
 				  &~~-\ln c_{13|2}\left\{F_{1|2}(tx_1|tx_2),F_{3|2}(tx_3|tx_2)\right\}-\ln c_{24|3}\left\{F_{2|3}(tx_2|tx_3),F_{4|3}(tx_4|tx_3)\right\}\nonumber\\
 				  &~~-\ln c_{14|23}\left\{F_{1|23}(tx_1|tx_2,tx_3),F_{4|23}(tx_4|tx_2,tx_3)\right\}.
\label{eqn:4Dlogexpansion}
\end{align}
We note that several of these terms can be thought of in terms of lower-dimensional vine copulas that are subsets of the four-dimensional vine. In particular, all terms in the trivariate formula~\eqref{eqn:negLogTriVine} for the set of variables $(X_1,X_2,X_3)$ appear in~\eqref{eqn:4Dlogexpansion}. Let $f_{123}$ denote the joint density corresponding to this trivariate case. The density $f_{234}$ corresponding to variables $(X_2,X_3,X_4)$ also comes from a trivariate vine copula equivalent to $f_{123}$ up to a labelling of the variables. The sections of the four-dimensional vine corresponding to these two trivariate subsets are highlighted in Fig.~\ref{fig:4DVineSeparated}, and can be thought of as sub-vines of the overall vine copula.
We note that these two sub-vines overlap in the centre, as they share the variables $(X_2,X_3)$. This suggests that if we try to represent $-\ln f$ for the overall model in terms of $-\ln f_{123}$ and $-\ln f_{234}$, we will count the section corresponding to $-\ln f_{23}$ twice, with $f_{23}$ denoting the joint density of $(X_2,X_3)$. Taking this inclusion-exclusion into account, equation~\eqref{eqn:4Dlogexpansion} can be simplified to
\begin{align}
	-\ln f(t\bm{x}) = &-\ln f_{123}(tx_1,tx_2,tx_3) -\ln f_{234}(tx_2,tx_3,tx_4) + \ln f_{23}(tx_2,tx_3)\nonumber\\
	&-\ln c_{14|23}\left\{F_{1|23}(tx_1|tx_2,tx_3),F_{4|23}(tx_4|tx_2,tx_3)\right\}.
	\label{eqn:-logf_4d}
\end{align}

\begin{figure}[!htbp]
\begin{center}
\includegraphics[width=0.5\textwidth]{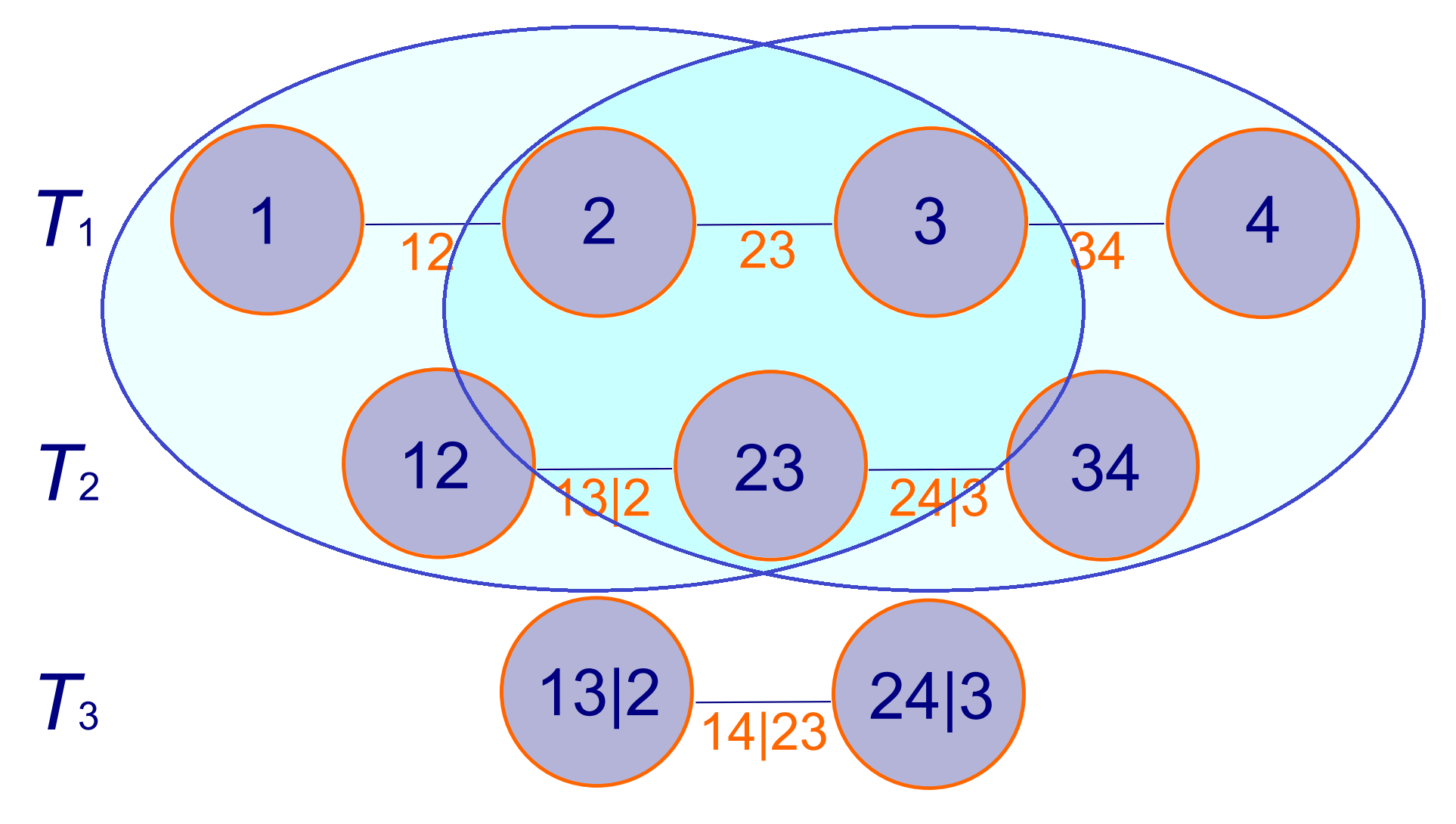}
\caption{Trivariate subsets of the four-dimensional $D$-vine.}
\label{fig:4DVineSeparated}
\end{center}
\end{figure}

In Section~\ref{SMsec:4Dvinegauge} of the Supplementary Material, we show that the gauge function can be written in terms of the gauge functions of the three sub-vines highlighted in Fig.~\ref{fig:4DVineSeparated} and the exponent measure corresponding to the pair copula in tree $T_3$. In particular,
\begin{align}
g(\bm{x}) &= g_{\{2,3\}}(x_2,x_3)+ V^{\{14|23\}}\left\{\frac{1}{g_{\{1,2,3\}}(x_1,x_2,x_3) - g_{\{2,3\}}(x_2,x_3)},\frac{1}{g_{\{2,3,4\}}(x_2,x_3,x_4) - g_{\{2,3\}}(x_2,x_3)}\right\},
\label{eqn:4-dim_D-vine_gauge}
\end{align}
for $x_1,x_2,x_3\geq 0$. For $D$-vine copulas, this same structure can be extended to higher dimensions, creating an iterative formula for calculating the gauge function; this is stated in Theorem~\ref{thm:d-dim_D-vine_gauge}.

\begin{thm}\label{thm:d-dim_D-vine_gauge}
The gauge function for a $d$-dimensional $D$-vine with inverted extreme value pair copula components is given by
\begin{align*}
	g&(\bm{x}) = g_{\mathcal{D}\backslash\{1,d\}}(\bm{x}_{-\{1,d\}})+V^{\{1,d\mid \mathcal{D}\backslash\{1,d\}\}}\left\{\frac{1}{g_{\mathcal{D}\backslash\{d\}}(\bm{x}_{-\{d\}})-g_{\mathcal{D}\backslash\{1,d\}}(\bm{x}_{-\{1,d\}})},\frac{1}{g_{\mathcal{D}\backslash\{1\}}(\bm{x}_{-\{1\}})-g_{\mathcal{D}\backslash\{1,d\}}(\bm{x}_{-\{1,d\}})}\right\},
\end{align*}
for $x_i\geq 0$, $i=1,\ldots,d$.
\end{thm}
\noindent Theorem~\ref{thm:d-dim_D-vine_gauge} is proved in the Appendix. We discuss how to obtain $\eta_\mathcal{D}$ in Section~\ref{subsubsec:etaDlargevines}.

\subsubsection{Gauge functions for $C$-vines}
Using similar arguments as for the $D$-vines in Section~\ref{subsubsec:DvineGauges}, we can construct an iterative formula for the gauge functions of $d$-dimensional $C$-vines. We now consider the sub-vines as corresponding to the sets of variables $\bm{X}_{-d}$ and $\bm{X}_{-(d-1)}$, which overlap at $\bm{X}_{-\{(d-1),d\}}$. This is demonstrated in Fig.~\ref{fig:4CVineSeparated} for the four-dimensional case. Following the same steps as in the previous section, we obtain the gauge function
\begin{align*}
	g(\bm{x}) =& g_{\mathcal{D}\backslash\{(d-1,d\}}(\bm{x}_{-\{d-1,d\}})\\ &+V^{\{d-1,d\mid \mathcal{D}\backslash\{d-1,d\}\}}\bigg\{\frac{1}{g_{\mathcal{D}\backslash\{d\}}(\bm{x}_{-\{d\}})-g_{\mathcal{D}\backslash\{d-1,d\}}(\bm{x}_{-\{d-1,d\}})},\frac{1}{g_{\mathcal{D}\backslash\{d-1\}}(\bm{x}_{-\{d-1\}})-g_{\mathcal{D}\backslash\{d-1,d\}}(\bm{x}_{-\{d-1,d\}})}\bigg\},
\end{align*}
with $x_i\geq 0$, $i=1,\ldots,d$.

\begin{figure}[!htbp]
\begin{center}
\includegraphics[width=0.5\textwidth]{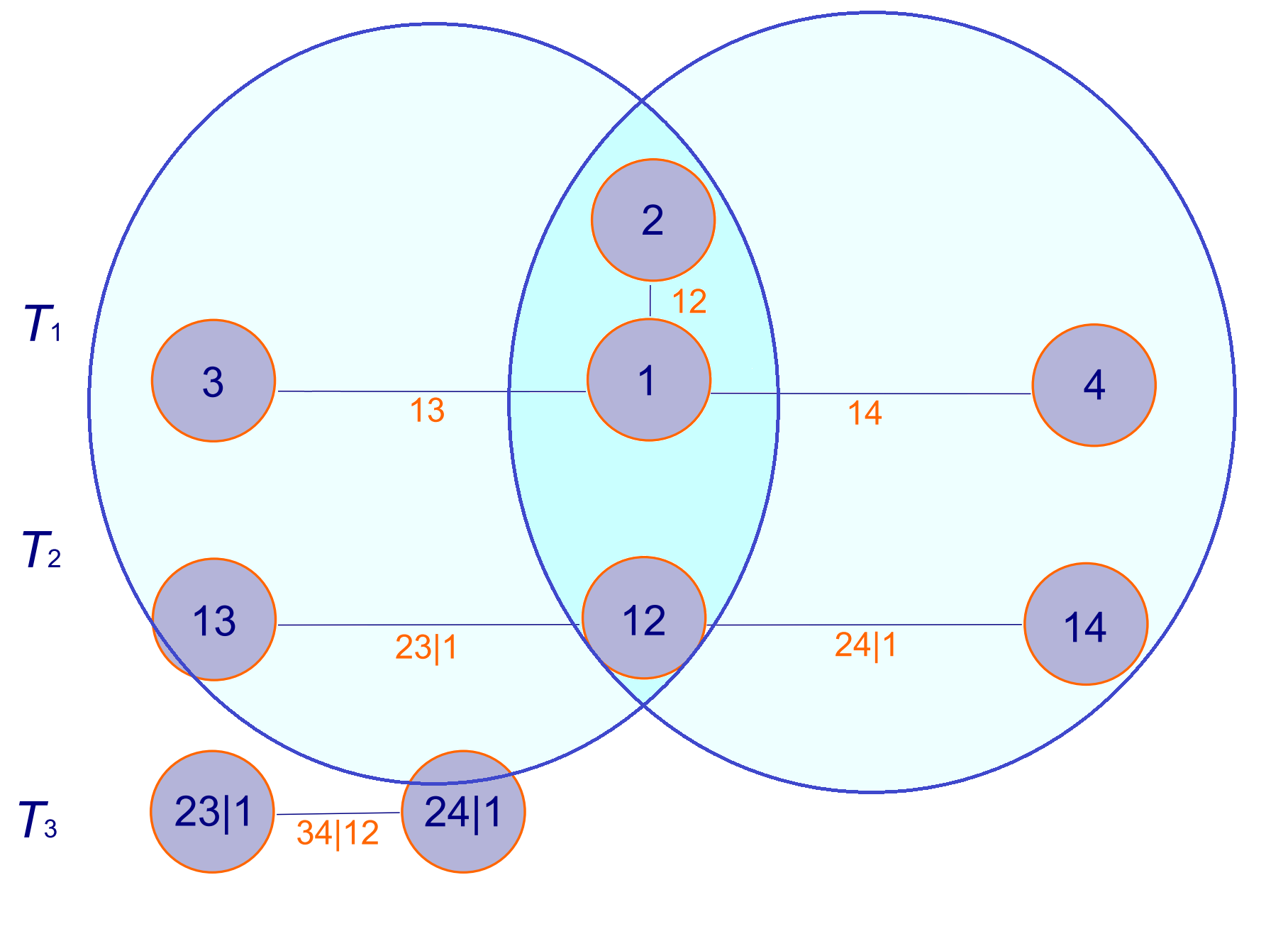}
\caption{Trivariate subsets of the four-dimensional $C$-vine.}
\label{fig:4CVineSeparated}
\end{center}
\end{figure}

\subsubsection{Calculating $\eta_\mathcal{D}$ for $d$-dimensional $D$-vines and $C$-vines with inverted logistic components}\label{subsubsec:etaDlargevines}
As for the trivariate vine copula examples with inverted logistic pair copula components, numerical results suggest that the intersection of the set $G=\{\bm{x}\in\mathbb{R}^d:g(\bm{x})=1\}$ and $[\eta_\mathcal{D},\infty)^d$ for these $D$-vines and $C$-vines occurs when $x_1=\ldots=x_d$. As before, this suggests that $\eta_\mathcal{D}=g(1,\ldots,1)^{-1}$ in this case.

\begin{figure}[ht]
\begin{center}
\includegraphics[width=0.6\textwidth]{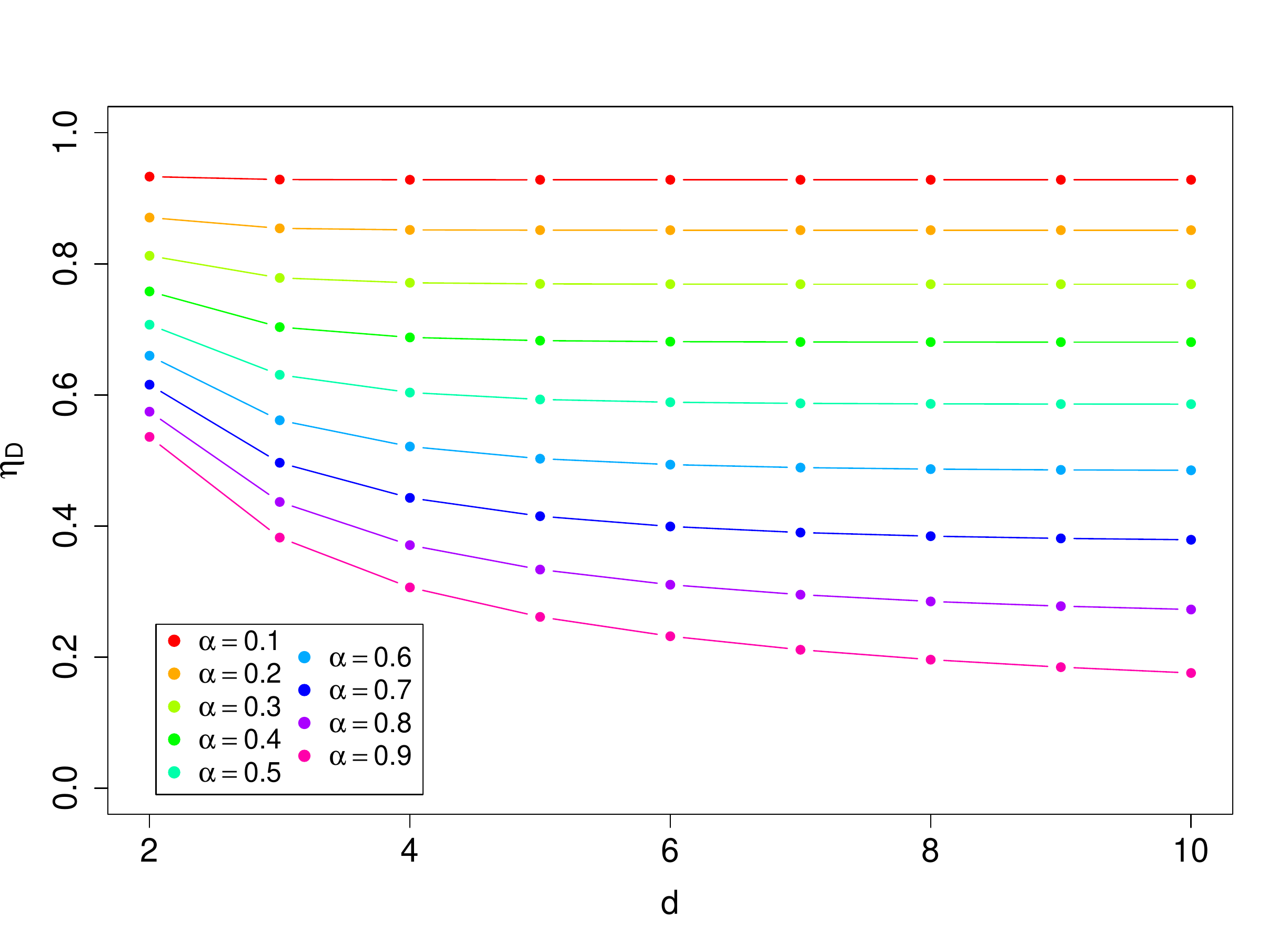}
\caption{Values of $\eta_\mathcal{D}$ for $d\in\{2,\ldots,d\}$ for a $d$-dimensional $D$-vine or $C$-vine constructed from inverted logistic pair copulas all having equal dependence parameter $\alpha\in\{0.1,\ldots,0.9\}$.}
\label{fig:etaDs_multivariateVine}
\end{center}
\end{figure}

Due to the nested structure of the gauge functions, the value of $\eta_\mathcal{D}$ can be written in terms of the values of $\eta_\mathcal{C}$ for various sub-vines of the copula, and the exponent measure corresponding to tree $T_{d-1}$ of the vine. In particular, for $D$-vines, we have
\begin{align}
\eta_\mathcal{D} = \left\{\eta_{\mathcal{D}\backslash\{1,d\}}^{-1} + V^{\{1,d|\mathcal{D}\backslash\{1,d\}\}}\left(\frac{1}{\eta_{\mathcal{D}\backslash\{d\}}^{-1} - \eta_{\mathcal{D}\backslash\{1,d\}}^{-1}},\frac{1}{\eta_{\mathcal{D}\backslash\{1\}}^{-1} - \eta_{\mathcal{D}\backslash\{1,d\}}^{-1}}\right)\right\}^{-1}
\label{eqn:dvineEta}
\end{align}
and for $C$-vines,
\[
\eta_\mathcal{D} = \left\{\eta_{\mathcal{D}\backslash\{d-1,d\}}^{-1} + V^{\{d-1,d|\mathcal{D}\backslash\{d-1,d\}\}}\left(\frac{1}{\eta_{\mathcal{D}\backslash\{d\}}^{-1} - \eta_{\mathcal{D}\backslash\{d-1,d\}}^{-1}},\frac{1}{\eta_{\mathcal{D}\backslash\{(d-1)\}}^{-1} - \eta_{\mathcal{D}\backslash\{d-1,d\}}^{-1}}\right)\right\}^{-1}.
\]
Setting $\eta_\mathcal{C}=1$ for $|\mathcal{C}|=1$, we now have an iterative method for calculating the values of $\eta_\mathcal{D}$ for these classes of model for $d\geq 3$ dimensions.

As an example, we consider the case where all the pair copulas of the vine are inverted logistic with the same dependence parameter $\alpha\in(0,1)$. In this case, the known value of $\eta_{\{1,2\}}$ for the bivariate copula is $2^{-\alpha}$. We can therefore use our iterative formulas to calculate $\eta_\mathcal{D}$ for higher dimensional vine copulas. Since the exponent is homogeneous of order $-1$, the expression for $\eta_\mathcal{D}$ in~\eqref{eqn:dvineEta} in this case simplifies to
\begin{align}
	\eta_\mathcal{D} = \left\{\eta^{-1}_{\mathcal{D}\backslash\{1,d\}} + 2^\alpha\left(\eta^{-1}_{\mathcal{D}\backslash\{d\}}-\eta^{-1}_{\mathcal{D}\backslash\{1,d\}}\right)\right\}^{-1},
\label{eqn:DvineEtaD.1}
\end{align}
and we can use the iterative method to derive the exact value of $\eta_\mathcal{D}$ for any $d$-dimensional $D$-vine copula. For this example, we can extend the results to higher $d$-dimensional vine copulas, yielding, for $d\geq 3$,
\begin{align}
	\eta_\mathcal{D} &=\begin{cases}
       	\left\{1 + 2^\alpha\sum_{k=1}^{(d-1)/2} \left(2^\alpha-1\right)^{2(k-1)+1}\right\}^{-1}, & \text{for $d$ odd},\\
        \left\{ 2^\alpha\sum_{k=1}^{d/2}\left(2^\alpha-1\right)^{2(k-1)} \right\}^{-1},& \text{for $d$ even,}
        \end{cases}~=~\begin{cases}
       	\left[1 + \frac{2^\alpha-1}{2-2^\alpha}\left\{1-(2^\alpha-1)^{d-1}\right\}\right]^{-1}, & \text{for $d$ odd},\\
        \left[\frac{1}{2-2^\alpha}\left\{1-(2^\alpha-1)^d\right\}\right]^{-1},& \text{for $d$ even},
        \label{eqn:DvineEtaD}
	\end{cases}
\end{align}
which can be shown to be decreasing in $d$. We prove result~\eqref{eqn:DvineEtaD} by induction in Section~\ref{app:DvineEtaD_proof} of the Supplementary Material. We note that when the pair copulas, and therefore the corresponding exponent measures, are all taken to be identical, the value of $\eta_\mathcal{D}$ is the same for the $D$-vines and $C$-vines of the same dimension. These values are demonstrated in Fig.~\ref{fig:etaDs_multivariateVine} for $\alpha\in\{0.1,\ldots,0.9\}$ and $d\in\{2,\ldots,10\}$, where we have $\eta_\mathcal{D}<1$ in all cases, corresponding to asymptotic independence. Complete independence in the $d$-dimensional vine copula corresponds to $\eta_\mathcal{D}=1/d$. We see from Fig.~\ref{fig:etaDs_multivariateVine} that for $\alpha=0.9$, we approach this case, while for $\alpha=0.1$, the values of $\eta_\mathcal{D}$ are close to 1, corresponding to strong residual dependence. These models are therefore able to capture a range of sub-asymptotic dependence strengths in the asymptotic independence case.

\section{Trivariate vine copulas with inverted extreme value and extreme value copula components}\label{sec:vinesIEVEV}
\subsection{Overview}
We have so far focused on the tail dependence properties of vine copulas with inverted extreme value pair copula components. Now, we investigate these same properties for trivariate vine copulas where the components are either extreme value or inverted extreme value copulas. We consider five such cases, which along with the results in Section~\ref{subsec:trivariateIEV} cover the range of possible scenarios. In Section~\ref{subsec:gauge2}, the two copulas in tree $T_1$ of Fig.~\ref{fig:triVine} belong to the inverted extreme value class, and there is an extreme value copula in tree $T_2$; tree $T_1$ has one extreme value and one inverted extreme value copula in Sections~\ref{subsec:gauge3}~and~\ref{subsec:gauge4}, with the copula in tree $T_2$ being either inverted extreme value or extreme value, respectively; and in Sections~\ref{subsec:gauge5}~and~\ref{subsec:gauge6}, both copulas in tree $T_1$ are from the extreme value family with the copula in tree $T_2$ being inverted extreme value and extreme value in the respective sections. This section will therefore consist of a series of examples, and the gauge functions resulting from these vine structures generally have a complicated form, with the corresponding sets $G$ exhibiting interesting shapes including non-convexity and non-smoothness. This differs from other well-known examples such as the multivariate Gaussian distribution. The gauge function calculations are provided in Sections~\ref{SM:evProperties}~to~\ref{SMsec:IEVEVgauges} of the Supplementary Material for each of these cases, with the extreme value components satisfying conditions analogous to~\eqref{eqn:EVassumption}. Specifically, let $h^{\{12\}}(w)$, $h^{\{23\}}(w)$, $h^{\{13|2\}}(w)$ denote the spectral density for each pair copula component. We assume that each of these densities has $h^{\{\cdot\}}(w)\sim c_1^{\{\cdot\}}(1-w)^{s_1^{\{\cdot\}}}$ as $w\nearrow 1$ and $h^{\{\cdot\}}(w)\sim c_2^{\{\cdot\}}w^{s_2^{\{\cdot\}}}$ as $w\searrow 0$, for some $c_1^{\{\cdot\}},c_2^{\{\cdot\}}\in\mathbb{R}$ and $s_1^{\{\cdot\}},s_2^{\{\cdot\}}>-1$.

Our results are summarised in Section~\ref{subsec:trivariateGauges}, where we also investigate the corresponding values of $\eta_{\{1,2,3\}}$ and $\eta_{\{1,3\}}$ for logistic and inverted logistic examples, with exponent measure~\eqref{eqn:Vlogistic}. \EDIT{In some subsections, this is achieved by obtaining results for more general gauge functions. In other cases this is not possible, and we focus only on the (inverted) logistic examples, but suggest that similar strategies could be used for other cases. Whether the copula is logistic or inverted logistic, we denote the parameters associated with exponent measures of copulas $c_{12}$, $c_{23}$ and $c_{13|2}$ by $\alpha,\beta,\gamma\in(0,1)$, respectively. We note that in the logistic case, we have $s^{\{12\}}_1=s^{\{12\}}_2=1/\alpha - 2$; $s^{\{23\}}_1=s^{\{23\}}_2=1/\beta - 2$ and $s^{\{13|2\}}_1=s^{\{13|2\}}_2=1/\gamma - 2$.}

\subsection{Gauge functions for trivariate vines with extreme value and inverted extreme value components}\label{subsec:trivariateGauges}
\subsubsection{Inverted extreme value copulas in $T_1$; extreme value copula in $T_2$}\label{subsec:gauge2}
The calculations in the Supplementary Material demonstrate that the gauge function is
\begin{align*}
	g(\bm{x}) = &\left(2+s_{\text{m}}^{\{13\mid 2\}}\right)\max\left\{V^{\{12\}}\left(x_1^{-1},x_2^{-1}\right),V^{\{23\}}\left(x_2^{-1},x_3^{-1}\right)\right\} - \left(1+s_{\text{m}}^{\{13\mid 2\}}\right)\min\left\{V^{\{12\}}\left(x_1^{-1},x_2^{-1}\right),V^{\{23\}}\left(x_2^{-1},x_3^{-1}\right)\right\},
\label{eqn:IEVIEVEV}
\end{align*}
with $\min(x_1,x_2,x_3)\geq 0$, and
\[
s^{\{13\mid 2\}}_{\text{m}} = s^{\{13\mid 2\}}_1\mathbbm{1}_{\left\{V^{\{12\}}\left(x_1^{-1},~x_2^{-1}\right) \geq V^{\{23\}}\left(x_2^{-1},~x_3^{-1}\right)\right\}} + s^{\{13\mid 2\}}_2\mathbbm{1}_{\left\{V^{\{12\}}\left(x_1^{-1},~x_2^{-1}\right) < V^{\{23\}}\left(x_2^{-1},~x_3^{-1}\right)\right\}}~\EDIT{>-1}.
\]
\EDIT{To calculate $\eta_{\{1,2,3\}}$, we must here consider two separate cases. First, we assume that $V^{\{12\}}\left(x_1^{-1},x_2^{-1}\right)\geq V^{\{23\}}\left(x_2^{-1},x_3^{-1}\right)$, so the gauge function simplifies to 
\[
	g(\bm{x}) = \left(2+s_{\text{m}}^{\{13\mid 2\}}\right) V^{\{12\}}\left(x_1^{-1},x_2^{-1}\right) - \left(1+s_{\text{m}}^{\{13\mid 2\}}\right)V^{\{23\}}\left(x_2^{-1},x_3^{-1}\right).
\]
Since $s^{\{13\mid 2\}}_{\text{m}}>-1$, $g(\bm{x})$ is non-decreasing in $x_1$ and we can set $x_1=1$ to find the solution of~\eqref{eqn:NWetaGauge}. We therefore need to minimise
\[
	g(1,x_2,x_3) = \left(2+s_{\text{m}}^{\{13\mid 2\}}\right) V^{\{12\}}\left(1,x_2^{-1}\right) - \left(1+s_{\text{m}}^{\{13\mid 2\}}\right)V^{\{23\}}\left(x_2^{-1},x_3^{-1}\right),
\]
such that $V^{\{12\}}\left(1,x_2^{-1}\right)\geq V^{\{23\}}\left(x_2^{-1},x_3^{-1}\right)$. Now, the function $g(1,x_2,x_3)$ is non-increasing in $x_3$, which should therefore be taken to be as large as possible. Since $V^{\{23\}}\left(x_2^{-1},x_3^{-1}\right)$ is non-decreasing in $x_3$, this function should also be as large as possible, i.e., the largest value of $x_3$ occurs when $V^{\{23\}}\left(x_2^{-1},x_3^{-1}\right)=V^{\{12\}}\left(1,x_2^{-1}\right)$, and the gauge function becomes
\[
	g(1,x_2,x_3) = \left(2+s_{\text{m}}^{\{13\mid 2\}}\right) V^{\{12\}}\left(1,x_2^{-1}\right) - \left(1+s_{\text{m}}^{\{13\mid 2\}}\right)V^{\{12\}}\left(1,x_2^{-1}\right)=V^{\{12\}}\left(1,x_2^{-1}\right).
\]
Again, this is non-decreasing in $x_2$, so the minimum occurs when $x_2=1$, i.e., $\min_{\bm{x}:\min(\bm{x})=1}g(\bm{x})=V^{\{12\}}\left(1,1\right)$.
For the case where $V^{\{12\}}\left(x_1^{-1},x_2^{-1}\right)\leq V^{\{23\}}\left(x_2^{-1},x_3^{-1}\right)$, by a similar argument, $\min_{\bm{x}:\min(\bm{x})=1}g(\bm{x})=V^{\{23\}}\left(1,1\right)$. In summary, we have
\[
\eta_{\{1,2,3\}} = \left\{\min\limits_{\bm{x}:\min(\bm{x})=1}g(\bm{x})\right\}^{-1}=\left[\max\left\{V^{\{12\}}\left(1,1\right),V^{\{23\}}\left(1,1\right)\right\}\right]^{-1}=\min\left[\left\{V^{\{12\}}\left(1,1\right)\right\}^{-1},\left\{V^{\{23\}}\left(1,1\right)\right\}^{-1}\right].
\]
For the value of $\eta_{\{1,3\}}$, we observe that 
\begin{align*}
g(1,0,1)&=\left(2+s_{\text{m}}^{\{13\mid 2\}}\right)\max\left\{V^{\{12\}}\left(1,\infty\right),V^{\{23\}}\left(\infty,1\right)\right\} - \left(1+s_{\text{m}}^{\{13\mid 2\}}\right)\min\left\{V^{\{12\}}\left(1,\infty\right),V^{\{23\}}\left(\infty,1\right)\right\}\\
&=\left(2+s_{\text{m}}^{\{13\mid 2\}}\right)- \left(1+s_{\text{m}}^{\{13\mid 2\}}\right)=1.
\end{align*}
As discussed in Section~\ref{subsec:lowerDimEta}, this is the smallest possible value of $\min_{x_1,x_3\geq 1,x_2\geq 0}g(\bm{x})$, and therefore must be our required minimum. We therefore have $\eta_{\{1,3\}}=1$.}

\EDIT{For an example with logistic and inverted logistic components, the gauge function is
\begin{align*}
g(\bm{x}) = &\left(1/\gamma\right)\max\left\{\left(x_1^{1/\alpha}+x_2^{1/\alpha}\right)^\alpha , \left(x_2^{1/\beta}+x_3^{1/\beta}\right)^\beta\right\}-\left(1/\gamma-1\right)\min\left\{\left(x_1^{1/\alpha}+x_2^{1/\alpha}\right)^\alpha , \left(x_2^{1/\beta}+x_3^{1/\beta}\right)^\beta\right\},
\end{align*}
with $\eta_{\{1,2,3\}}=\min\left(1/2^\alpha,1/2^\beta\right)$ and $\eta_{\{1,3\}}=1$.}

\EDIT{In Fig.~\ref{fig:ilogiloglogGauge}, we demonstrate the set $G=\{\bm{x}\in\mathbb{R}^3:g(\bm{x})=1\}$ for this example, with $\alpha\in\{0.25,0.5,0.75\}$, $\beta=0.25$ and $\gamma=0.5$, where the surface corresponding to the set $G$ turns out to be non-convex. The plots in Fig.~\ref{fig:ilogiloglogGauge} support our analytical calculations. The intersection of $G$ and $[\eta_{\{1,2,3\}},1)^3$ occurs at $x_1=x_2=x_3$ in the first panel with $\alpha=\beta=0.25$, and $x_3\geq x_1=x_2$ in the remaining two panels, where $\alpha>\beta$. The gauge function for the pair of variables $(X_1,X_3)$ is demonstrated in Fig.~\ref{fig:ilogiloglogGauge13}. This plot supports that $\eta_{\{1,3\}}=1$, and we note the non-convex shape of $G_{\{1,3\}}$.}

\begin{figure}[t]
\begin{center}
	\includegraphics[width=\textwidth]{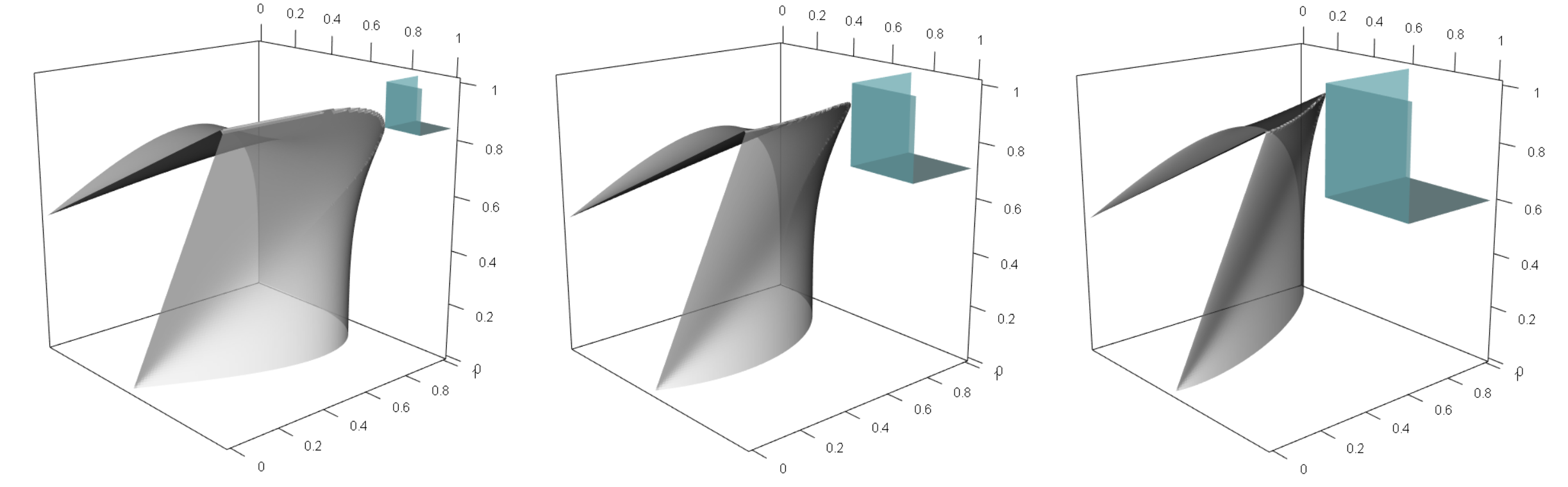}
	\caption{The set $G=\{\bm{x}\in\mathbb{R}^3:g(\bm{x})=1\}$ for a trivariate vine with inverted logistic copulas in $T_1$ and a logistic copula in $T_2$ (grey) and the boundary of the set $[\eta_{\{1,2,3\}},1)^3$ (blue): $\alpha=0.25$ (left), $\alpha=0.5$ (centre), $\alpha=0.75$ (right); $\beta=0.25$ and $\gamma=0.5$.}
	\label{fig:ilogiloglogGauge}
\end{center}
\end{figure}

%\EDIT{For the value of $\eta_{\{1,3\}}$, we note that $g(1,0,1)=(1/\gamma)\max(1,1)-(1/\gamma-1)\min(1,1)=1$. As discussed in Section~\ref{subsec:lowerDimEta}, this is the smallest possible value of $\min_{x_1,x_3\geq 1,x_2\geq 0}g(\bm{x})$, and therefore must be our required minimum. We therefore have $\eta_{\{1,3\}}=1$. This is supported by Fig.~\ref{fig:ilogiloglogGauge13}, where we note the non-convex shape of $G_{\{1,3\}}$.}

%This yields
%\[
% \eta_{\{1,3\}} = \left[ \frac{1}{\gamma}\max\left\{\left(1+v^{1/\alpha}\right)^\alpha,\left(1+v^{1/\beta}\right)^\beta\right\} -  \left(\frac{1}{\gamma}-1\right)\min\left\{\left(1+v^{1/\alpha}\right)^\alpha,\left(1+v^{1/\beta}\right)^\beta\right\} \right]^{-1},
%\]
%with $v$ satisfying $d g(1,v,1)/d v=0$. 

\begin{figure}[t]
\begin{center}
\includegraphics[width=0.35\textwidth]{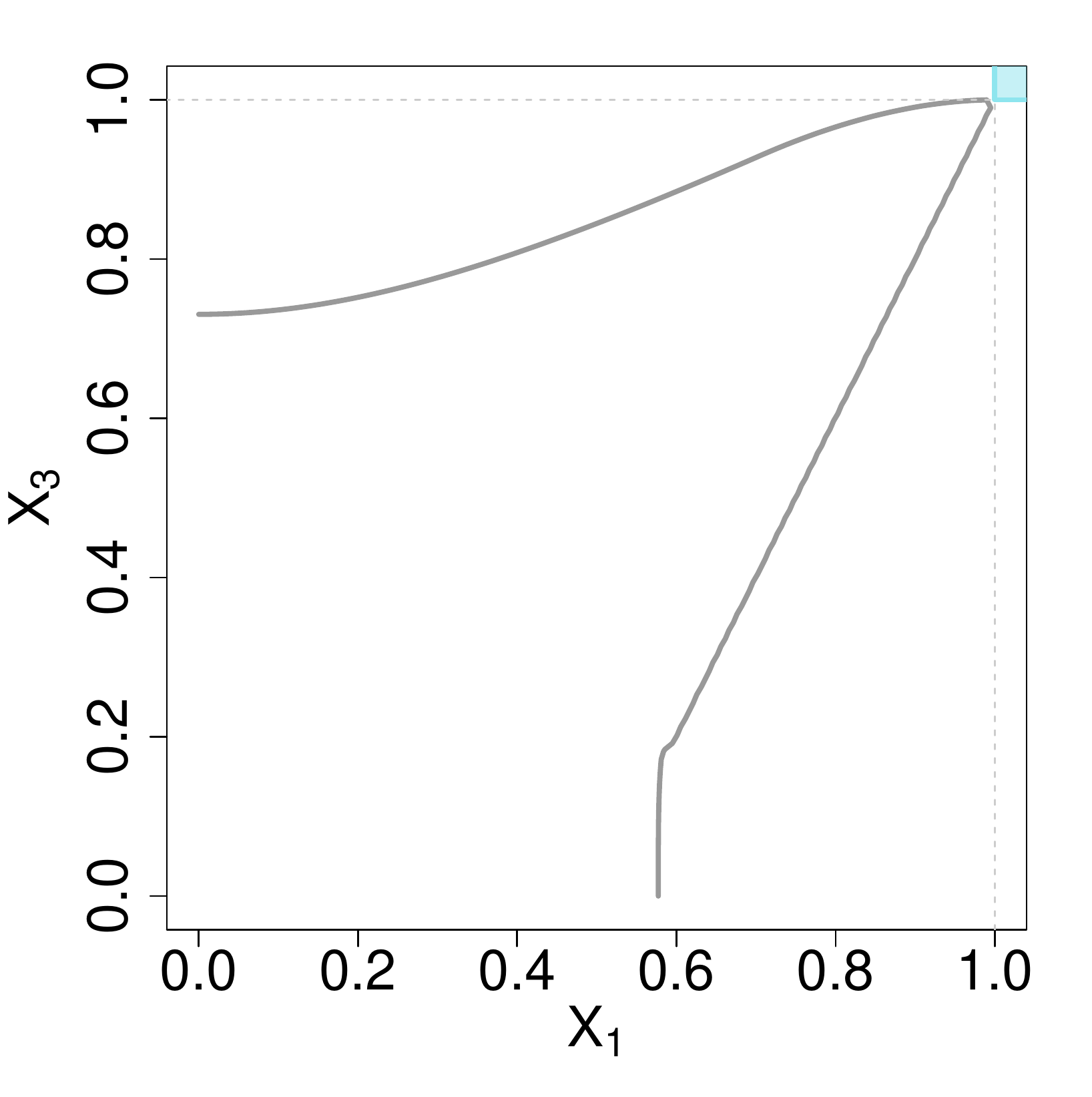}
\caption{The sets $G_{\{1,3\}}$ (grey) and $[\eta_{\{1,3\}},1)^2$ (blue) for $(\alpha,\beta,\gamma)=(0.5,0.25,0.5)$.}
\label{fig:ilogiloglogGauge13}
\end{center}
\end{figure}

\subsubsection{Extreme value and inverted extreme value copulas in $T_1$; inverted extreme value copula in $T_2$}\label{subsec:gauge3}
From the calculations in the Supplementary Material, the gauge function for this model is
\begin{align*}
g(\bm{x})=
	\begin{cases}
       \left(2+s_1^{\{13\mid 2\}}\right)\left(1+s_2^{\{12\}}\right)\left(x_2-x_1\right) + V^{\{23\}}\left(x_2^{-1},x_3^{-1}\right), & 0\leq x_1\leq x_2,\\
        x_2 + V^{\{13\mid 2\}} \left[\left\{\left(x_1-x_2\right)\left(2+s_1^{\{12\}}\right)\right\}^{-1}, \left\{V^{\{23\}}\left(x_2^{-1},x_3^{-1}\right)-x_2\right\}^{-1}\right] ,& 0\leq x_2 <x_1.
	\end{cases}
	\label{eqn:EVIEVIEV}
\end{align*}
\EDIT{To find $\eta_{\{1,2,3\}}$, there are two cases to consider. If $x_1\leq x_2$, it is clear that the function $g(\bm{x})$ is decreasing in $x_1$, which should therefore be taken to be as large as possible by fixing $x_1=x_2$. The gauge function then simplifies to
\[
g(x_2,x_2,x_3)=V^{\{23\}}\left(x_2^{-1},x_3^{-1}\right),
\]
which is non-decreasing in both $x_2$ and $x_3$, so we find the minimum by setting $x_2=x_3=1$, yielding $V^{\{23\}}(1,1)$. Similarly, for $x_2\leq x_1$, the gauge function is non-decreasing in $x_1$, and we can again fix $x_1=x_2$ with the resulting gauge function being
\[
g(x_2,x_2,x_3) = x_2 + V^{\{13\mid 2\}} \left[0^{-1}, \left\{V^{\{23\}}\left(x_2^{-1},x_3^{-1}\right)-x_2\right\}^{-1}\right] = V^{\{23\}}\left(x_2^{-1},x_3^{-1}\right).
\]
Again, this is non-decreasing in $x_2$ and $x_3$, so the minimum is given by $V^{\{23\}}\left(1,1\right)$. Hence, 
\[\eta_{\{1,2,3\}}=g(1,1,1)^{-1}=\left\{V^{\{23\}}\left(1,1\right)\right\}^{-1}.\]}

For the case with logistic and inverted logistic components, the gauge function is
\begin{align*}
g(\bm{x})=
	\begin{cases}
       \left(1/\gamma\right)\left(1/\alpha-1\right)\left(x_2-x_1\right) + \left(x_2^{1/\beta}+x_3^{1/\beta}\right)^\beta, & 0\leq x_1\leq x_2,\\
        x_2 + \left[\left\{\left(x_1-x_2\right)/\alpha\right\}^{1/\gamma} +  \left\{\left(x_2^{1/\beta}+x_3^{1/\beta}\right)^\beta-x_2\right\}^{1/\gamma}\right]^\gamma
 ,& 0\leq x_2<x_1.
	\end{cases}
\end{align*}
\EDIT{An example of this gauge function is shown in Fig.~\ref{fig:g3}, and we have  $\eta_{\{1,2,3\}}=g(1,1,1)^{-1}=1/2^\beta$.} 
%\EDIT{To find $\eta_{\{1,2,3\}}$, there are two cases to consider. If $x_1\leq x_2$, it is clear that the function $g(\bm{x})$ is increasing in $x_2$ and $x_3$, so the minimum required in~\eqref{eqn:NWetaGauge} occurs when $x_2=x_3=1$. But since $x_1\leq x_2$, we must also have $x_1=1$. Similarly, for $x_2\leq x_1$, the gauge function increases with both $x_1$ and $x_3$ so the minimum occurs at $x_1=x_3=1$, and since $x_2\leq x_1$, we must also have $x_2=1$. Hence, $\eta_{\{1,2,3\}}=g(1,1,1)^{-1}=1/2^\beta$.}

\EDIT{We now consider the bivariate coefficient of tail dependence between $X_1$ and $X_3$. Since we have already shown that $\min_{\bm{x}:\min(\bm{x})\geq 1}g(\bm{x})=g(1,1,1)$, we must have $\min_{\bm{x}:x_1,x_3\geq 1,x_2\geq 0}g(\bm{x})\leq g(1,1,1)$, and can focus only on the case where $x_2\leq x_1$. Here, the gauge function is increasing in $x_1$ and $x_3$, so we fix $x_1=x_3=1$, and should minimise the function
\[
g(1,v,1)=v + \left[\left\{\left(1-v\right)/\alpha\right\}^{1/\gamma} +  \left\{\left(1+v^{1/\beta}\right)^\beta-v\right\}^{1/\gamma}\right]^\gamma,
\]
for $v\geq 0$. We therefore find that $\eta_{\{1,3\}} = g(1,v,1)^{-1}$, with $v$ such that 
{\small \begin{align*}
1 &+ \left[\left\{(1-v)/\alpha\right\}^{1/\gamma} +\left\{\left(1+v^{1/\beta}\right)^\beta-v\right\}^{1/\gamma} \right]^{\gamma-1}\cdot\left[-\alpha^{-1/\gamma}(1-v)^{-1+1/\gamma} + \left\{\left(1+v^{1/\beta}\right)^\beta-v\right\}^{-1+1/\gamma}\left\{\left(1+v^{-1/\beta}\right)^{\beta-1}-1\right\}\right]=0.
\end{align*}}
Following an argument almost identical to the one presented for the calculation of $\eta_{\{1,3\}}$ in Section~\ref{subsec:eta13IEV}, this equation has a unique solution, with $v\in(0,1)$.}

\begin{figure}[t]
\centering
\begin{subfigure}{.375\textwidth}
  \centering
  \includegraphics[width=.8\linewidth]{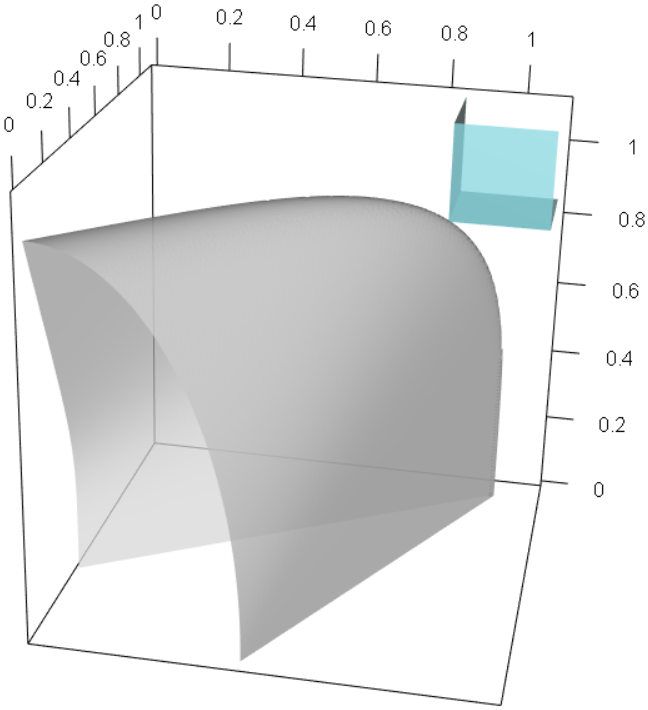}
  \caption{}
  \label{fig:g3}
\end{subfigure}%
\begin{subfigure}{.375\textwidth}
  \centering
  \includegraphics[width=.8\linewidth]{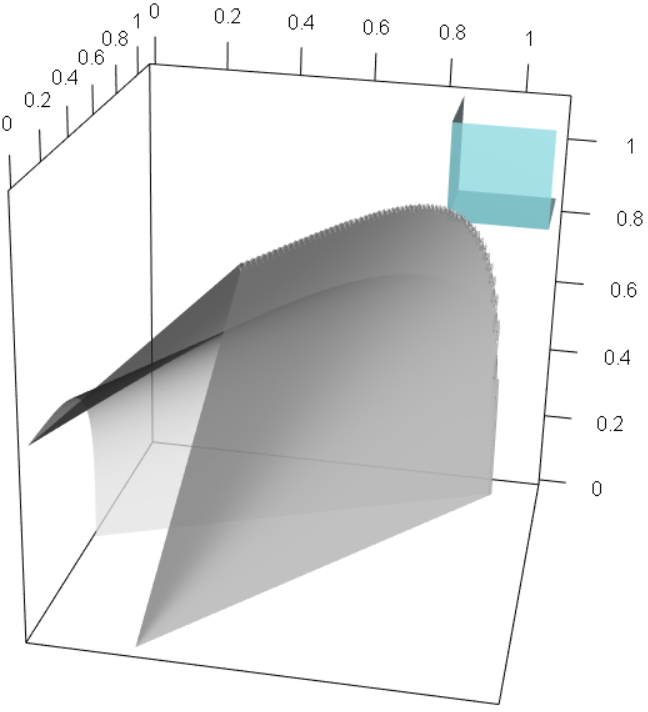}
  \caption{}
  \label{fig:g4}
\end{subfigure}\\
\begin{subfigure}{.375\textwidth}
  \centering
  \includegraphics[width=.8\linewidth]{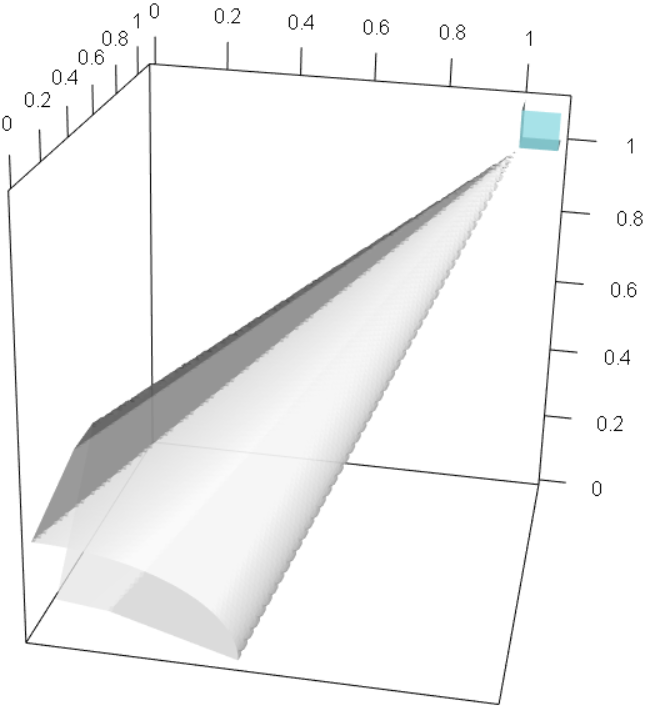}
  \caption{}
  \label{fig:g5}
\end{subfigure}%
\begin{subfigure}{.375\textwidth}
  \centering
  \includegraphics[width=.8\linewidth]{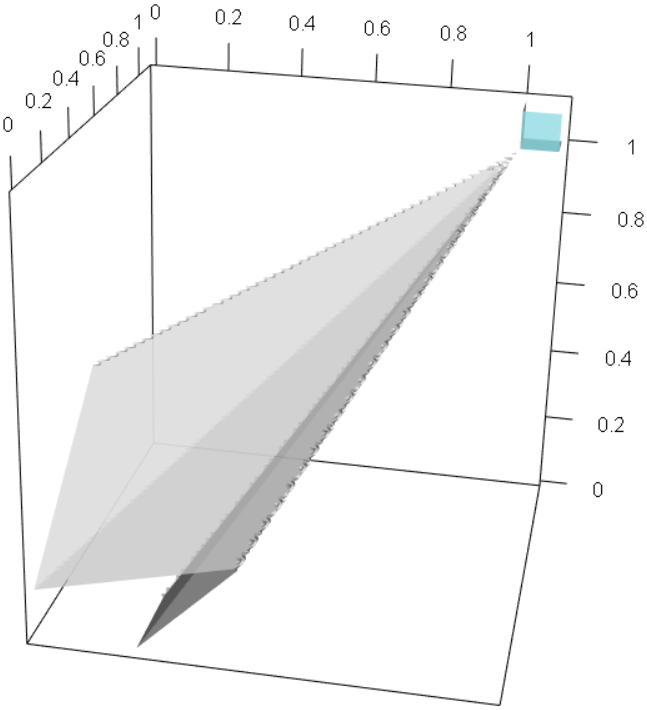}
  \caption{}
  \label{fig:g6}
\end{subfigure}
\caption{The sets $G=\{\bm{x}\in\mathbb{R}^3:g(\bm{x})=1\}$ for the trivariate vine copulas described in Sections~\ref{subsec:gauge3}--\ref{subsec:gauge6} (grey) with the boundary of the set $[\eta_{\{1,2,3\}},1.1)^3$ (blue): $\alpha=0.5, \beta=0.25, \gamma=0.5$.}
\label{fig:otherGauges}
\end{figure}

\subsubsection{Extreme value and inverted extreme value copulas in $T_1$; extreme value copula in $T_2$}\label{subsec:gauge4}
The gauge function for this copula is
\begin{align*}
g(\bm{x})=
	\begin{cases}
       x_2 + \left(1+s_2^{\{12\}}\right)\left(x_2-x_1\right) + \left(2+s_2^{\{13\mid 2\}}\right)\left\{ V^{\{23\}}\left(x_2^{-1},x_3^{-1}\right) -x_2\right\} , & 0\leq x_1\leq x_2,\\
       x_2 + \left( 2+s_{\text{m}}^{\{13\mid 2\}} \right)\max\left\{ \left(2+s_1^{\{12\}}\right)\left(x_1-x_2\right), V^{\{23\}}\left(x_2^{-1},x_3^{-1}\right) -x_2 \right\}\\
       ~~~-\left( 1+s_{\text{m}}^{\{13\mid 2\}} \right)\min\left\{ \left(2+s_1^{\{12\}}\right)\left(x_1-x_2\right), V^{\{23\}}\left(x_2^{-1},x_3^{-1}\right) -x_2 \right\}
 ,& 0\leq x_2<x_1,
	\end{cases}
	%\label{eqn:EVIEVEV}
\end{align*}
with 
\begin{align*}
s^{\{13\mid 2\}}_{\text{m}} &= s^{\{13\mid 2\}}_1\mathbbm{1}_{\left\{\left(2+s_1^{\{12\}}\right)\left(x_1-x_2\right) \geq V^{\{23\}}\left(x_2^{-1},~x_3^{-1}\right) -x_2\right\}}+ s^{\{13\mid 2\}}_2\mathbbm{1}_{\left\{\left(2+s_1^{\{12\}}\right)\left(x_1-x_2\right) < V^{\{23\}}\left(x_2^{-1},~x_3^{-1}\right) -x_2\right\}},
\end{align*}
so that for the case with logistic and inverted logistic components, we have 
\begin{align*}
g(\bm{x})=
	\begin{cases}
       \left(1/\alpha\right)x_2 - \left(1/\alpha-1\right)x_1 + \left(1/\gamma\right)\left\{  \left(x_2^{1/\beta}+x_3^{1/\beta}\right)^\beta -x_2\right\} , & 0\leq x_1\leq x_2,\\
       x_2 + \left(1/\gamma\right)\max\left\{ \left(x_1-x_2\right)/\alpha,  \left(x_2^{1/\beta}+x_3^{1/\beta}\right)^\beta -x_2 \right\}\\
       \hspace{1cm}~~~-\left(1/\gamma-1\right)\min\left\{\left(x_1-x_2\right)/\alpha,  \left(x_2^{1/\beta}+x_3^{1/\beta}\right)^\beta -x_2 \right\}
 ,& 0\leq x_2<x_1.
	\end{cases}
\end{align*}
\EDIT{As in the previous example, we consider the two cases $x_1\leq x_2$ and $x_2\leq x_1$ separately. In the former, the gauge function increases with $x_3$, so that the minimum required to obtain $\eta_{\{1,2,3\}}$ occurs when $x_3=1$. On the other hand, the function is decreasing with respect to $x_1$, which must therefore take its largest possible value, i.e., we can fix $x_1=x_2$. This implies we should focus on minimising
\[
h(v) = g(v,v,1)= v + (1/\gamma)\left\{\left(1+v^{1/\beta}\right)^\beta-v\right\},
\] 
under the constraint that $v\geq 1$. We have
\[
h'(v) = 1 + (1/\gamma)\left\{v^{1/\beta-1}\left(1+v^{1/\beta}\right)^{\beta-1}-1\right\}= 1 + (1/\gamma)\left\{\left(1+v^{-1/\beta}\right)^{\beta-1}-1\right\}.
\]
If we solve the equation $h'(v_0)=0$, we obtain the root $v_0 = \left\{(1-\gamma)^{-1/(1-\beta)}-1\right\}^{-\beta}$, and note that $v_0> 1$ if and only if $\gamma < 1-2^{\beta-1}$. In this case, the minimum value of $h(v)$ is given by
\[
h(v_0) = \left(\frac{1-\gamma}{\gamma}\right)\left\{(1-\gamma)^{-1/(1-\beta)}-1\right\}^{1-\beta}.
\]
On the other hand, if $\gamma \geq 1-2^{\beta-1}$, $h'(v)>0$ for $v\geq 1$, so $h(v)$ is an increasing function of $v$, and the minimum occurs at $v=1$, i.e., $h(1)=1+\left(2^\beta-1\right)/\gamma$. In summary, if $x_1\leq x_2$, we have
\begin{align}
\min\limits_{\bm{x}:\min(\bm{x})\geq 1}g(\bm{x})=\begin{cases}
1+(2^\beta-1)/\gamma , &\text{if }\gamma\geq 1-2^{\beta-1},\\
\left(\frac{1-\gamma}{\gamma}\right)\left\{(1-\gamma)^{-1/(1-\beta)}-1\right\}^{1-\beta}, &\text{if }\gamma< 1-2^{\beta-1}.
\end{cases}
\label{eqn:min.cases}
\end{align}}

\EDIT{For $x_2\leq x_1$, the problem spilts into a further two cases. If $(x_1-x_2)/\alpha\leq \left(x_2^{1/\beta}+x_3^{1/\beta}\right)^\beta-x_2$, the gauge function becomes
\[
g(\bm{x})= x_2 + \left(1/\gamma\right)\left\{\left(x_2^{1/\beta}+x_3^{1/\beta}\right)^\beta -x_2 \right\} -\left(1/\gamma-1\right)(1/\alpha)\left(x_1-x_2\right).
\]
This is an increasing function of $x_3$, but a decreasing function of $x_1$. These two variables should therefore take their smallest and largest possible values, respectively. This occurs with equality at $(x_1-x_2)/\alpha=\left(x_2^{1/\beta}+x_3^{1/\beta}\right)^\beta-x_2$, and an equivalent argument holds for the $(x_1-x_2)/\alpha\geq \left(x_2^{1/\beta}+x_3^{1/\beta}\right)^\beta-x_2$ case. To obtain the required minimum, we can therefore focus on a simplified version of the gauge function, i.e., 
\[
g^*(x_2,x_3)= \left(x_2^{1/\beta}+x_3^{1/\beta}\right)^\beta,
\]
with $x_1=\alpha\left(x_2^{1/\beta}+x_3^{1/\beta}\right)^\beta+(1-\alpha)x_2>x_2$. The function $g^*(x_2,x_3)$ is increasing with respect to both $x_2$ and $x_3$, so we have 
\[
\min\limits_{\bm{x}:\min(\bm{x})=1}g(\bm{x})=g^*(1,1)=g\left[\left\{\alpha2^\beta+1-\alpha\right\},1,1\right]=2^\beta.
\]
We now have two candidates for the required minimum of the full gauge function; either $2^\beta$, or the form given in~\eqref{eqn:min.cases}. For $\gamma\geq 1-2^{\beta-1}$, it is straightforward to see that $1+(2^\beta-1)/\gamma\geq 2^\beta$. For $\gamma< 1-2^{\beta-1}$, we have $1-\gamma>2^{\beta-1}$, $\gamma^{-1}>\left(1-2^{\beta-1}\right)^{-1}$, and $(1-\gamma)^{-1/(1-\beta)}-1>1$, so
\[
\left(\frac{1-\gamma}{\gamma}\right)\left\{(1-\gamma)^{-1/(1-\beta)}-1\right\}^{1-\beta}>\left(\frac{1-\gamma}{\gamma}\right)>\frac{2^{\beta-1}}{1-2^{\beta-1}}>\frac{2^{\beta}}{2-2^{\beta}}>2^\beta.
\]
Hence, we find that the minimum in~\eqref{eqn:NWetaGauge} occurs when $x_1>x_2=x_3=1$, and $\eta_{\{1,2,3\}}=1/2^\beta$. This is supported by the plot in Fig.~\ref{fig:g4}, and suggests that the inverted logistic copula in tree $T_1$ particularly controls the level of asymptotic independence in the overall model.} 

\EDIT{We now consider $\eta_{\{1,3\}}$. Following the example in Section~\ref{subsec:gauge3}, we can focus on the case where $x_2\leq x_1$. Moreover, by a similar argument to the one used in the calculation of $\eta_{\{1,2,3\}}$, we only need to consider the case where $(x_1-x_2)/\alpha=\left(x_2^{1/\beta}+x_3^{1/\beta}\right)^\beta-x_2$, yielding
\[
	g(\bm{x}) = x_1/\alpha + (1-1/\alpha)x_2 = \left(x_2^{1/\beta}+x_3^{1/\beta}\right)^\beta. 
\]
Since these functions are increasing in $x_1$ and $x_3$, respectively, we set $x_1=x_3=1$. Hence, the minimum of the gauge function corresponds to $\min_{v\geq 0}(1+v^{1/\beta})^\beta$, with $v$ such that $(1-v)/\alpha= \left(v^{1/\beta}+1\right)^\beta-v$. We therefore have
\[
	\eta_{\{1,3\}} = \left(1+v^{1/\beta}\right)^{-\beta},~~~\text{with $v$ such that}~(1+v^{1/\beta})^\beta - (1-v)/\alpha - v=0.
\]
We note that if $h(v)=(1+v^{1/\beta})^\beta - (1-v)/\alpha - v$, we have $h'(v)=(1+v^{-1/\beta})^{\beta-1} + 1/\alpha - 1>0$ for $v\geq 0$. Moreover, $h(0)=1-1/\alpha<0$ and $h(1)=2^\beta-1>0$. Hence, $h(v)$ is an increasing function for $v\geq 0$, and the equation $h(v)=0$ has a unique root in the range $(0,1)$. This also implies that $\eta_{\{1,3\}} > 2^\beta=\eta_{\{1,2,3\}}$.}

\subsubsection{Extreme value copulas in $T_1$; inverted extreme value copula in $T_2$}\label{subsec:gauge5}
The gauge function in this case is
\begin{align*}
g(\bm{x})=
	\begin{cases}
       x_2 + \left( 2+s_{\text{m}}^{\{13\mid 2\}} \right)\max\left\{ \left(1+s_2^{\{12\}}\right)\left(x_2-x_1\right), \left(1+s_1^{\{23\}}\right)\left(x_2-x_3\right) \right\}\\
       ~~~-\left( 1+s_{\text{m}}^{\{13\mid 2\}} \right)\min\left\{ \left(1+s_2^{\{12\}}\right)\left(x_2-x_1\right), \left(1+s_1^{\{23\}}\right)\left(x_2-x_3\right) \right\}, &~~ \max(x_1,x_3)< x_2,\\
       x_2 + \left(2 + s_1^{\{13\mid 2\}}\right)\left(1+s_2^{\{12\}}\right)\left(x_2-x_1\right) + \left(2+s_2^{\{23\}}\right)\left(x_3-x_2\right),&~~ x_1 < x_2\leq x_3,\\
       x_2 + \left(2 + s_2^{\{13\mid 2\}}\right)\left(1+s_1^{\{23\}}\right)\left(x_2-x_3\right) + \left(2+s_1^{\{12\}}\right)\left(x_1-x_2\right),&~~ x_3< x_2\leq x_1,\\
       x_2 + V^{\{13\mid 2\}}\left[\left\{\left(2+s_1^{\{12\}}\right)\left(x_1-x_2\right)\right\}^{-1} , \left\{\left(2+s_2^{\{23\}}\right)\left(x_3-x_2\right)\right\}^{-1}\right],&~~ x_2 \leq \min(x_1,x_3),
	\end{cases}
\end{align*}
with $\min(x_1,x_2,x_3)\geq 0$ and
\begin{align*}
s^{\{13\mid 2\}}_{\text{m}} &= s^{\{13\mid 2\}}_1\mathbbm{1}_{\left\{\left(1+s_2^{\{12\}}\right)\left(x_2-x_1\right) \geq \left(1+s_1^{\{23\}}\right)\left(x_2-x_3\right) \right\}}+ s^{\{13\mid 2\}}_2\mathbbm{1}_{\left\{\left(1+s_2^{\{12\}}\right)\left(x_2-x_1\right) < \left(1+s_1^{\{23\}}\right)\left(x_2-x_3\right) \right\}}.
\end{align*}
\EDIT{For this gauge function, we observe that $g(1,1,1)=1$, which implies that $\eta_{\{1,2,3\}}=1$. Moreover, if $\eta_\mathcal{D}=1$, then $\eta_\mathcal{C}=1$ for any set $\mathcal{C}\subset \mathcal{D}$ with $|\mathcal{C}|\geq 2$. As such, we also find that $\eta_{\{1,3\}}=1$ in this case. This result agrees with the findings of \cite{Joeetal2010}, who show that a vine copula will have overall upper tail dependence if each of the copulas in tree $T_1$ also have this property and the copula in tree $T_2$ has support on $(0,1)^2$, as is the case here.}

For the case with logistic and inverted logistic components, 
\begin{align*}
g(\bm{x})=
	\begin{cases}
       x_2 + \left(1/\gamma \right)\max\left\{ \left(1/\alpha-1\right)\left(x_2-x_1\right), \left(1/\beta-1\right)\left(x_2-x_3\right) \right\}\\
       ~~~~~~~~+\left(1-1/\gamma\right)\min\left\{ \left(1/\alpha-1\right)\left(x_2-x_1\right), \left(1/\beta-1\right)\left(x_2-x_3\right) \right\}, ~~~& \max(x_1,x_3)< x_2,\\
       x_2 + \left(1/\gamma\right)\left(1/\alpha-1\right)\left(x_2-x_1\right) + \left(1/\beta\right)\left(x_3-x_2\right),& x_1 < x_2\leq x_3,\\
       x_2 + \left(1/\gamma\right)\left(1/\beta-1\right)\left(x_2-x_3\right) + \left(1/\alpha\right)\left(x_1-x_2\right),& x_3< x_2\leq x_1,\\
       x_2 + \left[\left\{\left(x_1-x_2\right)/\alpha\right\}^{1/\gamma} + \left\{\left(x_3-x_2\right)/\beta\right\}^{1/\gamma}\right]^\gamma,& x_2 \leq \min(x_1,x_3),
	\end{cases}
\end{align*}
with $\min(x_1,x_2,x_3)\geq 0$. \EDIT{This is demonstrated in Fig.~\ref{fig:g5}.}

\subsubsection{Extreme value copulas in $T_1$; extreme value copula in $T_2$}\label{subsec:gauge6}
The gauge function here has the form
\begin{align*}
g(\bm{x})=
	\begin{cases}
       x_2 + V^{\{13\mid 2\}}\left[\left\{\left(1+s_2^{\{12\}}\right)\left(x_2-x_1\right)\right\}^{-1} , \left\{\left(1+s_1^{\{23\}}\right)\left(x_2-x_3\right)\right\}^{-1}\right],&~~ \max(x_1,x_3)\leq x_2,\\
       x_2 + \left( 2+s_2^{\{13\mid 2\}} \right)\left(2+s_2^{\{23\}}\right)\left(x_3-x_2\right) +  \left(1+s_2^{\{12\}}\right)\left(x_2-x_1 \right),&~~ x_1 \leq x_2<x_3,\\
       x_2 + \left( 2+s_1^{\{13\mid 2\}} \right)\left(2+s_1^{\{12\}}\right)\left(x_1-x_2\right) +  \left(1+s_1^{\{23\}}\right)\left(x_2-x_3 \right),&~~ x_3\leq x_2<x_1,\\
       x_2 + \left( 2+s_{\text{m}}^{\{13\mid 2\}} \right)\max\left\{ \left(2+s_1^{\{12\}}\right)\left(x_1-x_2\right), \left(2+s_2^{\{23\}}\right)\left(x_3-x_2\right) \right\}\\
       ~~~-\left( 1+s_{\text{m}}^{\{13\mid 2\}} \right)\min\left\{ \left(2+s_1^{\{12\}}\right)\left(x_1-x_2\right), \left(2+s_2^{\{23\}}\right)\left(x_3-x_2\right) \right\},&~~ x_2 <\min(x_1,x_3),
	\end{cases}
\end{align*}
with $\min(x_1,x_2,x_3)\geq 0$ and
\begin{align*}
s^{\{13\mid 2\}}_{\text{m}} &= s^{\{13\mid 2\}}_1\mathbbm{1}_{\left\{\left(2+s_1^{\{12\}}\right)\left(x_1-x_2\right) \geq \left(2+s_2^{\{23\}}\right)\left(x_3-x_2\right) \right\}}+ s^{\{13\mid 2\}}_2\mathbbm{1}_{\left\{\left(2+s_1^{\{12\}}\right)\left(x_1-x_2\right) < \left(2+s_2^{\{23\}}\right)\left(x_3-x_2\right) \right\}}.
\end{align*}
\EDIT{As for the previous case, we note that $g(1,1,1)=1$, which implies that $\eta_{\{1,2,3\}}=\eta_{\{1,3\}}=1$.} For a trivariate vine consisting of three logistic pair copulas, this gives the gauge function
\begin{align*}
g(\bm{x})=
	\begin{cases}
       x_2 + \left[\left\{\left(1/\alpha-1\right)\left(x_2-x_1\right)\right\}^{1/\gamma} + \left\{\left(1/\beta-1\right)\left(x_2-x_3\right)\right\}^{1/\gamma}\right]^\gamma, & \max(x_1,x_3)\leq x_2,\\
       x_2 + \left(1/\gamma\right)\left(1/\beta\right)\left(x_3-x_2\right) +  \left(1/\alpha-1\right)\left(x_2-x_1 \right),& x_1 \leq x_2<x_3,\\
       x_2 + \left(1/\gamma\right)\left(1/\alpha\right)\left(x_1-x_2\right) +  \left(1/\beta-1\right)\left(x_2-x_3 \right),& x_3\leq x_2<x_1,\\
       x_2 + \left( 1/\gamma \right)\max\left\{ \left(x_1-x_2\right)/\alpha, \left(x_3-x_2\right)/\beta \right\}+\left( 1-1/\gamma \right)\min\left\{ \left(x_1-x_2\right)/\alpha, \left(x_3-x_2\right)/\beta \right\},&x_2 <\min(x_1,x_3),
	\end{cases}
\end{align*}
with $\min(x_1,x_2,x_3)\geq 0$; \EDIT{see Fig.~\ref{fig:g6} for an example.} 

\section{Discussion}
The aim of this paper was to investigate some of the tail dependence properties of vine copulas, via the coefficient of tail dependence of \cite{Ledford1996}. We demonstrated how to apply the geometric approach of \cite{Nolde2014} to calculate these values from a density, and applied further theory from \cite{Nolde2020} for cases where the joint density of $(X_i:i\in \mathcal{C})$ cannot be obtained analytically, but the joint density of $(X_i:i\in \mathcal{C}')$ with $\mathcal{C}'\supset \mathcal{C}$ is known. While values of $\eta_\mathcal{C}<1$ allow us to deduce that there is asymptotic independence between the variables $\bm{X}_\mathcal{C}$, these geometric approaches do not enable distinction between asymptotic independence and asymptotic dependence when $\eta_\mathcal{C}=1$.

We focused on trivariate vine copulas constructed from extreme value and inverted extreme value pair copulas, and higher dimensional $D$-vine and $C$-vine copulas constructed only from inverted extreme value pair copulas. In the latter case, there is overall asymptotic independence between the variables. In the former case, the copulas in tree $T_1$ particularly influence the overall tail dependence properties of the vine. If there are two asymptotically dependent extreme value copulas in tree $T_1$, there is overall asymptotic dependence in the vine, as found by \cite{Joeetal2010}, otherwise, all three variables cannot be large together, although other subsets of the variables could be simultaneously extreme.

In Section~\ref{sec:intro}, we discussed the idea of extremal dependence structures, i.e., that different subsets of variables can take their largest values simultaneously while others are of smaller order \citep{Simpson2020}. Let the extremal dependence structure of the variables $\bm{X}=(X_1,X_2,X_3)$ be denoted by a set $\mathcal{A}$, such that if $A\in\mathcal{A}$ the variables indexed by $A\subseteq\{1,2,3\}$ can be simultaneously large while the others are small. For the trivariate case, our examples comprise all possible combinations of asymptotically independent and asymptotically dependent pair copulas for the three components of the vine. Throughout the paper, the spectral density of the asymptotically dependent components was restricted to placing mass on $(0,1)$ as in~\eqref{eqn:EVassumption}, while asymptotic independence corresponds to mass on $\{0\}$ and $\{1\}$. Our results suggest that the only extremal dependence structures possible in this setting are $\mathcal{A}=\{\{1\},\{2\},\{3\}\}$, $\{\{1\},\{2,3\}\}$, $\{\{2\},\{1,3\}\}$, $\{\{3\},\{1,2\}\}$ and $\{\{1,2,3\}\}$. While it is unclear whether the structure $\mathcal{A}=\{\{2\},\{1,3\}\}$ is possible for the specific form of the vine we consider (Fig.~\ref{fig:triVine}), it is possible with relabelling of the variables, hence its inclusion here. This suggests that each variable is only represented in one of the simultaneously-extreme subsets, and it is likely that this issue would also occur in higher dimensions. Obtaining more complicated structures would require pair copulas that place extremal mass on different combinations of the sets $\{0\},(0,1),$ and $\{1\}$, such as the asymmetric logistic model of \cite{Tawn1990} discussed in case~(vi) of Section~\ref{subsec:bivariateEtas}. However, we conjecture that certain extremal dependence structures will never be possible due to restrictions imposed by the vine.

As an example, suppose we are interested in the structure $\{\{1,2\},\{1,3\},\{2,3\}\}$, so that only pairs of variables can be large simultaneously while the third is of smaller order. If both pair copulas in tree $T_1$ place mass on $(0,1)$, the set $\{1,2,3\}$ will be included in the extremal dependence structure \citep{Joeetal2010}. This implies that at least one component of $T_1$ must exhibit asymptotic independence to obtain our required structure. However, any pair of variables for which asymptotic independence is imposed in $T_1$ can never be simultaneously extreme, i.e., it would not possible for both $\{1,2\}$ and $\{2,3\}$ to be included in the dependence structure in this case. The structure $\{\{1,2\},\{1,3\},\{2,3\}\}$ can therefore not be achieved, and actually the pairs $\{1,2\}$ and $\{2,3\}$ cannot both be included in the extremal dependence structure unless $\{1,2,3\}$ also is.

Although the full set of extremal dependence structures may not be captured using vine copulas, it appears that they do allow for a wide range of possibilities, and investigating this topic further presents a possible avenue for future work.

\section*{Acknowledgments}
This paper is based on work completed while E.\ Simpson was part of the EPSRC funded STOR-i centre for doctoral training (EP/L015692/1). J.\ Wadsworth gratefully acknowledges the support of EPSRC fellowship EP/P002838/1. We thank Ingrid Hob{\ae}k Haff and Arnoldo Frigessi for helpful discussions, \EDIT{and the Editor, Associate Editor and referees for their comments. In particular, we appreciate feedback from the referees that allowed us obtain some results in closed form that we previously thought could only be found numerically.}

\appendix
\section*{Appendix}

\section{Proof of Theorem~\ref{thm:d-dim_D-vine_gauge}}\label{app:DvineEtaD_generalProof}
\subsection{Identifying sub-vines of $D$-vines to construct the gauge function}
A $D$-vine is represented graphically by a series of $d-1$ trees, labelled $T_1,\ldots,T_{d-1}$. Each of these trees is a path, and we suppose that the nodes are labelled in ascending order, as in the left plot of Fig.~\ref{fig:4D_Vines} for the case where $d=4$. Moving from a $D$-vine of dimension $d\geq 4$ to one of dimension $d+1$ involves first adding an extra node and edge onto each tree in the graph. In tree $T_1$, the extra node has label $d+1$, and the extra edge label is $\{d,d+1\}$. In tree $T_2$ the extra node is labelled $\{d,d+1\}$ and the edge is labelled $\{d-1,d+1\}|d$, and this continues until we reach tree $T_{d-1}$, where the extra node is labelled $\{3,d+1\}|\{4,\ldots,d\}$ and the corresponding edge is labelled $\{2,d+1\}|\{3,\ldots,d\}$. We finally must also add the tree $T_d$, with nodes labelled $\{1,d\}|\{2,\ldots,d-1\}$ and $\{2,d+1\}|\{3,\ldots,d\}$, and corresponding edge label $\{1,d+1\}|\{2,\ldots,d\}$. This is demonstrated in Fig.~\ref{fig:4Dto5DVine}, for an example where we move from a $D$-vine of dimension four to one of dimension five.

\begin{figure}[!htbp]
\begin{center}
	\includegraphics[width=0.5\textwidth]{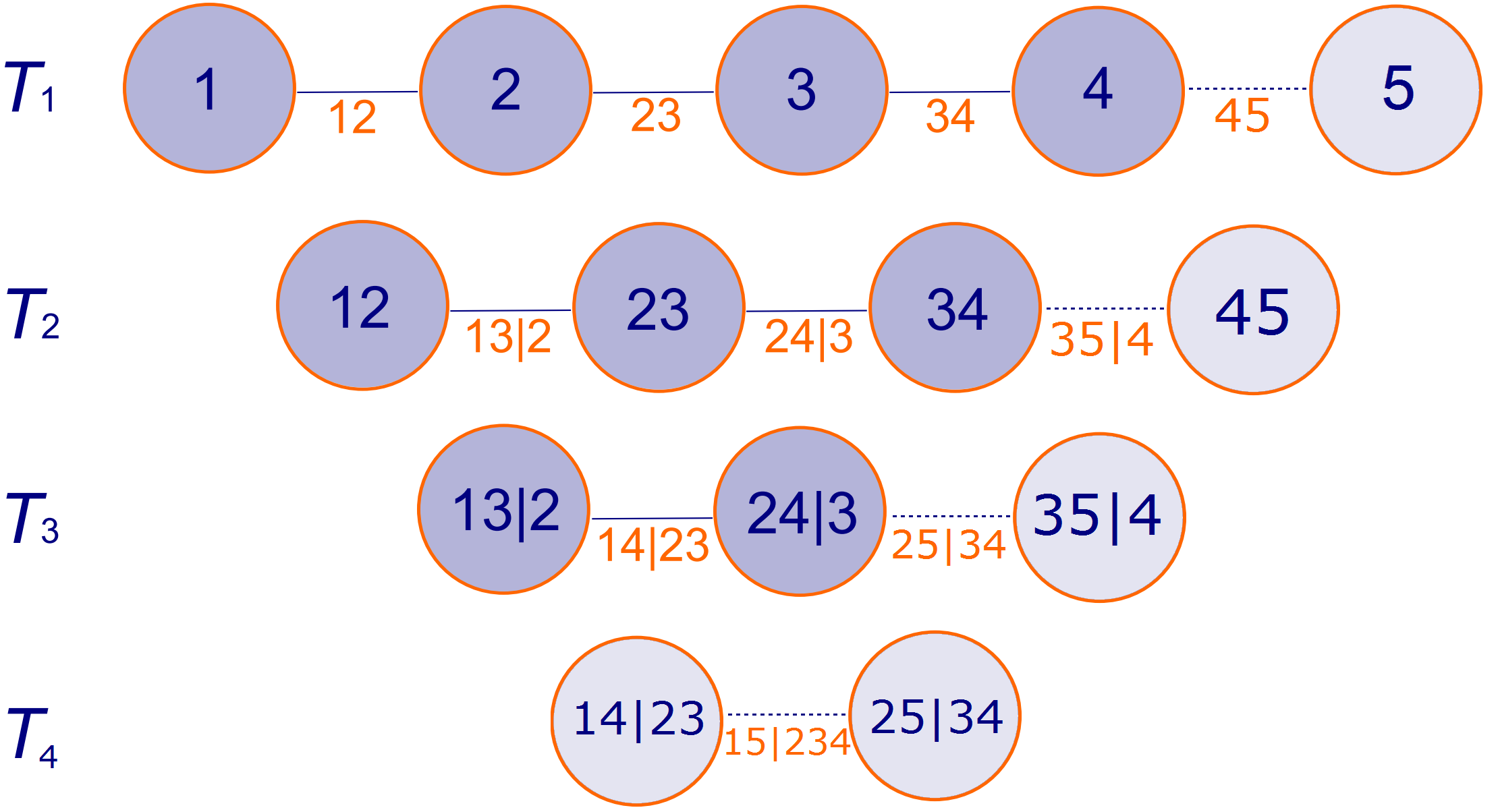}
	\caption{Example of the extending a four-dimensional $D$-vine to a five-dimensional $D$-vine.}
	\label{fig:4Dto5DVine}
\end{center}
\end{figure}

Due to this iterative construction, we can consider a $d$-dimensional $D$-vine in terms of three lower dimensional `sub-vines' in a similar way to in Fig.~\ref{fig:4DVineSeparated} for the $d=4$ case. In particular, in trees $T_1,\ldots,T_{d-2}$, we have two sub-vines of dimension $d-1$; the first corresponds to variables with labels in $\{1,\ldots,d-1\}=\mathcal{D}\backslash\{d\}$, and the second to variables with labels in $\{2,\ldots,d\}=\mathcal{D}\backslash\{1\}$. In the graph, these two sub-vines will overlap in the region corresponding to a further sub-vine, this time of dimension $d-2$ and corresponding to variables with labels in $\{2,\ldots,d-1\}=\mathcal{D}\backslash\{1,d\}$.

In order to calculate the gauge function, we consider the behaviour of $-\ln f(t\bm{x})$, as $t\rightarrow\infty$. By considering these three sub-vines, we see that this can be written as
\begin{align*}
	-\ln f(t\bm{x}) = &-\ln f_{\mathcal{D}\backslash\{d\}}\left(t\bm{x}_{-\{d\}}\right) -\ln f_{\mathcal{D}\backslash\{1\}}\left(t\bm{x}_{-\{1\}}\right) + \ln f_{\mathcal{D}\backslash\{1,d\}}\left(t\bm{x}_{-\{1,d\}}\right) \\
	&- \ln c_{\{1,d\}|\mathcal{D}\backslash\{1,d\}}\left\{F_{1|\mathcal{D}\backslash\{1,d\}}\left(tx_1|t\bm{x}_{-\{1,d\}}\right),F_{d|\mathcal{D}\backslash\{1,d\}}\left(tx_d|t\bm{x}_{-\{1,d\}}\right)\right\},
\end{align*}
$\bm{x}\in\mathbb{R}^d$. Note that this is the form given for the $d=4$ case in equation~\eqref{eqn:-logf_4d}. We can therefore infer that the $d$-dimensional gauge function $g(\bm{x})$, defined as $-\ln f(t\bm{x})\sim tg(\bm{x})$ as $t\rightarrow\infty$, satisfies
\begin{align}
g(\bm{x}) = g_{\mathcal{D}\backslash\{d\}}\left(\bm{x}_{-\{d\}}\right) + g_{\mathcal{D}\backslash\{1\}}\left(\bm{x}_{-\{1\}}\right) - g_{\mathcal{D}\backslash\{1,d\}}\left(\bm{x}_{-\{1,d\}}\right) + \tilde{g}_\mathcal{D}(\bm{x}),
\label{eqn:gaugeIncomplete}
\end{align}
$\bm{x}\in\mathbb{R}^d$, where, as $t\rightarrow\infty$,
\begin{align}
- \ln c_{\{1,d\}|\mathcal{D}\backslash\{1,d\}}\left\{F_{1|\mathcal{D}\backslash\{1,d\}}\left(tx_1|t\bm{x}_{-\{1,d\}}\right),F_{d|\mathcal{D}\backslash\{1,d\}}\left(tx_d|t\bm{x}_{-\{1,d\}}\right)\right\}\sim t\tilde{g}_\mathcal{D}(\bm{x}).
\label{eqn:gtilde}
\end{align}
In Section~\ref{subsec:IEVcopProperties} of the Supplementary Material, we present two lemmas concerning properties of inverted extreme value copulas that will be used in \ref{subsec:IEV_dvine_gaugeCalc} to find $\tilde{g}_\mathcal{D}(\bm{x})$, and hence the form of the gauge function for a $d$-dimensional $D$-vine with inverted extreme value components.

\subsection{Calculation of the gauge function}\label{subsec:IEV_dvine_gaugeCalc}
We claim that the $d$-dimensional $D$-vine has a gauge function of the form stated in Theorem~\ref{thm:d-dim_D-vine_gauge}. From equation~\eqref{eqn:4-dim_D-vine_gauge}, we have already shown this to be the case for $d=4$. To prove this more generally, we assume that the result holds for the two $(d-1)$-dimensional sub-vines of the $d$-dimensional $D$-vine, i.e.,
\begin{align}
	g_{\mathcal{D}\backslash\{d\}} = g_{\mathcal{D}\backslash\{1,d-1,d\}} + V^{\{1,d-1|\mathcal{D}\backslash\{1,d-1,d\}\}}\left(\frac{1}{g_{\mathcal{D}\backslash\{d-1,d\}}-g_{\mathcal{D}\backslash\{1,d-1,d\}}},\frac{1}{g_{\mathcal{D}\backslash\{1,d\}}-g_{\mathcal{D}\backslash\{1,d-1,d\}}}\right),
\label{gaugeAssumption1}
\end{align}
and
\[
	g_{\mathcal{D}\backslash\{1\}} = g_{\mathcal{D}\backslash\{1,2,d\}} + V^{\{2,d|\mathcal{D}\backslash\{1,2,d\}\}}\left(\frac{1}{g_{\mathcal{D}\backslash\{1,d\}}-g_{\mathcal{D}\backslash\{1,2,d\}}},\frac{1}{g_{\mathcal{D}\backslash\{1,2\}}-g_{\mathcal{D}\backslash\{1,2,d\}}}\right),
\]
where we have dropped the arguments to simplify notation. Further, we claim that the conditional distribution functions used in the calculation of~\eqref{eqn:gtilde} have the form
\begin{align}
	F_{1|\mathcal{D}\backslash\{1,d\}}&\left(tx_1|t\bm{x}_{-\{1,d\}}\right)= 1 - k_{1|\mathcal{D}\backslash\{1,d\}}\exp\left\{-t\left(g_{\mathcal{D}\backslash\{d\}} - g_{\mathcal{D}\backslash\{1,d\}}\right)\right\}\{1+o(1)\},
\label{eqn:F_1givenrest}
\end{align}
and
\begin{align}
	F_{d|\mathcal{D}\backslash\{1,d\}}&\left(tx_d|t\bm{x}_{-\{1,d\}}\right) = 1 - k_{d|\mathcal{D}\backslash\{1,d\}}\exp\left\{-t\left(g_{\mathcal{D}\backslash\{1\}} - g_{\mathcal{D}\backslash\{1,d\}}\right)\right\}\{1+o(1)\},
\label{eqn:F_dgivenrest}
\end{align}
as $t\rightarrow\infty$, for some $k_{1|\mathcal{D}\backslash\{1,d\}},k_{d|\mathcal{D}\backslash\{1,d\}}>0$ not depending on $t$. From result~\eqref{eqn:F1|23asymp}, we see that this claim holds for $d=4$. To prove this more generally, we assume that \eqref{eqn:F_1givenrest}~and~\eqref{eqn:F_dgivenrest} hold in the $(d-1)$-dimension case, so that  as $t\rightarrow\infty$,
\begin{align*}
	&F_{1|\mathcal{D}\backslash\{1,d-1,d\}}\left(tx_1|t\bm{x}_{-\{1,d-1,d\}}\right) = 1 - k_{1|\mathcal{D}\backslash\{1,d-1,d\}}\exp\left\{-t\left(g_{\mathcal{D}\backslash\{d-1,d\}} - g_{\mathcal{D}\backslash\{1,d-1,d\}}\right)\right\}\{1+o(1)\},
\end{align*}
and
\begin{align*}
	&F_{d-1|\mathcal{D}\backslash\{1,d-1,d\}}\left(tx_{d-1}|t\bm{x}_{-\{1,d-1,d\}}\right) = 1 - k_{d-1|\mathcal{D}\backslash\{1,d-1,d\}}\exp\left\{-t\left(g_{\mathcal{D}\backslash\{1,d\}} - g_{\mathcal{D}\backslash\{1,d-1,d\}}\right)\right\}\{1+o(1)\},
\end{align*}
for some $k_{1|\mathcal{D}\backslash\{1,d-1,d\}}, k_{d-1|\mathcal{D}\backslash\{1,d-1,d\}}>0$. Results from \cite{Joe1996} show that
\begin{align*}
	F&_{1|\mathcal{D}\backslash\{1,d\}}\left(x_1|\bm{x}_{-\{1,d\}}\right) = \frac{\partial C_{1,d-1|\mathcal{D}\backslash \{1,d-1,d\}} \left\{F_{1|\mathcal{D}\backslash \{1,d-1,d\}}(x_1|\bm{x}_{-\{1,d-1,d\}}), F_{d-1|\mathcal{D}\backslash \{1,d-1,d\}}(x_{d-1}|\bm{x}_{-\{1,d-1,d\}})\right\} }{\partial  F_{d-1|\mathcal{D}\backslash \{1,d-1,d\}}(x_{d-1}|\bm{x}_{-\{1,d-1,d\}})},
\end{align*}
with result~\eqref{eqn:conditional} giving the form of the required derivative of an inverted extreme value copula. Applying Lemma~\ref{lemma:conditional}, with $b_1=g_{\mathcal{D}\backslash\{d-1,d\}} - g_{\mathcal{D}\backslash\{1,d-1,d\}}$ and $b_2=g_{\mathcal{D}\backslash\{1,d\}} - g_{\mathcal{D}\backslash\{1,d-1,d\}}$, we see that for some $k_{1|\mathcal{D}\backslash\{1,d\}}$, as $t\rightarrow\infty$,
\begin{align*}
F&_{1|\mathcal{D}\backslash\{1,d\}}\left(tx_1|t\bm{x}_{-\{1,d\}}\right)= 1 - k_{1|\mathcal{D}\backslash\{1,d\}}\\
&\cdot\exp\bigg(
-t\bigg[V^{\{1,d-1|\mathcal{D}\backslash\{1,d-1,d\}\}}\left\{\frac{1}{g_{\mathcal{D}\backslash\{d-1,d\}} - g_{\mathcal{D}\backslash\{1,d-1,d\}}}, \frac{1}{g_{\mathcal{D}\backslash\{1,d\}} - g_{\mathcal{D}\backslash\{1,d-1,d\}}}\right\}-g_{\mathcal{D}\backslash\{1,d\}} + g_{\mathcal{D}\backslash\{1,d-1,d\}}\bigg]\bigg)\{1+o(1)\}\\
&=1 - k_{1|\mathcal{D}\backslash\{1,d\}}\exp\left\{ -t\left( g_{\mathcal{D}\backslash\{d\}}-g_{\mathcal{D}\backslash\{1,d\}}\right)\right\}\{1+o(1)\}~~~~~~~~~~\text{by assumption~\eqref{gaugeAssumption1}}.
\end{align*}
Result~\eqref{eqn:F_dgivenrest} can be proved by a similar argument. From results~\eqref{eqn:F_1givenrest}~and~\eqref{eqn:F_dgivenrest}, we see that $F_{1|\mathcal{D}\backslash\{1,d\}}\left(tx_1|t\bm{x}_{-\{1,d\}}\right)$ and $F_{d|\mathcal{D}\backslash\{1,d\}}\left(tx_d|t\bm{x}_{-\{1,d\}}\right)$ can be written in the form required to apply Lemma~\ref{lemma:copulaDensity}, with $b_1 = g_{\mathcal{D}\backslash\{d\}} - g_{\mathcal{D}\backslash\{1,d\}}$ and $b_2 = g_{\mathcal{D}\backslash\{1\}}  - g_{\mathcal{D}\backslash\{1,d\}}$. Applying Lemma~\ref{lemma:copulaDensity}, we have
\begin{align*}
- &\ln c_{\{1,d\}|\mathcal{D}\backslash\{1,d\}}\left\{F_{1|\mathcal{D}\backslash\{1,d\}}\left(tx_1|t\bm{x}_{-\{1,d\}}\right),F_{d|\mathcal{D}\backslash\{1,d\}}\left(tx_d|t\bm{x}_{-\{1,d\}}\right)\right\}\\
&\sim t\left\{2g_{\mathcal{D}\backslash\{1,d\}} - g_{\mathcal{D}\backslash\{d\}} - g_{\mathcal{D}\backslash\{1\}}+V^{\{1,d\mid \mathcal{D}\backslash\{1,d\}\}}\left(\frac{1}{g_{\mathcal{D}\backslash\{d\}} -g_{\mathcal{D}\backslash\{1,d\}}},\frac{1}{g_{\mathcal{D}\backslash\{1\}} -g_{\mathcal{D}\backslash\{1,d\}}}\right)\right\}\\
	&=t\tilde{g}_\mathcal{D}(\bm{x}),
\end{align*}
and combining this with the gauge function result in~\eqref{eqn:gaugeIncomplete}, we have
\begin{align*}
	g&(\bm{x}) = g_{\mathcal{D}\backslash\{1,d\}}(\bm{x}_{-\{1,d\}})+V^{\{1,d\mid \mathcal{D}\backslash\{1,d\}\}}\left\{\frac{1}{g_{\mathcal{D}\backslash\{d\}}(\bm{x}_{-\{d\}})-g_{\mathcal{D}\backslash\{1,d\}}(\bm{x}_{-\{1,d\}})},\frac{1}{g_{\mathcal{D}\backslash\{1\}}(\bm{x}_{-\{1\}})-g_{\mathcal{D}\backslash\{1,d\}}(\bm{x}_{-\{1,d\}})}\right\},
\end{align*}
hence proving Theorem~\ref{thm:d-dim_D-vine_gauge} by induction.

%~~~~~~~~~~~~~~~~~~~~~~~~~~~~~~~~~~~~~~~~~~~~~~~~~~~~~~~~~~~~~~~~~~~~~~~~~~~~~~~~~~~~~~~~~~~~~~~~~~~~~

%\bibliographystyle{apalike}
{\small\bibliography{refs}}

%~~~~~~~~~~~~~~~~~~~~~~~~~~~~~~~~~~~~~~~~~~~~~~~~~~~~~~~~~~~~~~~~~~~~~~~~~~~~~~~~~~~~~~~~~~~~~~~~~~~~~

\newpage
\appendix
\begin{center}
{\bf\Large Supplementary Material for `A geometric investigation into the tail dependence of vine copulas'}\\
{\large Emma S.\ Simpson, Jennifer L.\ Wadsworth and Jonathan, A.\ Tawn\\}
Lancaster University
\end{center}

\section{Calculation of the gauge function for a trivariate vine copula with inverted extreme value pair copula components}\label{SMsec:triIEVproof}
In Section~\ref{subsec:trivariateIEV}, we presented the form of the gauge function for a trivariate vine copula constructed from three inverted extreme value components. The exponent measures are denoted $V^{\{12\}}$, $V^{\{23\}}$ and $V^{\{13|2\}}$, corresponding to the links in the vine. Here, we provide a proof of this result.

To investigate the form of the gauge function, we study the asymptotic form of $-\ln f(t\bm{x})$, and for a trivariate vine this can be broken down into the six components in~\eqref{eqn:negLogTriVine}. We consider each of these in turn. Working in standard exponential margins, we have $-\ln f_i(tx_i) = -\ln\left(e^{-tx_i}\right)= tx_i$, for $i=1,2,3$ with $x_i\geq 0$. We also have marginal distribution functions $F_{i}(tx_i)=1-e^{-tx_i}$, for $i=1,2,3$. Noting that $-\ln\left\{1-F_i(tx_i)\right\} = tx_i$; the exponent measure $V^{\{12\}}$ is homogeneous of order $-1$; $V^{\{12\}}_1$ and $V^{\{12\}}_2$ are homogeneous of order $-2$; and $V^{\{12\}}_{12}$ is homogeneous of order $-3$, we have
\begin{align}
	-\ln c_{12}\left\{F_1(tx_1),F_2(tx_2)\right\} = &-tx_1 - tx_2 + 2\ln(tx_1) + 2\ln(tx_2)+ V^{\{12\}}\left\{(tx_1)^{-1},(tx_2)^{-1}\right\}\nonumber\\ 
	&-\ln\left[V^{\{12\}}_1\left\{(tx_1)^{-1},(tx_2)^{-1}\right\}V^{\{12\}}_2\left\{(tx_1)^{-1},(tx_2)^{-1}\right\} - V^{\{12\}}_{12}\left\{(tx_1)^{-1},(tx_2)^{-1}\right\}\right]\nonumber\\
	=&~t\left\{V^{\{12\}}\left(x_1^{-1},x_2^{-1}\right)-x_1-x_2\right\} + 2\left\{\ln(tx_1)+\ln(tx_2)\right\}\nonumber\\
	&-\ln\left\{t^4 V^{\{12\}}_1\left(x_1^{-1},x_2^{-1}\right)V^{\{12\}}_2\left(x_1^{-1},x_2^{-1}\right) - t^3 V^{\{12\}}_{12}\left(x_1^{-1},x_2^{-1}\right)\right\}\nonumber\\
	=&~t\left\{V^{\{12\}}\left(x_1^{-1},x_2^{-1}\right)-x_1-x_2\right\} + O(\ln t),
\label{eqn:negLogTriVine(a)}
\end{align}
and by similar calculations,
\begin{align}
	-\ln c_{23}\left\{F_2(tx_2),F_3(tx_3)\right\} = t\left\{V^{\{23\}}\left(x_2^{-1},x_3^{-1}\right)-x_2-x_3\right\} + O(\ln t).
\label{eqn:negLogTriVine(b)}
\end{align}
Substituting the marginal distribution functions $F_{i}(tx_i)$, $i=1,2$, into the corresponding conditional copulas, we have 
\begin{align*}
	F_{1\mid 2}(tx_1\mid tx_2) &= 1 + e^{tx_2}(tx_2)^{-2}V_2^{\{12\}}\left\{(tx_1)^{-1},(tx_2)^{-1}\right\}\exp\left[-V^{\{12\}}\left\{(tx_1)^{-1},(tx_2)^{-1}\right\}\right]\\
	&=1 + x_2^{-2}V^{\{12\}}_2\left(x_1^{-1},x_2^{-1}\right)\exp\left[t\left\{x_2 - V^{\{12\}}\left(x_1^{-1}, x_2^{-1}\right)\right\}\right],
\end{align*}
and similarly,
\[
	F_{3\mid 2}(tx_3\mid tx_2) = 1 + x_2^{-2}V^{\{23\}}_1\left(x_2^{-1},x_3^{-1}\right)\exp\left[t\left\{x_2 - V^{\{23\}}\left(x_2^{-1}, x_3^{-1}\right)\right\}\right].
\]
Setting $h_{1\mid 2}=\ln\left\{- x_2^{-2}V_2^{\{12\}}\left(x_1^{-1},x_2^{-1}\right)\right\}$ and defining $h_{3\mid 2}$ analogously,
\begin{align*}
\ln\left\{1-F_{1\mid 2}(tx_1\mid tx_2)\right\} = h_{1\mid 2} + t\left\{x_2-V^{\{12\}}\left(x_1^{-1},x_2^{-1}\right)\right\};\\
\ln\left\{1-F_{3\mid 2}(tx_3\mid tx_2)\right\} = h_{3\mid 2} + t\left\{x_2-V^{\{23\}}\left(x_2^{-1},x_3^{-1}\right)\right\},
\end{align*}
with $h_{1\mid 2}$ and $h_{3\mid 2}$ not depending on $t$. This implies that
{\small \begin{align}
	-\ln c_{13\mid 2}\left\{F_{1\mid 2}(tx_1\mid tx_2), F_{3\mid 2}(tx_3\mid tx_2)\right\} &\nonumber\\
	&\hspace{-5cm}= \ln\left\{1-F_{1\mid 2}(tx_1\mid tx_2)\right\} + \ln\left\{1-F_{3\mid 2}(tx_3\mid tx_2)\right\}+ 2\ln\left[-\ln\left\{1-F_{1\mid 2}(tx_1\mid tx_2)\right\}\right] + 2\ln\left[-\ln\left\{1-F_{3\mid 2}(tx_3\mid tx_2)\right\}\right]\nonumber\\
	&\hspace{-4.75cm}+V^{\{13\mid 2\}}\left[\frac{-1}{\ln\left\{1-F_{1\mid 2}(tx_1\mid tx_2)\right\}},\frac{-1}{\ln\left\{1-F_{3\mid 2}(tx_3\mid tx_2)\right\}} \right]\nonumber\\
	&\hspace{-4.75cm}-\ln\Bigg(V_1^{\{13\mid 2\}}\left[\frac{-1}{\ln\left\{1-F_{1\mid 2}(tx_1\mid tx_2)\right\}},\frac{-1}{\ln\left\{1-F_{3\mid 2}(tx_3\mid tx_2)\right\}} \right]\cdot V_2^{\{13\mid 2\}}\left[\frac{-1}{\ln\left\{1-F_{1\mid 2}(tx_1\mid tx_2)\right\}},\frac{-1}{\ln\left\{1-F_{3\mid 2}(tx_3\mid tx_2)\right\}} \right]\nonumber\\
	&\hspace{-3.5cm}-V_{12}^{\{13\mid 2\}}\left[\frac{-1}{\ln\left\{1-F_{1\mid 2}(tx_1\mid tx_2)\right\}},\frac{-1}{\ln\left\{1-F_{3\mid 2}(tx_3\mid tx_2)\right\}} \right]\Bigg)\nonumber\\
	&\hspace{-5cm}=h_{1\mid 2} + h_{3\mid 2}+ t\left\{2x_2 - V^{\{12\}}(x_1^{-1},x_2^{-1})- V^{\{23\}}(x_2^{-1},x_3^{-1})\right\}\nonumber\\
	&\hspace{-4.75cm}+ 2\ln \left[t\left\{V^{\{12\}}\left(x_1^{-1},x_2^{-1}\right)-x_2\right\}-h_{1\mid 2}\right]	+ 2\ln \left[t\left\{V^{\{23\}}\left(x_2^{-1},x_3^{-1}\right)-x_2\right\}-h_{3\mid 2}\right]	\nonumber\\
	&\hspace{-4.75cm} + tV^{\{13\mid 2\}}\left\{\frac{1}{V^{\{12\}}(x_1^{-1},x_2^{-1})-x_2 -\frac{1}{t}h_{1\mid 2}},\frac{1}{V^{\{23\}}(x_2^{-1},x_3^{-1})-x_2 -\frac{1}{t}h_{3\mid 2}}\right\}\nonumber\\	
	&\hspace{-4.75cm}-\ln\Bigg[t^4V_1^{\{13\mid 2\}}\left\{\frac{1}{V^{\{12\}}(x_1^{-1},x_2^{-1})-x_2 -\frac{1}{t}h_{1\mid 2}},\frac{1}{V^{\{23\}}(x_2^{-1},x_3^{-1})-x_2 -\frac{1}{t}h_{3\mid 2}}\right\}\nonumber\\
	&\hspace{-2.25cm}\cdot V_2^{\{13\mid 2\}}\left\{\frac{1}{V^{\{12\}}(x_1^{-1},x_2^{-1})-x_2 -\frac{1}{t}h_{1\mid 2}},\frac{1}{V^{\{23\}}(x_2^{-1},x_3^{-1})-x_2 -\frac{1}{t}h_{3\mid 2}}\right\}\nonumber\\
	&\hspace{-2cm}-t^3V_{12}^{\{13\mid 2\}}\left\{\frac{1}{V^{\{12\}}(x_1^{-1},x_2^{-1})-x_2 -\frac{1}{t}h_{1\mid 2}},\frac{1}{V^{\{23\}}(x_2^{-1},x_3^{-1})-x_2 -\frac{1}{t}h_{3\mid 2}}\right\}	\Bigg]\nonumber\\
	&\hspace{-5cm}= t\left\{2x_2-V^{\{12\}}\left(x_1^{-1},x_2^{-1}\right)-V^{\{23\}}\left(x_2^{-1},x_3^{-1}\right) \right\}+ tV^{\{13\mid 2\}}\left[\left\{V^{\{12\}}\left(x_1^{-1},x_2^{-1}\right)-x_2\right\}^{-1},\left\{V^{\{23\}}\left(x_2^{-1},x_3^{-1}\right)-x_2\right\}^{-1} \right] + O(\ln t).
\label{eqn:negLogTriVine(c)}
\end{align}}
Substituting results~\eqref{eqn:negLogTriVine(a)},~\eqref{eqn:negLogTriVine(b)}~and~\eqref{eqn:negLogTriVine(c)} into \eqref{eqn:negLogTriVine} yields 
{\small \begin{align}
	-\ln f(t\bm{x}) &= ~t(x_1+x_2+x_3) + t\left\{V^{\{12\}}\left(x_1^{-1},x_2^{-1}\right)-x_1-x_2\right\} + t\left\{V^{\{23\}}\left(x_2^{-1},x_3^{-1}\right)-x_2-x_3\right\}\nonumber\\
		&\hspace{-0.5cm} + t\left\{2x_2 - V^{\{12\}}\left(x_1^{-1},x_2^{-1}\right) - V^{\{23\}}\left(x_2^{-1},x_3^{-1}\right)\right\}+ tV^{\{13\mid 2\}}\left[\left\{V^{\{12\}}\left(x_1^{-1},x_2^{-1}\right)-x_2\right\}^{-1},\left\{V^{\{23\}}\left(x_2^{-1},x_3^{-1}\right)-x_2\right\}^{-1} \right]+ O(\ln t)\nonumber\\
		&=~t\left(x_2 + V^{\{13\mid 2\}}\left[\left\{V^{\{12\}}\left(x_1^{-1},x_2^{-1}\right)-x_2\right\}^{-1},\left\{V^{\{23\}}\left(x_2^{-1},x_3^{-1}\right)-x_2\right\}^{-1} \right]\right)+ O(\ln t),
		\label{eqn:trivariateAsymptotics}
\end{align}}
i.e., the gauge function of this model is
\[
 g(\bm{x}) = x_2 + V^{\{13\mid 2\}}\left[\left\{V^{\{12\}}\left(x_1^{-1},x_2^{-1}\right)-x_2\right\}^{-1},\left\{V^{\{23\}}\left(x_2^{-1},x_3^{-1}\right)-x_2\right\}^{-1} \right].
\]

\EDIT{\section{Proof that~\eqref{eqn:ilogilogilogGauge} is an increasing function of $x_2$ and that \eqref{eqn:ilogilogilog_tconstraint} has a unique solution}\label{SMsec:4.1proofs}
Our calculations of $\eta_{\{1,2,3\}}$ and $\eta_{\{1,3\}}$ in Section~\ref{subsec:trivariateIEV} rely on certain properties of the gauge function in~\eqref{eqn:ilogilogilogGauge}, which we prove here. In particular, we are aiming to show that equation~\eqref{eqn:ilogilogilogGauge} is an increasing function of $x_2\geq 1$, and that equation~\eqref{eqn:ilogilogilog_tconstraint} has a unique solution.}

\EDIT{Let $h(v)=g(1,v,1)$, with the gauge function $g(\bm{x})$ defined as in equation~\eqref{eqn:ilogilogilogGauge}. Then,
\[
h(v) = v+\left(f_1^{1/\gamma}+f_2^{1/\gamma}\right)^\gamma, ~~~~~\text{with } ~f_1:=(1+v^{1/\alpha})^\alpha-v,~~f_2:=(1+v^{1/\beta})^\beta-v.
\]
Note that for $v\geq 0$ and $\alpha,\beta\in(0,1)$, $f_1,f_2\in(0,1]$, with $f_1,f_2\rightarrow 0$ as $v\rightarrow\infty$ and $f_1,f_2\rightarrow 1$ as $v\rightarrow 0$. Moreover,
\[
f_1'=v^{1/\alpha-1}(1+v^{1/\alpha})^{\alpha-1}-1=(1+v^{-1/\alpha})^{\alpha-1}-1<0,~~\forall~v\geq 0,
\]
so $f_1$, and by analogy $f_2$, is a decreasing function of $v$. We also note that the second derivative has the form
\[
f_1''=\frac{1-\alpha}{\alpha}v^{-1/\alpha-1}(1+v^{-1/\alpha})^{\alpha-2}>0,~~\forall~v\geq 0.
\]
Then we have
\[
h'(v)=1+\left(f_1^{1/\gamma}+f_2^{1/\gamma}\right)^{\gamma-1}\left(f_1^{1/\gamma-1}f_1' + f_2^{1/\gamma-1}f_2'\right),
\]
and
\begin{align}
&h''(v)=\frac{\gamma-1}{\gamma}\left(f_1^{1/\gamma}+f_2^{1/\gamma}\right)^{\gamma-2}\left(f_1^{1/\gamma-1}f_1' + f_2^{1/\gamma-1}f_2'\right)^2 \nonumber\\
&\hspace{3cm}+ \left(f_1^{1/\gamma}+f_2^{1/\gamma}\right)^{\gamma-1}\left\{\left(\frac{1}{\gamma}-1\right)f_1^{1/\gamma-2}f_1'^2 + f_1''f_1^{1/\gamma-1}+\left(\frac{1}{\gamma}-1\right)f_2^{1/\gamma-2}f_2'^2 + f_2''f_2^{1/\gamma-1}\right\}\nonumber\\
%&=\left(f_1^{1/\gamma}+f_2^{1/\gamma}\right)^{\gamma-2}\left[-\left(\frac{1-\gamma}{\gamma}\right)\left(f_1^{1/\gamma-1}f_1' + f_2^{1/\gamma-1}f_2'\right)^2+ \left(f_1^{1/\gamma}+f_2^{1/\gamma}\right)\left\{\left(\frac{1-\gamma}{\gamma}\right)f_1^{1/\gamma-2}f_1'^2 + f_1''f_1^{1/\gamma-1}+\left(\frac{1-\gamma}{\gamma}\right)f_2^{1/\gamma-2}f_2'^2 + f_2''f_2^{1/\gamma-1}\right\}\right]\nonumber\\
&=\left(f_1^{1/\gamma}+f_2^{1/\gamma}\right)^{\gamma-2}\nonumber\\
&\hspace{0.5cm}\cdot\left[\left(\frac{1-\gamma}{\gamma}\right)\left\{\left(f_1^{1/\gamma}+f_2^{1/\gamma}\right)\left(f_1^{1/\gamma-2}f_1'^2 +f_2^{1/\gamma-2}f_2'^2\right)-\left(f_1^{1/\gamma-1}f_1' + f_2^{1/\gamma-1}f_2'\right)^2\right\}
+\left(f_1^{1/\gamma}+f_2^{1/\gamma}\right)\left(f_1''f_1^{1/\gamma-1} + f_2''f_2^{1/\gamma-1}\right)\right].
\label{eqn:h''}
\end{align}
Since $f_1,f_2\in(0,1]$, we clearly have $\left(f_1^{1/\gamma}+f_2^{1/\gamma}\right)^{\gamma-2}>0$. We now consider the two terms within the square brackets in \eqref{eqn:h''} separately. First note that $(1-\gamma)/\gamma>0$ for $\gamma\in(0,1)$, and
\begin{align*}
\left(f_1^{1/\gamma}+f_2^{1/\gamma}\right)&\left(f_1^{1/\gamma-2}f_1'^2 +f_2^{1/\gamma-2}f_2'^2\right)-\left(f_1^{1/\gamma-1}f_1' + f_2^{1/\gamma-1}f_2'\right)^2\\
&=f_1'^2\left(f_1^{2/\gamma-2}+f_1^{1/\gamma-2}f_2^{1/\gamma}-f_1^{2/\gamma-2}\right) + f_2'^2\left(f_2^{2/\gamma-2}+f_2^{1/\gamma-2}f_1^{1/\gamma}-f_2^{2/\gamma-2}\right)-2f_1'f_2'f_1^{1/\gamma-1}f_2^{1/\gamma-1}\\
&=\left(f_1f_2\right)^{1/\gamma}\left\{\left(f_1'f_1^{-1}\right)^2+\left(f_2'f_2^{-1}\right)^2 - 2f_1'f_2'f_1^{-1}f_2^{-1}\right\}\\
&=\left(f_1f_2\right)^{1/\gamma}\left(f_1'f_1^{-1}+f_2'f_2^{-1}\right)^2>0.
\end{align*}
Secondly, since we have already shown that $f_1,f_2,f_1'',f_2''>0$, it follows that
\[
\left(f_1^{1/\gamma}+f_2^{1/\gamma}\right)\left(f_1''f_1^{1/\gamma-1} + f_2''f_2^{1/\gamma-1}\right)>0.
\]
So all terms in \eqref{eqn:h''} are strictly positive. That is, $h''(v)>0$ for all $v\geq 0$, which implies that $h'(v)$ is increasing for $v\geq 0$.}

\EDIT{We note that $h'(0)=1-2^\gamma<0$, and we have
\[
h'(1) = 1 - \left\{(2^\alpha-1)^{1/\gamma}+(2^\beta-1)^{1/\gamma}\right\}^{\gamma-1}\left\{(1-2^{\alpha-1})(2^\alpha-1)^{-1+1/\gamma}+(1-2^{\beta-1})(2^\beta-1)^{-1+1/\gamma}\right\}.
\]
Since $x^{1/\gamma}+y^{1/\gamma}> x^{1/\gamma}$, for $x,y>0$, we have $(x^{1/\gamma}+y^{1/\gamma})^{\gamma-1}< (x^{1/\gamma})^{\gamma-1}$ for $\gamma\in(0,1)$, so 
\[
h'(1)> 1- (2^\alpha-1)^{1-1/\gamma}\left\{(1-2^{\alpha-1})(2^\alpha-1)^{-1+1/\gamma}+(1-2^{\beta-1})(2^\beta-1)^{-1+1/\gamma}\right\}.
\]
Without loss of generality, assume that $\alpha\geq\beta$, so $1-2^{\alpha-1}\leq1-2^{\beta-1}$ and $2^\alpha-1\geq 2^\beta-1$, then
\[
h'(1)> 1- (2^\alpha-1)^{1-1/\gamma}\cdot 2\cdot(1-2^{\beta-1})(2^\alpha-1)^{-1+1/\gamma}=1-2(1-2^{\beta-1})=2^\beta-1>0.
\]
So $h'(v)$ is increasing in the range $[0,1]$, and we have shown that $h'(0)<0$ and $h'(1)>0$, so that solution of the equation $h'(v)=0$ in equation~\eqref{eqn:ilogilogilog_tconstraint} is unique, and lies in the range $(0,1)$.}

\EDIT{Moreover, this implies that $h'(v)>0$ for $v\geq 1$, so $h(v)=g(1,v,1)$ is increasing for $v\geq 1$, and therefore $\min_{v\geq 1}g(1,v,1)=g(1,1,1)$, which is required for equation~\eqref{eqn:ilogilogilogEta123} to hold.}

\section{Proof of~\eqref{eqn:4-dim_D-vine_gauge}}\label{SMsec:4Dvinegauge}
In Section~\ref{subsec:trivariateIEV}, we studied gauge functions of trivariate vine copulas with inverted extreme value components. From calculation~\eqref{eqn:trivariateAsymptotics}, we know that
\begin{align}
-\ln f_{123}(tx_1,tx_2,tx_3)= &t\left(x_2 + V^{\{13\mid 2\}}\left[\left\{V^{\{12\}}\left(x_1^{-1},x_2^{-1}\right)-x_2\right\}^{-1},\left\{V^{\{23\}}\left(x_2^{-1},x_3^{-1}\right)-x_2\right\}^{-1} \right]\right)  +O(\ln t),
		\label{eqn:-logf123}
\end{align}
$x_1,x_2,x_3\geq 0$, and considering $f_{234}$ as equivalent to $f_{123}$ up to labelling of the variables, we have
\begin{align}
-\ln f_{234}(tx_2,tx_3,tx_4)=&t\left(x_3 + V^{\{24\mid 3\}}\left[\left\{V^{\{23\}}\left(x_2^{-1},x_3^{-1}\right)-x_3\right\}^{-1},\left\{V^{\{34\}}\left(x_3^{-1},x_4^{-1}\right)-x_3\right\}^{-1} \right]\right)+ O(\ln t),
		\label{eqn:-logf234}
\end{align}
$x_2,x_3,x_4\geq 0$. Moreover, from the four-dimensional vine, the pair $(X_2,X_3)$ is modelled by an inverted extreme value copula with exponent measure $V^{\{23\}}$, so by equation~\eqref{eqn:bivariateIEV} we have
\begin{align}
	\ln f_{23}(tx_2,tx_3) = -tV^{\{23\}}\left(x_2^{-1} , x_3^{-1}\right) + O(\ln t).
		\label{eqn:-logf23}
\end{align}
As such, $-\ln c_{14|23}\left\{F_{1|23}(tx_1|tx_2,tx_3),F_{4|23}(tx_4|tx_2,tx_3)\right\}$ is the only term of \eqref{eqn:4Dlogexpansion} left for us to study. \cite{Joe1996} shows that 
\[
	F_{1|23}(x_1|x_2,x_3) = \frac{\partial C_{13|2}\left\{F_{1|2}(x_1|x_2),F_{3|2}(x_3|x_2)\right\}}{\partial F_{3|2}(x_3|x_2)},
\]
and
\[
	F_{4|23}(x_4|x_2,x_3) = \frac{\partial C_{24|3}\left\{F_{2|3}(x_2|x_3),F_{4|3}(x_4|x_3)\right\}}{\partial F_{2|3}(x_2|x_3)}.
\]
Since we focus on inverted extreme value copulas, the derivatives have the form~\eqref{eqn:conditional}, where the arguments $u$ and $v$ are replaced by the appropriate bivariate conditional distributions. We find that for some functions $h_{1|23},h_{4|23}$ not depending on $t$,
{\footnotesize\begin{align}
	&1-F_{1|23}(tx_1\mid tx_2,tx_3) \sim\exp(h_{1|23})\cdot\exp\bigg\{-t \bigg(x_2 - V^{\{23\}}\left(x_2^{-1} , x_3^{-1}\right) + V^{\{13\mid 2\}}\left[\left\{V^{\{12\}}\left(x_1^{-1},x_2^{-1}\right)-x_2\right\}^{-1},\left\{V^{\{23\}}\left(x_2^{-1},x_3^{-1}\right)-x_2\right\}^{-1} \right]\bigg) \bigg\}
	\label{eqn:F1|23asymp}
\end{align}}
and an analogous result for $1-F_{4|23}(tx_4\mid tx_2,tx_3)$. These are similar to the results for the trivariate vine, and applying an argument analogous to~\eqref{eqn:negLogTriVine(c)}, we see that
\begin{align}
-\ln c_{14|23}\left\{F_{1|23}(tx_1|tx_2,tx_3),F_{4|23}(tx_4|tx_2,tx_3)\right\}&\sim\nonumber\\
&\hspace{-6cm}-x_2 + V^{\{23\}}\left(x_2^{-1} , x_3^{-1}\right) - V^{\{13\mid 2\}}\left[\left\{V^{\{12\}}\left(x_1^{-1},x_2^{-1}\right)-x_2\right\}^{-1},\left\{V^{\{23\}}\left(x_2^{-1},x_3^{-1}\right)-x_2\right\}^{-1} \right]\nonumber\\
&\hspace{-6cm}-x_3 + V^{\{23\}}\left(x_2^{-1} , x_3^{-1}\right) - V^{\{24\mid 3\}}\left[\left\{V^{\{23\}}\left(x_2^{-1},x_3^{-1}\right)-x_3\right\}^{-1},\left\{V^{\{34\}}\left(x_3^{-1},x_4^{-1}\right)-x_3\right\}^{-1} \right]\nonumber\\
&\hspace{-6cm}+V^{\{14|23\}}\bigg\{\bigg(x_2 - V^{\{23\}}\left(x_2^{-1} , x_3^{-1}\right) + V^{\{13\mid 2\}}\left[\left\{V^{\{12\}}\left(x_1^{-1},x_2^{-1}\right)-x_2\right\}^{-1},\left\{V^{\{23\}}\left(x_2^{-1},x_3^{-1}\right)-x_2\right\}^{-1} \right]\bigg)^{-1} ,\nonumber\\
	& \hspace{-4cm}\bigg(x_3 - V^{\{23\}}\left(x_2^{-1} , x_3^{-1}\right) + V^{\{24\mid 3\}}\left[\left\{V^{\{23\}}\left(x_2^{-1},x_3^{-1}\right)-x_3\right\}^{-1},\left\{V^{\{34\}}\left(x_3^{-1},x_4^{-1}\right)-x_3\right\}^{-1} \right]\bigg)^{-1} \bigg\}.
\label{eqn:-logc1423}
\end{align}
We can now substitute the results~\eqref{eqn:-logf123},~\eqref{eqn:-logf234},~\eqref{eqn:-logf23}~and~\eqref{eqn:-logc1423} into equation~\eqref{eqn:-logf_4d} to obtain the gauge function for this four-dimensional vine copula as
\begin{align*}
	g(\bm{x}) = &V^{\{23\}}\left(x_2^{-1} , x_3^{-1}\right) \\
	&+V^{\{14|23\}}\bigg\{\bigg(x_2 - V^{\{23\}}\left(x_2^{-1} , x_3^{-1}\right)+ V^{\{13\mid 2\}}\left[\left\{V^{\{12\}}\left(x_1^{-1},x_2^{-1}\right)-x_2\right\}^{-1},\left\{V^{\{23\}}\left(x_2^{-1},x_3^{-1}\right)-x_2\right\}^{-1} \right]\bigg)^{-1} ,\\
	&\hspace{2cm} \bigg(x_3 - V^{\{23\}}\left(x_2^{-1} , x_3^{-1}\right)+  V^{\{24\mid 3\}}\left[\left\{V^{\{23\}}\left(x_2^{-1},x_3^{-1}\right)-x_3\right\}^{-1},\left\{V^{\{34\}}\left(x_3^{-1},x_4^{-1}\right)-x_3\right\}^{-1} \right]\bigg)^{-1} \bigg\}.
\end{align*}
This gauge function can be written more simply in terms of the gauge functions of the three sub-vines highlighted in Fig.~\ref{fig:4DVineSeparated} and the exponent measure corresponding to the pair copula in tree $T_3$. That is,
\begin{align*}
g(\bm{x}) &= g_{\{2,3\}}(x_2,x_3)+ V^{\{14|23\}}\left\{\frac{1}{g_{\{1,2,3\}}(x_1,x_2,x_3) - g_{\{2,3\}}(x_2,x_3)},\frac{1}{g_{\{2,3,4\}}(x_2,x_3,x_4) - g_{\{2,3\}}(x_2,x_3)}\right\}.
\end{align*}

\section{Properties of inverted extreme value copulas}\label{subsec:IEVcopProperties}
\begin{lemma}\label{lemma:copulaDensity}
For the density $c(u,v)$ of an inverted extreme value copula, of the form~\eqref{eqn:IEVcopuladensity}, if $u=1-a_1e^{-b_1t}\{1+o(1)\}$ and $v=1-a_2e^{-b_2t}\{1+o(1)\}$ for some $a_1,a_2,b_1,b_2>0$, then as $t\rightarrow\infty$,
\[
	-\ln c(u,v)\sim t\left\{-b_1-b_2+V\left(1/b_1,1/b_2\right)\right\}.
\]
\end{lemma}
\begin{proof}
\begin{align*}
	-\ln c(u,v) &= \ln(1-u) + \ln(1-v) + 2\ln\{-\ln(1-u)\} + 2\ln\{-\ln(1-v)\}\\
	&~~~+V\left\{\frac{-1}{\ln(1-u)},\frac{-1}{\ln(1-v)}\right\}-\ln\bigg[V_1\left\{\frac{-1}{\ln(1-u)},\frac{-1}{\ln(1-v)}\right\}\\
	&~~~\hspace{1cm}\cdot V_2\left\{\frac{-1}{\ln(1-u)},\frac{-1}{\ln(1-v)}\right\}-V_{12}\left\{\frac{-1}{\ln(1-u)},\frac{-1}{\ln(1-v)}\right\}\bigg]\\
	&= \ln a_1 + \ln a_2 - (b_1+b_2)t + 2\ln\{1+o(1)\}\\
	&~~~+2\ln[-\ln a_1 + b_1t - \ln\{1+o(1)\}] +2\ln[-\ln a_2 + b_2t - \ln\{1+o(1)\}]\\
	&~~~+tV\left[\frac{1}{-t^{-1}\ln a_1 + b_1 - t^{-1}\ln\{1+o(1)\}},\frac{1}{-t^{-1}\ln a_2 + b_2 - t^{-1}\ln\{1+o(1)\}}\right]\\
	&~~~-\ln\Bigg(t^4V_1\left[\frac{1}{-t^{-1}\ln a_1 + b_1 - t^{-1}\ln\{1+o(1)\}},\frac{1}{-t^{-1}\ln a_2 + b_2 - t^{-1}\ln\{1+o(1)\}}\right]\\
	&~~~\hspace{1cm}\cdot V_2\left[\frac{1}{-t^{-1}\ln a_1 + b_1 - t^{-1}\ln\{1+o(1)\}},\frac{1}{-t^{-1}\ln a_2 + b_2 - t^{-1}\ln\{1+o(1)\}}\right]\\
	&~~~-t^3V_{12}\left[\frac{1}{-t^{-1}\ln a_1 + b_1 - t^{-1}\ln\{1+o(1)\}},\frac{1}{-t^{-1}\ln a_2 + b_2 - t^{-1}\ln\{1+o(1)\}}\right]\Bigg)\\
&\sim t\left\{-b_1-b_2 + V\left(1/b_1,1/b_2\right)\right\}.
\end{align*}
\end{proof}

\begin{lemma}\label{lemma:conditional}
For the conditional distribution function $F(u\mid v)$ of an inverted extreme value copula, of the form~\eqref{eqn:conditional}, if $u=1-a_1e^{-b_1t}\{1+o(1)\}$ and $v=1-a_2e^{-b_2t}\{1+o(1)\}$ for some $a_1,a_2,b_1,b_2>0$, then as $t\rightarrow\infty$,
\[
	F(u\mid v) = 1 - a\exp\left[-t\left\{V(1/b_1,1/b_2)-b_2\right\}\right]\{1+o(1)\},
\]
for some $a>0$.
\end{lemma}
\begin{proof}
\begin{align*}
	&F(u\mid v) = 1 + \left(\frac{1}{1-v}\right)\left\{-\ln(1-v)\right\}^{-2}V_2\left\{\frac{-1}{\ln(1-u)},\frac{-1}{\ln(1-v)}\right\}\exp\left[-V\left\{\frac{-1}{\ln(1-u)},\frac{-1}{\ln(1-v)}\right\}\right]\\
	&=1 + \frac{e^{b_2t}}{a_2\{1+o(1)\}}\left[-\ln a_2 + b_2t - \ln\{1+o(1)\}\right]^{-2}V_2\left[\frac{1}{-\ln a_1 + b_1t -\ln\{1+o(1)\}},\frac{1}{-\ln a_2 + b_2t -\ln\{1+o(1)\}}\right]\\
	&\hspace{1cm} \cdot \exp\left(-V\left[\frac{1}{-\ln a_1 + b_1t -\ln\{1+o(1)\}},\frac{1}{-\ln a_2 + b_2t -\ln\{1+o(1)\}}\right]\right)\\
	&=1 + \frac{e^{b_2t}}{a_2\{1+o(1)\}}\left[-t^{-1}\ln a_2 + b_2 - t^{-1}\ln\{1+o(1)\}\right]^{-2}\\
	&\hspace{1cm} \cdot V_2\left[\frac{1}{-t^{-1}\ln a_1 + b_1 -t^{-1}\ln\{1+o(1)\}},\frac{1}{-t^{-1}\ln a_2 + b_2 -t^{-1}\ln\{1+o(1)\}}\right]\\
	&\hspace{1cm} \cdot \exp\left(-tV\left[\frac{1}{-t^{-1}\ln a_1 + b_1 -t^{-1}\ln\{1+o(1)\}},\frac{1}{-t^{-1}\ln a_2 + b_2 -t^{-1}\ln\{1+o(1)\}}\right]\right)\\
	&= 1 - a\exp\left[-t\left\{V(1/b_1,1/b_2)-b_2\right\}\right]\{1+o(1)\},
\end{align*}
for $a = -a_2^{-1}V_2(b_2/b_1,1)\exp\left\{-\ln a_1V_1\left(1,b_1/b_2\right) -\ln a_2V_2\left(b_2/b_1,1\right)\right\}$.
\end{proof}

\section{Proof of result~\eqref{eqn:DvineEtaD}}\label{app:DvineEtaD_proof}
\begin{prop}\label{prop:dVineresult}
For any even value of $n\geq 4$,
\[
\left\{2^\alpha\sum_{k=1}^{n/2}(2^\alpha-1)^{2(k-1)}\right\} - 1 - \left\{2^\alpha\sum_{k=1}^{(n-2)/2}(2^\alpha-1)^{2(k-1)+1}\right\} = (2^\alpha-1)^{n-1}.
\]
\end{prop}
\begin{proof}
We first show this is true for $n=4$. We have
\begin{align*}
\left\{2^\alpha\sum_{k=1}^{2}(2^\alpha-1)^{2(k-1)}\right\} - 1 - \left\{2^\alpha\sum_{k=1}^{1}(2^\alpha-1)^{2(k-1)+1}\right\}&=2^\alpha\left\{(2^\alpha-1)^0+(2^\alpha-1)^2-(2^\alpha-1)^1\right\}-1\\
&=2^\alpha - 1 + 2^\alpha(2^\alpha-1)\left\{(2^\alpha-1)-1\right\}\\
&=(2^\alpha-1)\left\{-2^\alpha+1+2^\alpha(2^\alpha-1)\right\}\\
&=(2^\alpha-1)^3.
\end{align*}
Now we assume the result holds for some $n=2m$ with $m\in\mathbb{Z}$ and $m\geq 2$, and show the result also holds for $n=2(m+1)$. We have
\begin{align*}
\left\{2^\alpha\sum_{k=1}^{m+1}(2^\alpha-1)^{2(k-1)}\right\} &- 1 - \left\{2^\alpha\sum_{k=1}^{m}(2^\alpha-1)^{2(k-1)+1}\right\}\\
&= \left\{2^\alpha\sum_{k=1}^{m}(2^\alpha-1)^{2(k-1)}\right\} - 1 - \left\{2^\alpha\sum_{k=1}^{m-1}(2^\alpha-1)^{2(k-1)+1}\right\} + 2^\alpha\left\{(2^\alpha-1)^{2m}-(2^\alpha-1)^{2m-1}\right\}\\
&=(2^\alpha-1)^{2m-1}+ 2^\alpha\left\{(2^\alpha-1)^{2m}-(2^\alpha-1)^{2m-1}\right\}\\
&=(2^\alpha-1)^{2m-1}\left\{1-2^\alpha+2^\alpha(2^\alpha-1)\right\}\\
&=(2^\alpha-1)^{2m+1}.
\end{align*}
Hence the result is proved by induction.
\end{proof}
\noindent To prove result \eqref{eqn:DvineEtaD}, we first show that it holds for $d=3$ and $d=4$. In the former case, we have 
\begin{align*}
	\eta_\mathcal{D} &= \left\{1 + 2^\alpha\sum_{k=1}^1 (2^\alpha-1)^{2(k-1)+1}\right\}^{-1}= \left\{1 + 2^\alpha(2^\alpha-1)\right\}^{-1},
\end{align*}
as obtained by setting $\beta=\gamma=\alpha$ in equation~\eqref{eqn:ilogilogilogGauge}, and in the latter case we have
\begin{align*}
	\eta_\mathcal{D} &= \left\{2^\alpha\sum_{k=1}^2 (2^\alpha-1)^{2(k-1)}\right\}^{-1}= \left\{2^\alpha+2^\alpha(2^\alpha-1)^2\right\}^{-1},
\end{align*}
as achieved by exploiting equation~\eqref{eqn:DvineEtaD.1}. We now assume the result holds for some $d=n-1$ and $d=n$ being odd and even, respectively, and show that it also holds for $d=n+1,n+2$. To make the notation clearer in the proof, we let $\eta_{\mathcal{D},(n)}$ denote the value of $\eta_\mathcal{D}$ for an $n$-dimensional $D$-vine copula with inverted logistic components with parameter $\alpha$. From equation~\eqref{eqn:DvineEtaD.1}, we have
\[
	\eta_{\mathcal{D},(n+1)} = \left\{ \eta_{\mathcal{D},(n-1)}^{-1} + 2^\alpha\left(\eta_{\mathcal{D},(n)}^{-1} -\eta_{\mathcal{D},(n-1)}^{-1}\right)\right\}^{-1};
\]
\[
	\eta_{\mathcal{D},(n+2)} = \left\{ \eta_{\mathcal{D},(n)}^{-1} + 2^\alpha\left(\eta_{\mathcal{D},(n+1)}^{-1} -\eta_{\mathcal{D},(n)}^{-1}\right)\right\}^{-1}.
\]
Under the assumption that~\eqref{eqn:DvineEtaD} holds for $d=n-1,n$, we have
\begin{align*}
\eta_{\mathcal{D},(n+1)} &= \Bigg[ 1 + 2^\alpha\sum_{k=1}^{(n-2)/2} \left(2^\alpha-1\right)^{2(k-1)+1} + 2^\alpha\left\{2^\alpha\sum_{k=1}^{n/2}\left(2^\alpha-1\right)^{2(k-1)} -1 - 2^\alpha\sum_{k=1}^{(n-2)/2} \left(2^\alpha-1\right)^{2(k-1)+1} \right\}\Bigg]^{-1}\\
&= \left\{ 1 + 2^\alpha\sum_{k=1}^{(n-2)/2} \left(2^\alpha-1\right)^{2(k-1)+1}+ 2^\alpha(2^\alpha-1)^{n-1}\right\}^{-1}~~~\text{(by Proposition~\ref{prop:dVineresult})}\\
&= \left\{ 1 + 2^\alpha\sum_{k=1}^{n/2} \left(2^\alpha-1\right)^{2(k-1)+1}\right\}^{-1}.
\end{align*}
Further, we see that
\begin{align*}
\eta_{\mathcal{D},(n+2)} &= \Bigg[2^\alpha\sum_{k=1}^{n/2}\left(2^\alpha-1\right)^{2(k-1)}+2^\alpha\left\{1 + 2^\alpha\sum_{k=1}^{n/2} \left(2^\alpha-1\right)^{2(k-1)+1} - 2^\alpha\sum_{k=1}^{n/2}\left(2^\alpha-1\right)^{2(k-1)}\right\}\Bigg]^{-1}\\
&=\left[2^\alpha\sum_{k=1}^{n/2}\left(2^\alpha-1\right)^{2(k-1)}+2^\alpha\left\{1+2^\alpha(2^\alpha-2)\sum_{k=1}^{n/2}\left(2^\alpha-1\right)^{2(k-1)}\right\}
\right]^{-1}\\
&=\left(2^\alpha\sum_{k=1}^{n/2}\left(2^\alpha-1\right)^{2(k-1)}+2^\alpha\left[1+2^\alpha(2^\alpha-2)\sum_{k=0}^{n/2-1}\left\{\left(2^\alpha-1\right)^{2}\right\}^k\right]\right)^{-1}\\
&=\left[2^\alpha\sum_{k=1}^{n/2}\left(2^\alpha-1\right)^{2(k-1)}+2^\alpha\left\{1+2^\alpha(2^\alpha-2)\frac{1-(2^\alpha-1)^n}{1-(2^\alpha-1)^2}\right\}
\right]^{-1}\\
&=\left[2^\alpha\sum_{k=1}^{n/2}\left(2^\alpha-1\right)^{2(k-1)}+2^\alpha\left\{1+2^\alpha(2^\alpha-2)\frac{1-(2^\alpha-1)^n}{2^\alpha(2-2^\alpha)}\right\}
\right]^{-1}\\
&=\left\{2^\alpha\sum_{k=1}^{n/2}\left(2^\alpha-1\right)^{2(k-1)}+2^\alpha\left(2^\alpha-1\right)^n
\right\}^{-1}\\
&=\left\{2^\alpha\sum_{k=1}^{(n+2)/2}\left(2^\alpha-1\right)^{2(k-1)}\right\}^{-1}.
\end{align*}
Hence result~\eqref{eqn:DvineEtaD} is proved by induction.\\

Result~\eqref{eqn:DvineEtaD} can be simplified further by exploiting sums of geometric series. For the case where $d$ is odd, we have
\begin{align*}
\left\{1 + 2^\alpha\sum_{k=1}^{(d-1)/2} \left(2^\alpha-1\right)^{2(k-1)+1}\right\}^{-1}
= \left[1 + 2^\alpha(2^\alpha-1)\sum_{k=0}^{\frac{d-1}{2} -1} \left\{\left(2^\alpha-1\right)^2\right\}^k\right]^{-1}&= \left[1 + 2^\alpha(2^\alpha-1)\left\{\frac{1-(2^\alpha-1)^{d-1}}{1-(2^\alpha-1)^2}\right\}\right]^{-1}\\
&=\left[1 + \frac{2^\alpha-1}{2-2^\alpha}\left\{1-(2^\alpha-1)^{d-1}\right\}\right]^{-1},
\end{align*}
and for $d$ even,
\begin{align*}
&\left\{ 2^\alpha\sum_{k=1}^{d/2}\left(2^\alpha-1\right)^{2(k-1)} \right\}^{-1}
=\left[ 2^\alpha\sum_{k=0}^{\frac{d}{2}-1}\left\{\left(2^\alpha-1\right)^2\right\}^k \right]^{-1}=\left[ 2^\alpha\left\{\frac{1-(2^\alpha-1)^d}{1-(2^\alpha-1)^2}\right\} \right]^{-1}
= \left[\frac{1}{2-2^\alpha}\left\{1-(2^\alpha-1)^d\right\}\right]^{-1}.
\end{align*}

\section{Properties of extreme value copulas}\label{SM:evProperties}
\subsection{Some properties of the exponent measure}
\begin{prop}\label{prop:V2regvar(a)}
If the exponent measure $V$ has corresponding spectral density $h(w)$ placing no mass on $\{0\}$, and with $h(w)\sim c_2w^{s_2}$ as $w\searrow 0$, for some $c_2\in\mathbb{R}$ and $s_2>-1$, then
\[
	V(r,1) = r^{-1} + \frac{2c_2}{(s_2+1)(s_2+2)}r^{s_2+1}\{1+o(1)\};
\]
\[
	-V_2(r,1)\sim \frac{2c_2}{(s_2+1)}r^{s_2+1}, ~~~ \text{as $r\rightarrow 0$}.
\]
\end{prop}
\begin{proof}
By the definition of the exponent measure,
\begin{align*}
V(x_1,x_2) = \frac{2}{x_1}\int_{0}^{1}wdH(w) - \frac{2}{x_1}\int^{\frac{x_1}{x_1+x_2}}_{0}wh(w)dw + \frac{2}{x_2}\int^{\frac{x_1}{x_1+x_2}}_{0}(1-w)h(w)dw.
\end{align*}
For $x_1\rightarrow 0$ and $x_2=O(1)$, by Karamata's theorem, we have
\begin{align}
	V(x_1,x_2) &= \frac{1}{x_1} - \frac{2c_{2}}{x_1(s_2+2)}\left(\frac{x_1}{x_1+x_2}\right)^{s_{2}+2} \{1+o(1)\} + \frac{2c_2}{x_2(s_{2}+1)}\left(\frac{x_1}{x_1+x_2}\right)^{s_{2}+1} \{1+o(1)\}\nonumber\\
		&= \frac{1}{x_1} + 2c_{2}\left(\frac{x_1}{x_1+x_2}\right)^{s_{2}+1}\left\{\frac{1}{x_2(s_2+1)}-\frac{1}{(s_2+2)(x_1+x_2)}\right\}  \{1+o(1)\}\nonumber\\
		&=\frac{1}{x_1} + \frac{2c_2x_1^{s_2+1}x_2^{-(s_2+2)}}{(s_2+1)(s_2+2)}\{1+o(1)\}.
\label{eqn:karamata(a)}
\end{align}
Hence, we have
\[
	V(r,1) = \frac{1}{r} + \frac{2c_2}{(s_2+1)(s_2+2)}r^{s_2+1}\{1+o(1)\}, ~~~\text{as $r\rightarrow 0$}.
\]
Moreover, differentiating expression~\eqref{eqn:karamata(a)} with respect to $x_2$ gives
\[
	V_2(x_1,x_2) = -\frac{2c_2x_1^{s_2+1}x_2^{-(s_2+3)}}{(s_2+1)}\{1+o(1)\},
\]
and we can therefore infer that
\[
	-V_2(r,1) = \frac{2c_{2}}{(s_{2}+1)}r^{s_2+1}\{1+o(1)\},~~~\text{as $r\rightarrow 0$}.
\]
\end{proof}

\begin{prop}\label{prop:V2regvar(b)}
If the exponent measure $V$ has corresponding spectral density $h(w)$ placing no mass on $\{1\}$, and with $h(w)\sim c_1(1-w)^{s_1}$ as $w\nearrow 1$, for some $c_1\in\mathbb{R}$ and $s_1>-1$, then 
\[
	V(r,1) = 1 + \frac{2c_1}{(s_1+1)(s_1+2)}r^{-(s_1+2)}\{1+o(1)\};
\]
\[
	-V_2(r,1) = 1 - \frac{2c_{1}}{(s_{1}+2)}r^{-(s_1+2)}\{1+o(1)\},~~~\text{as $r\rightarrow\infty$}.
\]
\end{prop}
\begin{proof}
By the definition of the exponent measure,
\begin{align*}
V(x_1,x_2) = \frac{2}{x_2}\int_{0}^{1}(1-w)dH(w) - \frac{2}{x_2}\int_{\frac{x_1}{x_1+x_2}}^{1}(1-w)h(w)dw + \frac{2}{x_1}\int_{\frac{x_1}{x_1+x_2}}^{1}wh(w)dw.
\end{align*}
For $x_1\rightarrow\infty$ and $x_2=o(x_1)$, by Karamata's theorem, we have
\begin{align}
	V(x_1,x_2) &= \frac{1}{x_2} - \frac{2c_{1}}{x_2(s_{1}+2)}\left(\frac{x_2}{x_1+x_2}\right)^{s_{1}+2} \{1+o(1)\}+ \frac{2c_{1}}{x_1(s_{1}+1)}\left(\frac{x_2}{x_1+x_2}\right)^{s_{1}+1} \{1+o(1)\}\nonumber\\
		&= \frac{1}{x_2} + 2c_{1}\left(\frac{x_2}{x_1+x_2}\right)^{s_{1}+1}\left\{\frac{1}{x_1(s_{1}+1)}-\frac{1}{(s_{1}+2)(x_1+x_2)}\right\}  \{1+o(1)\}\nonumber\\
		&=\frac{1}{x_2} + \frac{2c_{1}x_2^{s_{1}+1}x_1^{-(s_{1}+2)}}{(s_{1}+1)(s_{1}+2)}\{1+o(1)\}.
\label{eqn:karamata(b)}
\end{align}
Hence, we have
\[
	V(r,1) = 1 + \frac{2c_{1}}{(s_{1}+1)(s_{1}+2)}r^{-(s_{1}+2)}\{1+o(1)\},~~~\text{as $r\rightarrow\infty$}.
\]
Moreover, differentiating expression~\eqref{eqn:karamata(b)} with respect to $x_2$ gives
\[
	V_2(x_1,x_2) = -\frac{1}{x_2^2} + \frac{2c_{1}x_2^{s_{1}}x_1^{-(s_{1}+2)}}{(s_{1}+2)}\{1+o(1)\},
\]
and we can therefore infer that
\[
	-V_2(r,1) = 1 - \frac{2c_{1}}{(s_{1}+2)}r^{-(s_{1}+2)}\{1+o(1)\},~~~\text{as $r\rightarrow\infty$}.
\]
\end{proof}

\subsection{Asymptotic behaviour of $-\ln c\{F_1(tx_1),F_2(tx_2)\}$}\label{subsec:-logc_EV}
For $V$ representing an exponent measure defined as in equation~\eqref{eqn:exponentMeasure}, an extreme value copula has the form
\[
	C(u,v) = \exp\left\{-V\left(\frac{-1}{\ln u},\frac{-1}{\ln v}\right)\right\}.
\]
Differentiating $C(u,v)$ with respect to the second argument, we have
\[
	F(u\mid v) = \frac{\partial C(u,v)}{\partial u} = -\frac{1}{v}(-\ln v)^{-2}V_2\left(\frac{-1}{\ln u},\frac{-1}{\ln v}\right)\exp\left\{-V\left(\frac{-1}{\ln u},\frac{-1}{\ln v}\right)\right\},
\]
and differentiating $C(u,v)$ with respect to both arguments yields the copula density
\begin{align*}
	c(u,v) = \frac{\partial^2 C(u,v)}{\partial u\partial v} =& \frac{1}{uv}(-\ln u)^{-2}(-\ln v)^{-2}\exp\left\{-V\left(\frac{-1}{\ln u},\frac{-1}{\ln v}\right)\right\}\left\{V_1\left(\frac{-1}{\ln u},\frac{-1}{\ln v}\right)V_2\left(\frac{-1}{\ln u},\frac{-1}{\ln v}\right)-V_{12}\left(\frac{-1}{\ln u},\frac{-1}{\ln v}\right)\right\},
\end{align*}
with $V_1$, $V_2$ and $V_{12}$ denoting the derivatives of the exponent measure with respect to the first, second, and both components, respectively.

For this class of models, setting exponential margins with $F_i(tx_i)=1-e^{-tx_i}$, $i=1,2$, we have
\begin{align*}
	&-\ln c\{F_1(tx_1),F_2(tx_2)\} \\
	& \hspace{0.5cm}=\ln\left(1-e^{-tx_1}\right) + \ln\left(1-e^{-tx_2}\right)+ 2\ln\left\{-\ln\left(1-e^{-tx_1}\right)\right\} + 2\ln\left\{-\ln\left(1-e^{-tx_2}\right)\right\}+V\left\{\frac{-1}{\ln\left(1-e^{-tx_1}\right)},\frac{-1}{\ln\left(1-e^{-tx_2}\right)}\right\}\\
	&\hspace{1cm}-\ln \Bigg[V_1\left\{\frac{-1}{\ln\left(1-e^{-tx_1}\right)},\frac{-1}{\ln\left(1-e^{-tx_2}\right)}\right\}V_2\left\{\frac{-1}{\ln\left(1-e^{-tx_1}\right)},\frac{-1}{\ln\left(1-e^{-tx_2}\right)}\right\}-V_{12}\left\{\frac{-1}{\ln\left(1-e^{-tx_1}\right)},\frac{-1}{\ln\left(1-e^{-tx_2}\right)}\right\}\Bigg].
\end{align*}
We impose the assumption that the corresponding spectral density $h(w)$ places no mass on $\{0\}$ or $\{1\}$ and has regularly varying tails. In particular, that $h(w)\sim c_1(1-w)^{s_1}$ as $w\nearrow 1$ and $h(w)\sim c_2w^{s_2}$ as $w\searrow 0$, for some $c_1,c_2\in\mathbb{R}$ and $s_1,s_2>-1$. By a result from \cite{ColesAndTawn1991}, we have 
\[
-V_{12}(x_{1},x_{2}) = \frac{2}{(x_{1}+x_{2})^{3}} h\left(\frac{x_1}{x_1+x_2}\right),
\]
and can therefore deduce that
\begin{align}
-V_{12}(1,r) \sim 2c_1r^{s_1},~~~\text{and}~~~-V_{12}(r,1) \sim 2c_2r^{s_2},~~~\text{as $r\rightarrow 0$}.
\label{eqn:V12results}
\end{align}
To investigate the behaviour of $-\ln c\{F_1(tx_1),F_2(tx_2)\}$ as $t\rightarrow\infty$, we consider three cases; $x_1<x_2$, $x_1=x_2$, and $x_1>x_2$.
\paragraph{Case 1: $x_1<x_2$}
{\small \begin{align*}
	-&\ln c\{F_1(tx_1),F_2(tx_2)\}\\
	=& -e^{-tx_1} -e^{-tx_2}+ 2\ln(e^{-tx_1}) + e^{-tx_1} + 2\ln(e^{-tx_2}) + e^{-tx_2} + O(e^{-2tx_1}) + O(e^{-2tx_2})+V\left\{\frac{1}{e^{-tx_1} + O(e^{-2tx_1})},\frac{1}{e^{-tx_2} + O(e^{-2tx_2})}\right\}\\
	&-\ln\bigg[V_1\left\{\frac{1}{e^{-tx_1} + O(e^{-2tx_1})},\frac{1}{e^{-tx_2} + O(e^{-2tx_2})}\right\} V_2\left\{\frac{1}{e^{-tx_1} + O(e^{-2tx_1})},\frac{1}{e^{-tx_2} + O(e^{-2tx_2})}\right\}\\
	&\hspace{4cm}- V_{12}\left\{\frac{1}{e^{-tx_1} + O(e^{-2tx_1})},\frac{1}{e^{-tx_2} + O(e^{-2tx_2})}\right\}\bigg]\\
	=&-2t(x_1+x_2) + O(e^{-2tx_1}) + O(e^{-2tx_2}) + V\left\{\frac{e^{tx_1}}{1 + O(e^{-tx_1})},\frac{e^{tx_2}}{1 + O(e^{-tx_2})}\right\}\\
	&-\ln\bigg[e^{-4tx_2}V_1\left\{\frac{e^{t(x_1-x_2)}}{1 + O(e^{-tx_1})},\frac{1}{1 + O(e^{-tx_2})}\right\} V_2\left\{\frac{e^{t(x_1-x_2)}}{1 + O(e^{-tx_1})},\frac{1}{1 + O(e^{-tx_2})}\right\}- e^{-3tx_2}V_{12}\left\{\frac{e^{t(x_1-x_2)}}{1 + O(e^{-tx_1})},\frac{1}{1 + O(e^{-tx_2})}\right\}\bigg]\\
	\sim& -2t(x_1+x_2) - \ln\left(e^{-3tx_2}\right) - \ln\left\{e^{ts_2(x_1-x_2)}\right\}~~~\text{by result~\eqref{eqn:V12results}}\\
	=& t\left\{(1+s_2)x_2 - (2+s_2)x_1\right\}
\end{align*}}

\paragraph{Case 2: $x_1=x_2$}
\begin{align*}
	-\ln c\{F_1(tx_1),F_2(tx_2)\} &= 2\ln(1-e^{-tx_1}) + 4\ln\{-\ln(1-e^{-tx_1})\}- \ln(1-e^{-tx_1})V(1,1)\\
	&~~-\ln\Big[\left\{-\ln(1-e^{-tx_1})\right\}^4 V_1(1,1)V_2(1,1) - \left\{-\ln(1-e^{-tx_1})\right\}^3 V_{12}(1,1) \Big] \\
	&= -4tx_1 -\ln\left\{-V_{12}(1,1)e^{-3tx_1}\right\} + O\left(e^{-tx_1}\right)\\
	&= -tx_1 + O\left(e^{-tx_1}\right)
\end{align*}

\paragraph{Case 3: $x_1>x_2$}
{\small \begin{align*}
	-&\ln c\{F_1(tx_1),F_2(tx_2)\} \\
	=& -e^{-tx_1} -e^{-tx_2} + 2\ln(e^{-tx_1}) + e^{-tx_1} + 2\ln(e^{-tx_2}) + e^{-tx_2} + O(e^{-2tx_1}) + O(e^{-2tx_2})+V\left\{\frac{1}{e^{-tx_1} + O(e^{-2tx_1})},\frac{1}{e^{-tx_2} + O(e^{-2tx_2})}\right\}\\
	&-\ln\bigg[V_1\left\{\frac{1}{e^{-tx_1} + O(e^{-2tx_1})},\frac{1}{e^{-tx_2} + O(e^{-2tx_2})}\right\} V_2\left\{\frac{1}{e^{-tx_1} + O(e^{-2tx_1})},\frac{1}{e^{-tx_2} + O(e^{-2tx_2})}\right\}\\
	&\hspace{4cm}- V_{12}\left\{\frac{1}{e^{-tx_1} + O(e^{-2tx_1})},\frac{1}{e^{-tx_2} + O(e^{-2tx_2})}\right\}\bigg]\\
	=&-2t(x_1+x_2) + O(e^{-2tx_1}) + O(e^{-2tx_2}) + V\left\{\frac{e^{tx_1}}{1 + O(e^{-tx_1})},\frac{e^{tx_2}}{1 + O(e^{-tx_2})}\right\}\\
	&-\ln\bigg[e^{-4tx_1}V_1\left\{\frac{1}{1 + O(e^{-tx_1})},\frac{e^{t(x_2-x_1)}}{1 + O(e^{-tx_2})}\right\} V_2\left\{\frac{1}{1 + O(e^{-tx_1})},\frac{e^{t(x_2-x_1)}}{1 + O(e^{-tx_2})}\right\}- e^{-3tx_1}V_{12}\left\{\frac{1}{1 + O(e^{-tx_1})},\frac{e^{t(x_2-x_1)}}{1 + O(e^{-tx_2})}\right\}\bigg]\\
	\sim& -2t(x_1+x_2) - \ln\left(e^{-3tx_1}\right) - \ln\left\{e^{ts_1(x_2-x_1)}\right\}~~~\text{by result~\eqref{eqn:V12results}}\\
	=&~ t\left\{(1+s_1)x_1 - (2+s_1)x_2\right\}.
\end{align*}}
These three cases can be combined into a single expression, so that as $t\rightarrow\infty$, we have
\begin{align*}
	-\ln c\{F_1(tx_1),F_2(tx_2)\} \sim t\big\{&\left(1+s_1\mathbbm{1}_{\{x_1\geq x_2\}}+s_2\mathbbm{1}_{\{x_1< x_2\}}\right)\max(x_1,x_2)- \left(2+s_1\mathbbm{1}_{\{x_1\geq x_2\}}+s_2\mathbbm{1}_{\{x_1< x_2\}}\right)\min(x_1,x_2)\big\}.
\end{align*}

\subsection{Asymptotic behaviour of $F_{1\mid 2}(tx_1\mid tx_2)$}\label{subsec:F12_EV}
For a bivariate extreme value copula, the conditional distribution function has the form
\begin{align*}
	F_{1\mid 2}(tx_1\mid tx_2) =& -\frac{1}{1-e^{-tx_2}}\{-\ln(1-e^{-tx_2})\}^{-2}V_2\left\{\frac{-1}{\ln(1-e^{-tx_1})},\frac{-1}{\ln(1-e^{-tx_2})}\right\}\exp\left[-V\left\{\frac{-1}{\ln(1-e^{-tx_1})},\frac{-1}{\ln(1-e^{-tx_2})}\right\}\right].
\end{align*}
To investigate the behaviour of $F_{1\mid 2}(tx_1\mid tx_2)$ as $t\rightarrow\infty$, we again consider three cases.
\paragraph{Case 1: $x_1<x_2$}
\begin{align*}
	F_{1\mid 2}(tx_1\mid tx_2) =& -\frac{1}{1-e^{-tx_2}}V_2\left\{\frac{\ln(1-e^{-tx_2})}{\ln(1-e^{-tx_1})},1\right\}\exp\left[\ln\left(1-e^{-tx_2}\right)V\left\{\frac{\ln(1-e^{-tx_2})}{\ln(1-e^{-tx_1})},1\right\}\right]\\
	=& -\left[1+e^{-tx_2}\{1+o(1)\}\right]V_2\left[\exp\{t(x_1-x_2)\}\left\{1+o(1)\right\},1\right]\exp\left[\ln\left(1-e^{-tx_2}\right)V\left\{\exp\{t(x_1-x_2)\}\{1+o(1)\},1\right\}\right]\\
	=& \left[1+e^{-tx_2}\{1+o(1)\}\right]\left[\frac{2c_{2}}{(s_{2}+1)}\exp\left\{t(x_1-x_2)(s_2+1)\right\}\{1+o(1)\}\right]\\
	&\hspace{0.5cm}\exp\left[\ln\left(1-e^{-tx_2}\right)\left\{\exp\{-t(x_1-x_2)\}\{1+o(1)\right\}\right]\hspace{2cm}\text{(by Proposition~\ref{prop:V2regvar(a)})}\\
	=& \left[1+e^{-tx_2}\{1+o(1)\}\right]\left[\frac{2c_{2}}{(s_{2}+1)}\exp\left\{t(x_1-x_2)(s_2+1)\right\}\{1+o(1)\}\right]\left[1-e^{-tx_1}\{1+o(1)\}\right]\\
	=& \frac{2c_{2}}{(s_{2}+1)}\exp\left\{t(x_1-x_2)(s_2+1)\right\}\{1+o(1)\}.
\end{align*}

\paragraph{Case 2: $x_1=x_2$}
\begin{align*}
	F_{1\mid 2}(tx_1\mid tx_2)=& -\frac{1}{1-e^{-tx_1}}V_2(1,1)\exp\left\{\ln(1-e^{-tx_1})V(1,1)\right\}\\
	=&-V_2(1,1)(1-e^{-tx_1})^{V(1,1)-1}\\
	=&-V_2(1,1)\{1+o(1)\}.
\end{align*}

\paragraph{Case 3: $x_1>x_2$}
\begin{align*}
	F_{1\mid 2}(tx_1\mid tx_2) =& -\frac{1}{1-e^{-tx_2}}V_2\left\{\frac{\ln(1-e^{-tx_2})}{\ln(1-e^{-tx_1})},1\right\}\exp\left[\ln\left(1-e^{-tx_2}\right)V\left\{\frac{\ln(1-e^{-tx_2})}{\ln(1-e^{-tx_1})},1\right\}\right]\\
	=& -\frac{1}{1-e^{-tx_2}}V_2\left[\exp\{t(x_1-x_2)\}\left\{1+o(1)\right\},1\right]\exp\left[\ln\left(1-e^{-tx_2}\right)V\left\{\exp\{t(x_1-x_2)\}\{1+o(1)\},1\right\}\right]\\
	=& \frac{1}{1-e^{-tx_2}}\left[1 - \frac{2c_{1}}{(s_{1}+2)}\exp\left\{-t(x_1-x_2)(s_{1}+2)\right\}\{1+o(1)\}\right]\\
	&\hspace{0.5cm}\cdot\exp\left(\ln\left(1-e^{-tx_2}\right)\left[1 + \frac{2c_1}{(s_1+1)(s_1+2)}\exp\{-t(s_1+2)(x_1-x_2)\}\{1+o(1)\}\right]\right)\text{~~~(by Proposition~\ref{prop:V2regvar(b)})}\\
	=& \frac{1-e^{-tx_2}}{1-e^{-tx_2}}\left[1 - \frac{2c_{1}}{(s_{1}+2)}\exp\left\{-t(x_1-x_2)(s_{1}+2)\right\}\{1+o(1)\}\right]\\
	&\hspace{0.5cm}\cdot\left[1 + \frac{2c_1}{(s_1+1)(s_1+2)}\exp\{-t(s_1+2)(x_1-x_2)-tx_2\}\{1+o(1)\}\right]\\
	=& 1 - \frac{2c_{1}}{(s_{1}+2)}\exp\left\{-t(x_1-x_2)(s_{1}+2)\right\}\{1+o(1)\}.
\end{align*}

\section{$-\ln c_{13|2}(u,v)$ for (inverted) extreme value copulas}\label{app:logc132(u,v)}
In Section~\ref{sec:vinesIEVEV}, we aim to investigate cases where the copula density in tree $T_2$ belongs to either the extreme value or inverted extreme value families of distributions. Here, we denote these by $c^{EV}_{13\mid 2}(u,v)$ and $c^{IEV}_{13\mid 2}(u,v)$, respectively. In Section~\ref{SMsec:IEVEVgauges}, we find that we should focus on the conditional distributions $F_{1\mid 2}(tx_1\mid tx_2)$ and $F_{3\mid 2}(tx_3\mid tx_2)$ having three different asymptotic forms. In this section, we therefore consider the behaviour of $-\ln c^{EV}_{13\mid 2}(u,v)$ and $-\ln c^{IEV}_{13\mid 2}(u,v)$ for $u$ and $v$ of the form
\begin{align*}
a\{1+o(1)\}~~~;~~~b_1\exp(-b_2t)\{1+o(1)\}~~~;~~~1-c_1\exp(-c_2t)\{1+o(1)\},
\end{align*}
for $b_2,c_2>0$, using the results that
\begin{align*}
-\ln c^{EV}_{13\mid 2}(u,v) =& \ln u + \ln v + 2\ln\left(-\ln u\right) + 2\ln\left(-\ln v\right)+V^{\{13|2\}}\left(\frac{-1}{\ln u},\frac{-1}{\ln v}\right)\\
&\hspace{-2cm}-\ln\left\{V_1^{\{13|2\}}\left(\frac{-1}{\ln u},\frac{-1}{\ln v}\right)V_2^{\{13|2\}}\left(\frac{-1}{\ln u},\frac{-1}{\ln v}\right)-V^{\{13|2\}}_{12}\left(\frac{-1}{\ln u},\frac{-1}{\ln v}\right)\right\},
\end{align*}
and
\begin{align*}
-\ln c^{IEV}_{13\mid 2}(u,v) =&\ln(1-u) + \ln(1-v) + 2\ln\left\{-\ln (1-u)\right\} + 2\ln\left\{-\ln (1-v)\right\}+V^{\{13|2\}}\left\{\frac{-1}{\ln (1-u)},\frac{-1}{\ln (1-v)}\right\}\\
&-\ln\bigg[V_1^{\{13|2\}}\left\{\frac{-1}{\ln (1-u)},\frac{-1}{\ln (1-v)}\right\}V_2^{\{13|2\}}\left\{\frac{-1}{\ln (1-u)},\frac{-1}{\ln (1-v)}\right\}-V^{\{13|2\}}_{12}\left\{\frac{-1}{\ln (1-u)},\frac{-1}{\ln (1-v)}\right\}\bigg].
\end{align*}
We also have the assumption that the spectral density $h^{\{13|2\}}(w)$ corresponding to the copula in tree $T_2$ places no mass on $\{0\}$ or $\{1\}$, and has $h^{\{13|2\}}(w)\sim c_1^{\{13|2\}}(1-w)^{s_1^{\{13|2\}}}$ as $w\nearrow 1$, and $h^{\{13|2\}}(w)\sim c_2^{\{13|2\}}w^{s_2^{\{13|2\}}}$ as $w\searrow 0$, for some $c_1^{\{13|2\}},c_2^{\{13|2\}}\in\mathbb{R}$ and $s_1^{\{13|2\}},s_2^{\{13|2\}}>-1$. In the following nine cases, we provide results for the asymptotic behaviour of $-\ln c^{EV}_{13\mid 2}(u,v)$ and $-\ln c^{IEV}_{13\mid 2}(u,v)$, as $t\rightarrow\infty$, for $u$ and $v$ taking different combinations of the forms stated above.

\paragraph{Case 1: $u=a_u\{1+o(1)\}$, $v=a_v\{1+o(1)\}$}
\begin{align*}
	-\ln c^{EV}_{13|2}(u,v) = o(t)
\end{align*}
\begin{align*}
	-\ln c^{IEV}_{13|2}(u,v) = o(t)
\end{align*}

\paragraph{Case 2: $u=a_u\{1+o(1)\}$, $v=b_{v,1}\exp(-b_{v,2}t)\{1+o(1)\}$}
\begin{align*}
	-\ln c^{EV}_{13|2}(u,v) &= o(t)
\end{align*}
\begin{align*}
	-\ln c^{IEV}_{13|2}(u,v) \sim t b_{v,2}\left(1 + s_2^{\{13|2\}}\right)
\end{align*}

\paragraph{Case 3: $u=a_u\{1+o(1)\}$, $v=1-c_{v,1}\exp(-c_{v,2}t)\{1+o(1)\}$}
\begin{align*}
	-\ln c^{EV}_{13|2}(u,v) \sim t c_{v,2}\left(1 + s_2^{\{13|2\}}\right)
\end{align*}
\begin{align*}
	-\ln c^{IEV}_{13|2}(u,v) = o(t)
\end{align*}

\paragraph{Case 4: $u=b_{u,1}\exp(-b_{u,2}t)\{1+o(1)\}$, $v=a_v\{1+o(1)\}$}
\begin{align*}
	-\ln c^{EV}_{13|2}(u,v) = o(t)
\end{align*}
\begin{align*}
	-\ln c^{IEV}_{13|2}(u,v) \sim t b_{u,2}\left(1 + s_1^{\{13|2\}}\right)
\end{align*}

\paragraph{Case 5: $u=b_{u,1}\exp(-b_{u,2}t)\{1+o(1)\}$, $v=b_{v,1}\exp(-b_{v,2}t)\{1+o(1)\}$}
\begin{align*}
	-\ln c^{EV}_{13|2}(u,v) \sim t \left\{ -b_{u,2}-b_{v,2} + V^{\{13|2\}}\left(1/b_{u,2},1/b_{v,2}\right) \right\} 
\end{align*}
\begin{align*}
	-\ln c^{IEV}_{13|2}(u,v) \sim &t\Big\{\left(1+s^{\{13|2\}}_1\mathbbm{1}_{\{b_{u,2}\geq b_{v,2}\}}+s^{\{13|2\}}_2\mathbbm{1}_{\{b_{u,2}< b_{v,2}\}}\right)\max(b_{u,2},b_{v,2}) \\&- \left(2+s^{\{13|2\}}_1\mathbbm{1}_{\{b_{u,2}\geq b_{v,2}\}}+s^{\{13|2\}}_2\mathbbm{1}_{\{b_{u,2}< b_{v,2}\}}\right)\min(b_{u,2},b_{v,2})\Big\}
\end{align*}

\paragraph{Case 6: $u=b_{u,1}\exp(-b_{u,2}t)\{1+o(1)\}$, $v=1-c_{v,1}\exp(-c_{v,2}t)\{1+o(1)\}$}
\begin{align*}
	-\ln c^{EV}_{13|2}(u,v) \sim t c_{v,2}\left(1 + s_2^{\{13|2\}}\right)
\end{align*}
\begin{align*}
	-\ln c^{IEV}_{13|2}(u,v) \sim t b_{u,2}\left(1 + s_1^{\{13|2\}}\right)
\end{align*}

\paragraph{Case 7: $u=1-c_{u,1}\exp(-c_{u,2}t)\{1+o(1)\}$, $v=a_v\{1+o(1)\}$}
\begin{align*}
	-\ln c^{EV}_{13|2}(u,v) \sim t c_{u,2}\left(1 + s_1^{\{13|2\}}\right)
\end{align*}
\begin{align*}
	-\ln c^{IEV}_{13|2}(u,v) = o(t)
\end{align*}

\paragraph{Case 8: $u=1-c_{u,1}\exp(-c_{u,2}t)\{1+o(1)\}$, $v=b_{v,1}\exp(-b_{v,2}t)\{1+o(1)\}$}
\begin{align*}
	-\ln c^{EV}_{13|2}(u,v) \sim t c_{u,2}\left(1 + s_1^{\{13|2\}}\right) 
\end{align*}
\begin{align*}
	-\ln c^{IEV}_{13|2}(u,v) \sim t b_{v,2}\left(1 + s_2^{\{13|2\}}\right)
\end{align*}

\paragraph{Case 9: $u=1-c_{u,1}\exp(-c_{u,2}t)\{1+o(1)\}$, $v=1-c_{v,1}\exp(-c_{v,2}t)\{1+o(1)\}$}
\begin{align*}
	-\ln c^{EV}_{13|2}(u,v) \sim &t\Big\{\left(1+s^{\{13|2\}}_1\mathbbm{1}_{\{c_{u,2}\geq c_{v,2}\}}+s^{\{13|2\}}_2\mathbbm{1}_{\{c_{u,2}< c_{v,2}\}}\right)\max(c_{u,2},c_{v,2}) \\&- \left(2+s^{\{13|2\}}_1\mathbbm{1}_{\{c_{u,2}\geq c_{v,2}\}}+s^{\{13|2\}}_2\mathbbm{1}_{\{c_{u,2}< c_{v,2}\}}\right)\min(c_{u,2},c_{v,2})\Big\}
\end{align*}
\begin{align*}
	-\ln c^{IEV}_{13|2}(u,v) \sim t \left\{ -c_{u,2}-c_{v,2} + V^{\{13|2\}}\left(1/c_{u,2},1/c_{v,2}\right) \right\}
\end{align*}

\section{Gauge function calculations for Section~\ref{sec:vinesIEVEV}}\label{SMsec:IEVEVgauges}
In Section~\ref{sec:vinesIEVEV}, we aim to investigate the coefficients $\eta_{\{1,2,3\}}$ and $\eta_{\{1,3\}}$ for trivariate vine copulas constructed from inverted extreme value and extreme value pair copulas. We present the necessary calculations here.

For any trivariate vine copula, $-\ln f(t\bm{x})$ has the form~\eqref{eqn:negLogTriVine}. In our gauge function calculations, we work with variables having exponential margins, so that the first three terms of this expression always satisfy
\begin{align}
	-\ln f_1\left(tx_1\right) -\ln f_2\left(tx_2\right) -\ln f_3\left(tx_3\right) = t\left(x_1+x_2+x_3\right),
\label{eqn:exponentialMarginsAsymptotic}
\end{align}
and we have $F_i(tx_i)=1-e^{-tx_i}$, for $i=1,2,3$. For inverted extreme value copulas, we have shown in equations~\eqref{eqn:negLogTriVine(a)}~and~\eqref{eqn:negLogTriVine(b)} that
\begin{align}
	-\ln c_{12}\left\{F_1(tx_1),F_2(tx_2)\right\} = t\left\{V^{\{12\}}\left(x_1^{-1},x_2^{-1}\right)-x_1-x_2\right\} + O(\ln t)
	\label{eqn:IEVcopula(12)}
\end{align}
and
\begin{align}
	-\ln c_{23}\left\{F_2(tx_2),F_3(tx_3)\right\} = t\left\{V^{\{23\}}\left(x_2^{-1},x_3^{-1}\right)-x_2-x_3\right\} + O(\ln t).
	\label{eqn:IEVcopula(23)}
\end{align}
Recall that in order to investigate the behaviour of the extreme value pair copula components, we impose the condition that the corresponding spectral densities place no mass on $\{0\}$ or $\{1\}$ and have regularly varying tails. Let $h^{\{12\}}(w)$, $h^{\{23\}}(w)$, $h^{\{13|2\}}(w)$ denote the spectral density for each pair copula component. We assume that each of these densities has $h^{\{\cdot\}}(w)\sim c_1^{\{\cdot\}}(1-w)^{s_1^{\{\cdot\}}}$ as $w\nearrow 1$ and $h^{\{\cdot\}}(w)\sim c_2^{\{\cdot\}}w^{s_2^{\{\cdot\}}}$ as $w\searrow 0$, for some $c_1^{\{\cdot\}},c_2^{\{\cdot\}}\in\mathbb{R}$ and $s_1^{\{\cdot\}},s_2^{\{\cdot\}}>-1$. In Section~\ref{subsec:-logc_EV} of the Supplementary Material, we show that
\begin{align}
	-\ln c_{12}\{F_1(tx_1),F_2(tx_2)\} \sim &t\Big\{\left(1+s^{\{12\}}_1\mathbbm{1}_{\{x_1\geq x_2\}}+s^{\{12\}}_2\mathbbm{1}_{\{x_1< x_2\}}\right)\max(x_1,x_2) \\&- \left(2+s^{\{12\}}_1\mathbbm{1}_{\{x_1\geq x_2\}}+s^{\{12\}}_2\mathbbm{1}_{\{x_1< x_2\}}\right)\min(x_1,x_2)\Big\},
	\label{eqn:EVcopula(12)}
\end{align}
and
\begin{align}
	-\ln c_{23}\{F_2(tx_2),F_3(tx_3)\} \sim &t\Big\{\left(1+s^{\{23\}}_1\mathbbm{1}_{\{x_2\geq x_3\}}+s^{\{23\}}_2\mathbbm{1}_{\{x_2< x_3\}}\right)\max(x_2,x_3) \\&- \left(2+s^{\{23\}}_1\mathbbm{1}_{\{x_2\geq x_3\}}+s^{\{23\}}_2\mathbbm{1}_{\{x_2< x_3\}}\right)\min(x_2,x_3)\Big\},
	\label{eqn:EVcopula(23)}
\end{align}
so that the final term of~\eqref{eqn:negLogTriVine}, i.e., $-\ln c_{13|2}\left\{F_{1|2}(tx_1|tx_2),F_{3|2}(tx_3|tx_2)\right\}$, is the only one left for us to consider.

In Section~\ref{subsec:IEVcopProperties} of the Supplementary Material, we showed that for the components in tree $T_1$ being inverted extreme value copulas, we have
\begin{align}
	F_{1\mid 2}(tx_1\mid tx_2) &= 1 + x_2^{-2}V^{\{12\}}_2\left(x_1^{-1},x_2^{-1}\right)\exp\left[t\left\{x_2 - V^{\{12\}}\left(x_1^{-1}, x_2^{-1}\right)\right\}\right]\nonumber\\
	&= 1 - a_{1|2}\exp\left[-t\left\{V^{\{12\}}\left(x_1^{-1}, x_2^{-1}\right)-x_2\right\}\right];
	\label{eqn:Fconditional(12)}
\end{align}
\begin{align}
	F_{3\mid 2}(tx_3\mid tx_2) &= 1 + x_2^{-2}V^{\{23\}}_1\left(x_2^{-1},x_3^{-1}\right)\exp\left[t\left\{x_2 - V^{\{23\}}\left(x_2^{-1}, x_3^{-1}\right)\right\}\right]\nonumber\\
	&= 1 - a_{3|2}\exp\left[-t\left\{V^{\{23\}}\left(x_2^{-1}, x_3^{-1}\right)-x_2\right\}\right],
	\label{eqn:Fconditional(23)}
\end{align}
for $a_{1|2} = -x_2^{-2}V^{\{12\}}_2\left(x_1^{-1},x_2^{-1}\right)$ and $a_{3|2}=- x_2^{-2}V^{\{23\}}_1\left(x_2^{-1},x_3^{-1}\right)$. We consider the extreme value case in Section~\ref{subsec:F12_EV} of the Supplementary Material. Placing the same conditions on the spectral density as for the $-\ln c_{12}$ and $-\ln c_{23}$ calculations, we obtain
\begin{align}
F_{1\mid 2}(tx_1\mid tx_2) =
	\begin{cases}
        \frac{2c^{\{12\}}_{2}}{(s^{\{12\}}_{2}+1)}\exp\left\{t(x_1-x_2)(s^{\{12\}}_2+1)\right\}\{1+o(1)\}, & x_1<x_2,\\
       -V^{\{12\}}_2(1,1)\{1+o(1)\}, & x_1=x_2,\\
        1 - \frac{2c^{\{12\}}_{1}}{(s^{\{12\}}_{1}+2)}\exp\left\{-t(x_1-x_2)(s^{\{12\}}_{1}+2)\right\}\{1+o(1)\}, & x_1>x_2,
	\end{cases}
	\label{eqn:FconditionalEV(12)}
\end{align}
and
\begin{align}
F_{3\mid 2}(tx_3\mid tx_2) =
	\begin{cases}
        \frac{2c^{\{23\}}_{1}}{(s^{\{23\}}_{1}+1)}\exp\left\{t(x_3-x_2)(s^{\{23\}}_1+1)\right\}\{1+o(1)\}, & x_3<x_2,\\
       -V^{\{23\}}_1(1,1)\{1+o(1)\}, & x_3=x_2,\\
        1 - \frac{2c^{\{23\}}_{2}}{(s^{\{23\}}_{2}+2)}\exp\left\{-t(x_3-x_2)(s^{\{23\}}_{2}+2)\right\}\{1+o(1)\}, & x_3>x_2.
	\end{cases}
	\label{eqn:FconditionalEV(23)}
\end{align}
So, as $t\rightarrow\infty$, the conditional distributions of the extreme value copula in equation~\eqref{eqn:FconditionalEV(12)} tends towards either $0$, $-V_2^{\{12\}}(1,1)$ or $1$, and expression~\eqref{eqn:FconditionalEV(23)} tends towards $0$, $-V_1^{\{23\}}(1,1)$ or $1$. We must therefore consider three different cases in our gauge function calculations for any extreme value copula in tree $T_1$. Based on the asymptotic forms of $F_{1\mid 2}(tx_1\mid tx_2)$ and $F_{3\mid 2}(tx_3\mid tx_2)$ for (inverted) extreme value copulas, we focus on investigating $-\ln c_{13\mid 2}(u,v)$ for $u$ and $v$ of the form
\begin{align*}
a\{1+o(1)\}~~~;~~~b_1\exp(-b_2t)\{1+o(1)\}~~~;~~~1-c_1\exp(-c_2t)\{1+o(1)\},
\end{align*}
for $b_2,c_2>0$. We provide these results in Section~\ref{app:logc132(u,v)} of the Supplementary Material for all nine combinations of the asymptotic forms of $u$ and $v$, and for $c_{13\mid 2}(u,v)$ being either an extreme value or inverted extreme value copula density.

Together, these results provide all the necessary information to calculate the gauge functions. We first demonstrate how to combine all these results to obtain the gauge function of a vine copula with two inverted extreme value copulas in tree $T_1$, and an extreme value copula in tree $T_2$, and subsequently find the gauge functions for the remaining cases.

\subsection{Inverted extreme value copulas in $T_1$; extreme value copula in $T_2$}
For this case, we use results \eqref{eqn:exponentialMarginsAsymptotic},~\eqref{eqn:IEVcopula(12)}~and~\eqref{eqn:IEVcopula(23)} in equation~\eqref{eqn:negLogTriVine}, which gives the form of $-\ln f(t\bm{x})$ for a trivariate vine copula. We have
\begin{align}
	-\ln f(t\bm{x}) =& t(x_1+x_2+x_3) + t\left\{V^{\{12\}}\left(x_1^{-1},x_2^{-1}\right)-x_1-x_2\right\}+ t\left\{V^{\{23\}}\left(x_2^{-1},x_3^{-1}\right)-x_2-x_3\right\}\nonumber\\
	& ~~~- \ln c_{13|2}\left\{F_{1|2}(tx_1|tx_2),F_{3|2}(tx_3|tx_2)\right\}+ O(\ln t)\nonumber\\
	=&t\left\{V^{\{12\}}\left(x_1^{-1},x_2^{-1}\right) + V^{\{23\}}\left(x_2^{-1},x_3^{-1}\right)-x_2\right\}- \ln c_{13|2}\left\{F_{1|2}(tx_1|tx_2),F_{3|2}(tx_3|tx_2)\right\}+ O(\ln t).
\label{eqn:IEVIEVEV_fiveterms}
\end{align}
From equations~\eqref{eqn:Fconditional(12)}~and~\eqref{eqn:Fconditional(23)}, we see that $F_{1|2}(tx_1\mid tx_2) = 1 - a_{1|2}\exp\{-b_{1|2}t\}$ and $F_{3|2}(tx_3\mid tx_2) = 1 - a_{3|2}\exp\{-b_{3|2}t\}$, for $b_{1|2} = V^{\{12\}}(x_1^{-1},x_2^{-1})-x_2$ and $b_{3|2} = V^{\{23\}}(x_2^{-1},x_3^{-1})-x_2$. Using results from case~9 of Section~\ref{app:logc132(u,v)}, we deduce that
\begin{align*}
- \ln c_{13|2}\left\{F_{1|2}(tx_1|tx_2),F_{3|2}(tx_3|tx_2)\right\}\sim& t\Big\{\left(1+s^{\{13|2\}}_1\mathbbm{1}_{\{b_{1|2}\geq b_{3|2}\}}+s^{\{13|2\}}_2\mathbbm{1}_{\{b_{1|2}< b_{3|2}\}}\right)\max(b_{1|2},b_{3|2})\\
	&\hspace{1cm}- \left(2+s^{\{13|2\}}_1\mathbbm{1}_{\{b_{1|2}\geq b_{3|2}\}}+s^{\{13|2\}}_2\mathbbm{1}_{\{b_{1|2}< b_{3|2}\}}\right)\min(b_{1|2},b_{3|2})\Big\}.
\end{align*}
Combining this with result~\eqref{eqn:IEVIEVEV_fiveterms}, we find that the required gauge function has the form
\begin{align*}
g(\bm{x}) &= x_2 + b_{1|2} + b_{3|2} + \left(1 + s_{\text{m}}^{\{13\mid 2\}}\right)\max(b_{1|2},b_{3|2}) - \left(2 + s_{\text{m}}^{\{13\mid 2\}}\right)\min(b_{1|2},b_{3|2})\\
&=x_2 + \left(2 + s_{\text{m}}^{\{13\mid 2\}}\right)\max(b_{1|2},b_{3|2}) - \left(1 + s_{\text{m}}^{\{13\mid 2\}}\right)\min(b_{1|2},b_{3|2})\\
&=\left(2 + s_{\text{m}}^{\{13\mid 2\}}\right)\max(b_{1|2}-x_2,b_{3|2}-x_2) - \left(1 + s_{\text{m}}^{\{13\mid 2\}}\right)\min(b_{1|2}-x_2,b_{3|2}-x_2),
\end{align*}
i.e.,
\begin{align*}
	g(\bm{x}) = &\left(2+s_{\text{m}}^{\{13\mid 2\}}\right)\max\left\{V^{\{12\}}\left(x_1^{-1},x_2^{-1}\right),V^{\{23\}}\left(x_2^{-1},x_3^{-1}\right)\right\} - \left(1+s_{\text{m}}^{\{13\mid 2\}}\right)\min\left\{V^{\{12\}}\left(x_1^{-1},x_2^{-1}\right),V^{\{23\}}\left(x_2^{-1},x_3^{-1}\right)\right\},
%\label{eqn:IEVIEVEV}
\end{align*}
with $\min(x_1,x_2,x_3)\geq 0$, and 
\[
s^{\{13\mid 2\}}_{\text{m}} = s^{\{13\mid 2\}}_1\mathbbm{1}_{\left\{V^{\{12\}}\left(x_1^{-1},x_2^{-1}\right) \geq V^{\{23\}}\left(x_2^{-1},x_3^{-1}\right)\right\}} + s^{\{13\mid 2\}}_2\mathbbm{1}_{\left\{V^{\{12\}}\left(x_1^{-1},x_2^{-1}\right) < V^{\{23\}}\left(x_2^{-1},x_3^{-1}\right)\right\}}.
\]

\subsection{Extreme value and inverted extreme value copulas in $T_1$; inverted extreme value copula in $T_2$}
To calculate the gauge function for this model, we use results \eqref{eqn:exponentialMarginsAsymptotic},~\eqref{eqn:IEVcopula(23)}~and~\eqref{eqn:EVcopula(12)} to give the asymptotic form of the first five terms of equation~\eqref{eqn:negLogTriVine}. For the final term, equations~\eqref{eqn:Fconditional(23)}~and~\eqref{eqn:FconditionalEV(12)} give the required form of the conditional distributions, and we apply the inverted extreme value results from cases~3,~6~and~9 of Section~\ref{app:logc132(u,v)} to yield the gauge function
\begin{align*}
g(\bm{x})=
	\begin{cases}
       \left(2+s_1^{\{13\mid 2\}}\right)\left(1+s_2^{\{12\}}\right)\left(x_2-x_1\right) + V^{\{23\}}\left(x_2^{-1},x_3^{-1}\right), & 0\leq x_1\leq x_2,\\
        x_2 + V^{\{13\mid 2\}} \left[\left\{\left(x_1-x_2\right)\left(2+s_1^{\{12\}}\right)\right\}^{-1}, \left\{V^{\{23\}}\left(x_2^{-1},x_3^{-1}\right)-x_2\right\}^{-1}\right]
 ,& 0\leq x_2<x_1 .
	\end{cases}
	%\label{eqn:EVIEVIEV}
\end{align*}

\subsection{Extreme value and inverted extreme value copulas in $T_1$; extreme value copula in $T_2$}
Since the pair copulas in tree $T_1$ of this model are the same as in the previous example, the first five terms of equation~\eqref{eqn:negLogTriVine} will also be the same. To study the final term, we apply the extreme value results from cases~3,~6~and~9 of Section~\ref{app:logc132(u,v)}, and find that the gauge function is
\begin{align*}
g(\bm{x})=
	\begin{cases}
       x_2 + \left(1+s_2^{\{12\}}\right)\left(x_2-x_1\right) + \left(2+s_2^{\{13\mid 2\}}\right)\left\{ V^{\{23\}}\left(x_2^{-1},x_3^{-1}\right) -x_2\right\} , & 0\leq x_1\leq x_2,\\
       x_2 + \left( 2+s_{\text{m}}^{\{13\mid 2\}} \right)\max\left\{ \left(2+s_1^{\{12\}}\right)\left(x_1-x_2\right), V^{\{23\}}\left(x_2^{-1},x_3^{-1}\right) -x_2 \right\}\\
       ~~~-\left( 1+s_{\text{m}}^{\{13\mid 2\}} \right)\min\left\{ \left(2+s_1^{\{12\}}\right)\left(x_1-x_2\right), V^{\{23\}}\left(x_2^{-1},x_3^{-1}\right) -x_2 \right\}
 ,&  0\leq x_2<x_1 ,
	\end{cases}
	%\label{eqn:EVIEVEV}
\end{align*}
with 
\begin{align*}
s^{\{13\mid 2\}}_{\text{m}} &= s^{\{13\mid 2\}}_1\mathbbm{1}_{\left\{\left(2+s_1^{\{12\}}\right)\left(x_1-x_2\right) \geq V^{\{23\}}\left(x_2^{-1},x_3^{-1}\right) -x_2\right\}}+ s^{\{13\mid 2\}}_2\mathbbm{1}_{\left\{\left(2+s_1^{\{12\}}\right)\left(x_1-x_2\right) < V^{\{23\}}\left(x_2^{-1},x_3^{-1}\right) -x_2\right\}}.
\end{align*}

\subsection{Extreme value copulas in $T_1$; inverted extreme value copula in $T_2$}
For this model, the first five terms of equation~\eqref{eqn:negLogTriVine} are given by results~\eqref{eqn:exponentialMarginsAsymptotic}, \eqref{eqn:EVcopula(12)}~and~\eqref{eqn:EVcopula(23)}. Since both of the extreme value copulas in tree $T_1$ can have three different asymptotic behaviours, we require all nine inverted extreme value cases from Section~\ref{app:logc132(u,v)}, and the conditional distributions~\eqref{eqn:FconditionalEV(12)}~and~\eqref{eqn:FconditionalEV(23)} to study the final term of~\eqref{eqn:negLogTriVine}. The gauge function in this case is
\begin{align*}
g(\bm{x})=
	\begin{cases}
       x_2 + \left( 2+s_{\text{m}}^{\{13\mid 2\}} \right)\max\left\{ \left(1+s_2^{\{12\}}\right)\left(x_2-x_1\right), \left(1+s_1^{\{23\}}\right)\left(x_2-x_3\right) \right\}\\
       ~~~-\left( 1+s_{\text{m}}^{\{13\mid 2\}} \right)\min\left\{ \left(1+s_2^{\{12\}}\right)\left(x_2-x_1\right), \left(1+s_1^{\{23\}}\right)\left(x_2-x_3\right) \right\}, ~~~& \max(x_1,x_3)< x_2,\\
       x_2 + \left(2 + s_1^{\{13\mid 2\}}\right)\left(1+s_2^{\{12\}}\right)\left(x_2-x_1\right) + \left(2+s_2^{\{23\}}\right)\left(x_3-x_2\right), ~~~& x_1 < x_2\leq x_3 ,\\
       x_2 + \left(2 + s_2^{\{13\mid 2\}}\right)\left(1+s_1^{\{23\}}\right)\left(x_2-x_3\right) + \left(2+s_1^{\{12\}}\right)\left(x_1-x_2\right), ~~~& x_3< x_2 \leq x_1,\\
       x_2 + V^{\{13\mid 2\}}\left[\left\{\left(2+s_1^{\{12\}}\right)\left(x_1-x_2\right)\right\}^{-1} , \left\{\left(2+s_2^{\{23\}}\right)\left(x_3-x_2\right)\right\}^{-1}\right], ~~~& x_2 \leq\min(x_1,x_3),
	\end{cases}
	%\label{eqn:EVEVIEV}
\end{align*}
with $\min(x_1,x_2,x_3)\geq 0$ and
\begin{align*}
s^{\{13\mid 2\}}_{\text{m}} &= s^{\{13\mid 2\}}_1\mathbbm{1}_{\left\{\left(1+s_2^{\{12\}}\right)\left(x_2-x_1\right) \geq \left(1+s_1^{\{23\}}\right)\left(x_2-x_3\right) \right\}}+ s^{\{13\mid 2\}}_2\mathbbm{1}_{\left\{\left(1+s_2^{\{12\}}\right)\left(x_2-x_1\right) < \left(1+s_1^{\{23\}}\right)\left(x_2-x_3\right) \right\}}.
\end{align*}

\subsection{Extreme value copulas in $T_1$; extreme value copula in $T_2$}
Finally, for a vine copula where all three components are extreme value copulas, the first five terms of~\eqref{eqn:negLogTriVine} are the same as in the previous section, and we require all nine extreme value cases from Section~\ref{app:logc132(u,v)} to obtain the gauge function
\begin{align*}
g(\bm{x})=
	\begin{cases}
       x_2 + V^{\{13\mid 2\}}\left[\left\{\left(1+s_2^{\{12\}}\right)\left(x_2-x_1\right)\right\}^{-1} , \left\{\left(1+s_1^{\{23\}}\right)\left(x_2-x_3\right)\right\}^{-1}\right], ~~~& \max(x_1,x_3)\leq x_2,\\
       x_2 + \left( 2+s_2^{\{13\mid 2\}} \right)\left(2+s_2^{\{23\}}\right)\left(x_3-x_2\right) +  \left(1+s_2^{\{12\}}\right)\left(x_2-x_1 \right), ~~~& x_1 \leq x_2< x_3,\\
       x_2 + \left( 2+s_1^{\{13\mid 2\}} \right)\left(2+s_1^{\{12\}}\right)\left(x_1-x_2\right) +  \left(1+s_1^{\{23\}}\right)\left(x_2-x_3 \right), ~~~& x_3 \leq x_2< x_1 ,\\
       x_2 + \left( 2+s_{\text{m}}^{\{13\mid 2\}} \right)\max\left\{ \left(2+s_1^{\{12\}}\right)\left(x_1-x_2\right), \left(2+s_2^{\{23\}}\right)\left(x_3-x_2\right) \right\}\\
       ~~~-\left( 1+s_{\text{m}}^{\{13\mid 2\}} \right)\min\left\{ \left(2+s_1^{\{12\}}\right)\left(x_1-x_2\right), \left(2+s_2^{\{23\}}\right)\left(x_3-x_2\right) \right\}, ~~~& x_2 <\min(x_1,x_3),
	\end{cases}
	%\label{eqn:EVEVEV}
\end{align*}
with $\min(x_1,x_2,x_3)\geq 0$ and
\begin{align*}
s^{\{13\mid 2\}}_{\text{m}} &= s^{\{13\mid 2\}}_1\mathbbm{1}_{\left\{\left(2+s_1^{\{12\}}\right)\left(x_1-x_2\right) \geq \left(2+s_2^{\{23\}}\right)\left(x_3-x_2\right) \right\}}+ s^{\{13\mid 2\}}_2\mathbbm{1}_{\left\{\left(2+s_1^{\{12\}}\right)\left(x_1-x_2\right) < \left(2+s_2^{\{23\}}\right)\left(x_3-x_2\right) \right\}}.
\end{align*}

\end{document}